\documentclass{article}
\usepackage{amssymb,amsmath}
\usepackage{graphicx,enumerate}
\usepackage[utf8]{inputenc}
\title{Reticulation functor and the transfer properties}

\author{George Georgescu \\ \footnotesize University of Bucharest\\ \footnotesize Faculty of Mathematics and Computer Science\\ \footnotesize Bucharest, Romania\\ \footnotesize Email: georgescu.capreni@yahoo.com}

\date{}

\begin{document}
\maketitle

\begin{abstract}
It is known that by using the commutator operation, for each congruence modular algebra $A$ one can define a notion of prime congruence. The set $Spec(A)$ of prime congruences of $A$ is endowed with a Zariski style topology. The reticulation of the algebra $A$ is a bounded distributive lattice $L(A)$ whose prime spectrum $Spec(L(A))$ (with the Stone topology) is homemorphic to $Spec(A)$.

In a recent paper, C. Mure\c{s}an and the author have  proven the existence of reticulation for a semidegenerate congruence modular algebra $A$.

 The present paper aims to give an answer to two types of problems:

 $(I)$ how some properties of the algebra $A$ can be transferred to the lattice $L(A)$ and viceversa, how some  properties of $L(A)$ can be transferred to $A$;

 $(II)$ how the transfer properties from $(I)$ can be used to prove some old and new characterizations of some remarkable classes of algebras.

 We study the transfer properties related to Boolean centers, annihilators, patch and flat topologies of spectra, Pierce spectrum, pure and $w$ - pure congruences, the operators $Ker(\cdot)$ and $O(\cdot)$,etc.

By using these transfer properties, we obtain characterization theorems for several types of algebras : hyperarchimedean algebras, congruence normal and congruence $B$ - normal algebras, $mp$ - algebras, $PF$ - algebras, congruence purified algebras and $PP$ - algebras.

\end{abstract}

\newtheorem{definitie}{Definition}[section]
\newtheorem{propozitie}[definitie]{Proposition}
\newtheorem{remarca}[definitie]{Remark}
\newtheorem{exemplu}[definitie]{Example}
\newtheorem{intrebare}[definitie]{Open question}
\newtheorem{lema}[definitie]{Lemma}
\newtheorem{teorema}[definitie]{Theorem}
\newtheorem{corolar}[definitie]{Corollary}

\newenvironment{proof}{\noindent\textbf{Proof.}}{\hfill\rule{2mm}{2mm}\vspace*{5mm}}

\section*{Contents}

\begin{enumerate}
\item Introduction
\item Preliminaries
\item Reticulation of a universal algebra
\item Pure congruences
\item Transferring the operators $Ker(\cdot)$ and $O(\cdot)$ by reticulation
\item Flat and patch topologies on the spectra
\item Hyperarchimedean algebras
\item Congruence normal algebras
\item $Mp$ - algebras
\item Congruence purified algebras
\item $PP$ - algebras
 \end{enumerate}

\section{Introduction}

 \hspace{0.5cm} The reticulation of a (unitary) commutative ring $R$ is a bounded distributive lattice $L(A)$ whose prime spectrum $Spec(L(A))$ (endowed with the Stone topology) is homeomorphic to the prime spectrum $Spec(R)$ of $R$ (endowed with the Zariski topology). The notion of reticulation is due to Joyal in \cite{Joyal} and extensively studied by Simmons in \cite{Simmons}, who named $L(R)$ "the reticulation of $R$". In fact, the assignment $R\mapsto L(R)$ is a covariant functor (named "reticulation functor") from the category of commutative rings to the category of bounded distributive lattices. Paper \cite{Simmons} emphasises the way in which the reticulation functor transfers some properties from rings to lattices and viceversa and Chapter $V$ of the book \cite{Johnstone} uses the reticulation for many sheaf constructions in ring theory.

 Inspired by these ideas, many authors have proposed the similar notions of "reticulation" for various structures: $F$ - rings \cite{Johnstone}, $MV$ - algebras \cite{Belluce}, $BL$ - algebras \cite{g}, residuated lattices \cite{Muresan},etc. Then the problem to define a notion of reticulation in a context of universal algebra appeared in a natural manner. In order to define a reticulation for a universal algebra $A$ we need to have a notion of "prime spectrum". In 1993 Agliano defined a prime spectrum of a congruence
 - modular algebra by using the commutator operation (cf. \cite{Agliano}).

 The commutator theory, developed by Fresee and McKenzie in the monograph \cite{Fresee} for congruence - modular algebras, is a wonderful tool in proving some deep results of universal algebra (see \cite{Hobby}, \cite{McKenzie}). The commutator operation $[\cdot,\cdot]$, defined on the lattice $Con(A)$ of congruences of a congruence - modular algebra $A$, is a commutative multiplication, but not necessarily associative. It allows us to define a notion of prime congruence of $A$ and the set $Spec(A)$ of prime congruences of $A$ (= the prime spectrum of $A$) has a canonical topological structure \cite{Agliano} that generalizes the Zariski topology \cite{Atiyah} and the Stone topology \cite{BalbesDwinger}. Thus this topology on $Spec(A)$ will be called the Zariski topology and  the corresponding topological space will be denoted by $Spec_Z(A)$. In \cite{Agliano} the properties of topological space $Spec(A)$ for various types of algebras are analyzed. One can remark from \cite{Agliano} that the semidegenerate congruence - modular algebras \cite{Kollar} provide rich topological properties of the prime spectra.

 The reticulation of a congruence - modular algebra $A$ is a bounded distributive lattice $L(A)$ having the property that the prime spectrum $Spec(A)$ of $A$ is homeomorphic with the prime spectrum $Spec(L(A))$ of $L(A)$. The paper \cite{GM2} introduced  the reticulation for  semidegenerate congruence - modular algebras and \cite{GKM} studied the functorial properties of this reticulation. The construction given in \cite{GM2} extends all the reticulations existing in literature.

 Let us fix a variety $\mathcal{V}$ of semidegenerate congruence - modular algebras and an algebra $A$ in $\mathcal{V}$ such that the set $K(A)$ of compact congruence of $A$ is closed under the commutator operation. This paper concerns two types of problems:

 $(I)$ how some properties of the algebra $A$ can be transferred to the lattice $L(A)$ and viceversa, how some  properties of $L(A)$ can be transferred to $A$;

 $(II)$ how the transfer properties from $(I)$ can be used to prove some old and new characterizations of some remarkable classes of algebras.

 A celebrated theorem of Hochster \cite{Hochster} says that for any bounded distributive lattice $L$ there exists a commutative ring $R$ such that the reticulation $L(R)$ of $R$ is isomorphic to $L$. Thus for the fixed algebra $A$ there exists a commutative ring $R$ such that the reticulations $L(A),L(R)$ are isomorphic and the prime spectra $Spec(A), Spec(R)$ are homeomorphic. This result assures a transfer of properties from rings to algebras and viceversa, obtaining a new way for exporting results from rings to algebras.

Now we shall describe the content of this paper.

Section 2 contains some basic properties on the commutator operation and on the semidegenerate congruence - modular varieties (cf. \cite{Fresee}, \cite{Kollar}, \cite{Agliano}). Some facts on the residuated structure of $Con(A)$, the set of congruences of a semidegenerate congruence - modular algebra $A$, are reminded. In general, the commutator operation is not associative, so $Con(A)$ is not a quantale. But we observe that $Con(A)$ is a commutative $mi$ - structure in the sense of \cite{GeorgescuVoiculescu2}, so all results from the mentioned paper can be applied to our setting.

The results of this paper are proved for algebras in a fixed semidegenerate congruence - modular variety $\mathcal{V}$.

In Section 3 we remind the principal properties of the reticulation $L(A)$ of a semidegenerate congruence - modular algebra $A$ such that the set $K(A)$ of the compact congruences of $A$ is closed under the commutator. We prove some new results regarding the way in which the reticulation preserves the annihilators.

Section 3 deals with pure congruences of an algebra $A$ of $V$. They are generalizations of pure ideals in commutative rings \cite{Borceux1},\cite{Borceux2} and $\sigma$ - ideals in bounded distributive lattices \cite{Cornish}, \cite{Cornish1}, \cite{GeorgescuVoiculescu}. We remark that the pure congruences of $A$ are exactly the virginal elements of the $mi$ - structure $Con(A)$ (see \cite{GeorgescuVoiculescu2}). Then the properties established in \cite{GeorgescuVoiculescu2} for virginal elements remain valid for the pure congruences. The set $VCon(A)$ of pure congruences of the algebra $A$ is a frame in the sense of \cite{Johnstone}.
The main result of section is Theorem 4.17, that establishes a frame isomorphism between $VCon(A)$ and the frame $VId(L(A))$ of the $\sigma$ - ideals of the reticulation $L(A)$.

A lot of properties of pure congruences are expressed in terms of the operators $Ker(\cdot): Con(A)\rightarrow Con(A)$ and $O(\cdot): Spec(A)\rightarrow Con(A)$, defined in the previous section. In Section 5 we study the behaviour of the reticulation w.r.t. these two operators.

By using the reticulation, it is easy to see that $Spec_Z(A)$ is a spectral space \cite{Hochster}, \cite{Dickmann}. By a standard way described in \cite{Hochster}, \cite{Dickmann}, \cite{Johnstone} one can define on $Spec(A)$ two other important topologies: the patch topology and the flat topology (the corresponding topological spaces are denoted by $Spec_P(A)$ and $Spec_F(A)$, respectively). Section 6 studies the spaces $Spec_P(A)$ and $Spec_F(A)$. Some results on the patch and the flat topologies are extended from rings \cite{Doobs}, \cite{Tar1} to algebras. We define the regular and max - regular congruences of $A$ and prove that the reticulation preserves these notions.

The hyperarchimedean algebras, introduced in \cite{GM2}, are the objects of $\mathcal{V}$ characterized by a Nachbin style property: $Spec(A)$ coincides with the maximal spectrum $Max(A)$ of $A$ (Max(A) is the set of maximal congruences of $A$). The main result of Section 7 is a characterization theorem of hyperarchimedean algebras (Theorem 7.7). It can be viewed as a generalization of a theorem proved by Aghajani and Tarizadeh for rings (see Theorem 3.3 from \cite{Aghajani}); in fact, our Theorem 7.7 follows from the Aghajani and Tarizadeh result by using some transfer properties of reticulation.

In Section 8 we study two subclasses of the variety $\mathcal{V}$:

$\bullet$ the congruence normal algebras, introduced in \cite{GKM} as structures that generalize the properties of Gelfand rings \cite{Johnstone}, \cite{Aghajani}, $F$ -rings \cite{Johnstone}, normal lattices \cite{Cornish}, \cite{Johnstone}, \cite{GeorgescuVoiculescu}, etc.;

$\bullet$ the congruence $B$ - normal algebras, defined in \cite{GKM} as structures that generalize the properties of clean rings \cite{a}, $B$ - normal lattices \cite{Cignoli}, etc.

We obtain many characterizations of congruence normal algebras (see Propositions 8.3 and 8.8) and of congruence $B$ - normal algebras ( see Theorem 8.19), extending a lot of results obtained for some concrete structures : commutative rings \cite{Aghajani}, \cite {Banaschewski}, \cite{Borceux1}, \cite {Borceux2}, \cite {Simmons}, commutative l - groups \cite{f}, bounded distributive lattices \cite{Cornish}, \cite {Cignoli}, \cite{GeorgescuVoiculescu}, \cite{Pawar}, \cite{R1}, $MV$ - algebras \cite{FG}, $BL$ - algebras \cite{g}, residuated lattices  \cite{Cheptea}, \cite{R2}, etc.

Two new classes of universal algebras are introduced in Section 9 : $mp$- algebras and $PF$ - algebras. The $mp$ - algebras generalize the $mp$ -rings \cite{Aghajani}, \cite{Tar1+} and the conormal lattices \cite{Cornish} and $PF$ - algebras are exactly semiprime $mp$ - algebras. The $PF$ - algebras extend $PF$ - rings (see \cite{Al-Ezeh1+}).  By using the reticulation, we obtain some characterization theorems for $mp$- algebras and $PF$ - algebras.

The purified rings are introduced in Section 8 of \cite{Aghajani} and a characterization theorem of these objects
was proved in the mentioned paper. In Section 10 of our paper we define the notion of congruence purified algebras as a generalization of purified rings. We prove a characterization theorem for congruence purified algebras (Theorem 10.15). Applying this general result in the particular case of commutative rings we obtain the main part of Theorem 8.5 from \cite{Aghajani}.

Section 11 deals with $PP$ - algebras, a notion that generalizes the $PP$ - rings ( = Baer rings) (see \cite{Tar1+}). We prove that the reticulation functor transforms a $PP$ - algebra into a Stone lattice and an admissible $PP$ - morphism of semiprime algebras into a Stone morphism of bounded distributive lattices. Two characterization theorems for $PP$ - algebras are obtained.

\section{Preliminaries}

 \hspace{0.5cm} Let $\tau$ be a finite signature of universal algebras. Throughout this paper we shall assume that the algebras have the signature $\tau$. Let $A$ be an algebra and $Con(A)$ the complete lattice of its congruences; $\Delta_A$ and $\nabla_A$ shall be the first and the last elements of $Con(A)$. If $X\subseteq A^2$ then $Cg_A(X)$ will be the congruence of $A$ generated by $X$; if $X = \{(a,b)\}$ with $a,b\in A$ then $Cg_A(a,b)$ will denote the (principal) congruence generated by $\{(a,b)\}$. We shall denote by $PCon(A)$ the set of principal congruences of $A$. $Con(A)$ is an algebraic lattice: the finitely generated congruences of $A$ are its compact elements. $K(A)$ will denote the set of compact congruences of $A$. We observe that $K(A)$ is closed under finite joins of $Con(A)$ and $\Delta_A\in K(A)$.

 For any $\theta\in Con(A)$, $A/\theta$ is the quotient algebra of $A$ w.r.t. $\theta$; if $a\in A$ then $a/\theta$ is the congruence class of $a$ modulo $\theta$. We shall denote by $p_\theta:A\rightarrow A/\theta$ the canonical surjective $\tau$ - morphism $p_\theta(a) = a/\theta$, for all $a\in A$.

 Let  $\mathcal{V}$  be a congruence - modular variety of $\tau$ - algebras. Following \cite{Fresee}, p.31, the commutator is the greatest operation $[\cdot,\cdot]_A$ on the congruence lattices $Con(A)$ of members $A$ of $\mathcal{V}$ such that for any surjective morphism $f:A\rightarrow B$ of $\mathcal{V}$ and for any $\alpha,\beta\in Con(A)$, the following conditions hold

 (2.1) $[\alpha,\beta]_A\subseteq \alpha\bigcap \beta$;

 (2.2) $[\alpha,\beta]_A\lor Ker(f)$ = $f^{-1}([f(\alpha\lor Ker(f)),f(\beta\lor Ker(f))]_B)$.

 If $\alpha, \beta, \theta\in Con(A)$ then, by (2.2) we get

 (2.3) $([\alpha,\beta]_A\lor \theta)/\theta$ = $[(\alpha\lor \theta)/\theta,(\beta\lor \theta)/\theta]_{A/\theta}$.

 The commutator operation is commutative, increasing in each argument and distributive with respect to arbitrary joins. If there is no danger of confusion then we write $[\alpha,\beta]$ instead of $[\alpha,\beta]_A$.

 \begin{propozitie}\cite{Fresee}
 For any congruence - modular variety $\mathcal{V}$ the following are equivalent:
\newcounter{nr}
\begin{list}{(\arabic{nr})}{\usecounter{nr}}
\item  $\mathcal{V}$ has Horn - Fraser property: if $A,B$ are members of $\mathcal{V}$ then the lattices $Con(A\times B)$ and $Con(A)\times Con(B)$ are isomorphic;
\item $[\nabla_A,\nabla_A] = \nabla_A$, for all $A\in \mathcal{V}$;
\item $[\theta,\nabla_A] = \theta$, for all $A\in \mathcal{V}$ and $\theta\in Con(A)$.
\end{list}
\end{propozitie}

Following \cite{Kollar}, a variety $\mathcal{V}$ is semidegenerate if no nontrivial algebra in $\mathcal{V}$ has one - element subalgebras. By \cite{Kollar}, a variety $\mathcal{V}$ is semidegenerate if and only if for any algebra $A$ in $\mathcal{V}$, the congruence $\nabla_A$ is compact.

\begin{propozitie}\cite{Agliano}
If $\mathcal{V}$ is a semidegenerate congruence - modular variety then for each algebra $A$ in $\mathcal{V}$ we have $[\nabla_A,\nabla_A] = \nabla_A$.
\end{propozitie}

\begin{remarca}
Let $A$ be a semidegenerate congruence - modular algebra. One can define on the lattice $Con(A)$ a residuation operation ( = implication) $\alpha \rightarrow \beta = \bigvee \{ \gamma|[\alpha,\gamma] \subseteq \beta\}$ and an annihilator operation ( = negation) $\alpha^{\bot } = \alpha^{\bot_A } =  \alpha \rightarrow \Delta_A =\bigvee \{\gamma |[\alpha,\gamma] = \Delta_A\}$. The implication $\rightarrow$ fulfills the usual residuation property: for all $\alpha,\beta,\gamma\in Con(A)$, $\alpha\subseteq \beta\rightarrow\gamma$ if and only if $[\alpha,\beta]\subseteq\gamma$. We remark that $(Con(A), \lor, \land, [\cdot,\cdot], \rightarrow, \Delta_A, \nabla_A)$ is commutative and integral complete $l$ - groupoid (see \cite{Birkhoff}). In fact, $(Con(A),\lor,\bigcap,[\cdot,\cdot]_A,\Delta_A,\nabla_A)$ is a multiplicative - ideal structure (= mi - structure) in the sense of \cite{GeorgescuVoiculescu2}. Thus all the results contained in \cite{GeorgescuVoiculescu2} hold for the particular $mi$ - structure $Con(A)$.

\end{remarca}

For the rest of the section we fix a semidegenerate congruence - modular algebra $A$ in a semidegenerate congruence - modular variety $\mathcal{V}$.

\begin{lema}\cite{GM2}
 For all congruences $\alpha, \beta, \gamma, \alpha_1,...,\alpha_k $ the following hold:
\usecounter{nr}
\begin{list}{(\arabic{nr})}{\usecounter{nr}}
\item $\alpha\lor \beta = \nabla_A$ implies $[\alpha,\beta] = \alpha\bigcap \beta$;
\item $\alpha\lor \beta = \alpha\lor \gamma = \nabla_A$ implies $\alpha\lor [\beta,\gamma] = \alpha\lor (\beta\bigcap \gamma) = \nabla_A$;
\item $\alpha\lor \beta = \nabla_A$ implies $[\alpha,\alpha]^n\lor [\beta,\beta]^n = \nabla_A$, for all integer $n > 0$;
\item For any integer $n\geq 1$, $[\alpha\lor\beta,\alpha\lor\beta]^{n^2}\subseteq [\alpha,\alpha]^n\lor [\beta,\beta]^n$;
\item For all integers $n, k\geq 1$,$[\alpha_1\lor\cdots\alpha_k]^{n^k}\subseteq [\alpha_1,\alpha_1]^n\lor\cdots \lor [\alpha_k,\alpha_k]^n$.
\end{list}
\end{lema}

For all congruences $\alpha, \beta \in Con(A)$ and for any integer $n\geq 1$ we define by induction the congruence $[\alpha,\beta]^n$: $[\alpha,\beta]^1$ = $[\alpha,\beta]$ and $[\alpha,\beta]^{n+1} =[[\alpha,\beta]^n,[\alpha,\beta]^n]$. By convention, we set $[\alpha,\alpha]^0 = \alpha$.

\begin{lema}
If $\alpha, \beta, \theta\in Con(A)$ and $\theta\subseteq \alpha\bigcap \beta$ then we have $[\alpha/\theta,\beta/\theta]^n$ = $([\alpha,\beta]^n\lor \theta)/\theta$. For all $a,b\in A$ and $\theta\in Con(A)$ and for any integer $n\geq 0$ we have $[Cg_{A/\theta}(a/\theta,b/\theta),Cg_{A/\theta}(a/\theta,b/\theta)]^n_{A/\theta}$ = $[Cg_A(a,b),Cg_A(a,b)]_A^n\lor \theta)/\theta$.
\end{lema}

\begin{lema} (see \cite{GM1}, Lemma 6.1)
For all integers $n\geq 1$ and all $\alpha,\beta\in Con(A)$: $[\alpha,\beta]_A^{n+1}= [[\alpha,\beta]_A,[\alpha,\beta]_A ]^n_A$.

\end{lema}

Following \cite{McKenzie}, p.82 or \cite{Agliano}, p.582, a congruence $\phi\in Con(A)- \{\nabla_A \}$  is ${\emph{prime}}$ if for all $\alpha, \beta \in Con(A)$, $[\alpha,\beta] \subseteq \phi$ implies $\alpha \subseteq \phi$ or $\beta \subseteq \phi$. Let us introduce the following notations: $Spec(A)$ is the set of prime congruences and $Max(A)$ is the set of maximal elements of $Con(A)$. If $\theta \in Con(A)- \{\nabla_A \}$  then there exists $\phi \in Max(A)$ such that $\theta \subseteq \phi$. By \cite{Agliano}, the following inclusion $Max(A) \subseteq Spec(A)$ holds.

According to \cite{Agliano}, p.582, the {\emph{radical}} $\rho(\theta)=\rho_A(\theta)$ of a congruence $\theta \in A$ is defined by $\rho_A(\theta)=\bigwedge \{\phi\in Spec(A)|\theta \subseteq \phi\}$; if $\theta=\rho(\theta)$ then $\theta$ is a radical congruence. We shall denote by $RCon(A)$ the set of radical congruences of $A$. The algebra $A$ is {\emph{semiprime}} if $\rho(\Delta_A)=\Delta_A$.

\begin{lema}
\cite{Agliano},\cite{GM2}
For all congruences $\alpha, \beta \in Con(A)$ the following hold:
\usecounter{nr}
\begin{list}{(\arabic{nr})}{\usecounter{nr}}
\item $\alpha \subseteq \rho(\alpha)$;
\item $\rho(\alpha \land \beta)=\rho ([\alpha,\beta])=\rho(\alpha) \land \rho(\beta)$;
\item $\rho(\alpha)= \nabla_A$ iff $\alpha = \nabla_A$;
\item $\rho(\alpha \lor \beta)=\rho(\rho(\alpha) \lor \rho(\beta))$;
\item $\rho(\rho(\alpha))=\rho(\alpha)$;
\item $\rho(\alpha) \lor \rho(\beta) = \nabla_A$ iff $\alpha \lor \beta = \nabla_A$;
\item $\rho([\alpha,\alpha]^n)=\rho(\alpha)$, for all integer $n \geq 0$.
\end{list}
\end{lema}

For an arbitrary family $(\alpha_i)_{i\in I} \subseteq A$, the following equality holds: $\rho(\displaystyle \bigvee_{i \in I}\alpha_i)=\rho(\bigvee_{i \in I} \rho(\alpha_i))$. If $(\alpha_i)_{i\in I} \subseteq RCon(A)$ then we denote $\displaystyle \bigvee_{i \in I}^{\cdot} \alpha_i=\rho(\bigvee_{i \in I}\alpha_i)$. Thus it easy to prove that $(RCon(A), \displaystyle \bigvee^{\cdot}, \wedge, \rho(\Delta_A), \nabla_A)$ is a frame (see \cite{Johnstone} as a basic text for the frame theory).

\begin{propozitie}\cite{Agliano}
Assume that $K(A)$ is closed under the commutator operation $[\cdot,\cdot]$. For any congruence $\theta$ of $A$ the following equality holds:

$\rho(\theta) = \bigvee\{\alpha\in K(A)|[\alpha,\alpha]^n\subseteq \theta$, for some  $n\geq 0\}$.
\end{propozitie}

In particular, we have $\rho(\Delta_A) = \bigvee\{\alpha\in K(A)|[\alpha,\alpha]^n = \Delta_A$, for some  $n\geq 0\}$.

Then the algebra $A$ is semiprime if and only if for any $\alpha\in K(A)$ and for any integer $n\geq 0$, $[\alpha,\alpha]^n = \Delta_A$ implies $\alpha = \Delta_A$.

Let us consider the following property:

$(\star)$ For all $\alpha,\beta \in K(A)$ and for any integer $n\geq1$ the exists an integer $m\geq 0$ such that $[[\alpha,\alpha]^m,[\beta,\beta]^m]]\subseteq [\alpha,\beta]^n$.

We remark that $[[\alpha,\alpha]^0, [\alpha,\alpha]^0]=[\alpha,\alpha]\subseteq \alpha=[\alpha,\alpha]^0$ so $(\star)$ is equivalent with the following condition:

$(\star \star)$ For all $\alpha,\beta \in K(A)$ and for any integer $n\geq 0$ there exists an integer $m\geq 0$ such that $[[\alpha,\alpha]^m,[\beta,\beta]^m]]\subseteq [\alpha,\beta]^n$.

If the commutator operation $[\cdot,\cdot]$ is associative then it is obvious that the algebra $A$ verifies the condition $(\star)$.

Ring theory is the main source of inspiration for most results of this paper. Since the commutator operation is not always associative it is difficult to extend some theorems from rings to algebras of $\mathcal{V}$. Many times we shall use the condition $(\star)$ as a sufficient condition in order to obtain generalizations of some results of ring theory.

\begin{propozitie}
Assume that $A\in \mathcal{V}$ and $\theta\in Con(A)$. If $A$ verifies $(\star)$ then the quotient algebra $A/\theta$ verifies $(\star)$.
\end{propozitie}

\begin{lema}
Let $A$ be an algebra in the variety $\mathcal{V}$ and $\alpha, \beta\in K(A)$. If $A$ is semiprime or verifies the condition $(\star)$ then $[\alpha,\beta]\subseteq \rho(\Delta_A)$ implies that there exists an integer $m\geq0$ such that $[[\alpha,\alpha]^m,[\beta,\beta]^m] = \Delta_A$.
\end{lema}

\begin{proof} If $A$ is semiprime then $[\alpha,\beta]\subseteq \rho(\Delta_A)$ implies that $[[\alpha,\alpha]^0, [\beta,\beta]^0]=[\alpha,\beta]\subseteq \rho(\Delta_A)=\Delta_A$.
Assume that $A$ verifies the condition $(\star)$ and $[\alpha,\beta]\subseteq \rho(\Delta_A)$. By Proposition 2.8, there exists an integer $n\geq 0$ such that $[\alpha,\beta],[\alpha,\beta]]^n = 0$. According to Lemma 2.6 we have $[\alpha,\beta]^{n+1} = \Delta_A$. Taking into account the condition $(\star)$ it follows that there exists an integer $m\geq 0$ such that $[[\alpha,\alpha]^m,[\beta,\beta]^m] = \Delta_A$.
\end{proof}

Let $B(Con(A))$ be the set of complemented elements in the bounded lattice $Con(A)$. By Lemma 4 from \cite{Jipsen} or \cite{Galatos}, $B(Con(A))$ is a Boolean algebra in which $\alpha^{\perp}$ is the complement of a congruence $\alpha\in B(Con(A))$. If $\alpha\in Con(A)$ then $\alpha\in B(Con(A)$ iff $\alpha\lor \alpha^{\perp} = \nabla_A$. For all $\theta,\vartheta\in Con(A)$ and $\alpha\in B(Con(A))$ we have $\theta\land \alpha =[\theta,\alpha]$, $\alpha\rightarrow \theta  = \alpha^{\perp}\lor \theta$ and $(\theta\land\vartheta)\lor \alpha = (\theta\lor \alpha)\land (\vartheta\land \alpha)$.

\begin{lema}
 For all congruences $\theta, \vartheta\in Con(A)$ the following hold:
\usecounter{nr}
\begin{list}{(\arabic{nr})}{\usecounter{nr}}
\item If $\theta\lor \vartheta = \nabla_A$ and $[\theta,\vartheta] = \Delta_A$ then $\theta,\vartheta\in B(Con(A))$;
\item For any integer $n\geq 1$, if $\theta\lor \vartheta = \nabla_A$ and $[[\theta,\theta]^n,[\vartheta,\vartheta]^n] = \Delta_A$ then $[\theta,\theta]^n,[\vartheta,\vartheta]^n\in B(Con(A))$;
\item $B(Con(A))\subseteq K(A)$.

\end{list}
\end{lema}

In what follows we shall identify the variety $\mathcal{V}$ with the category whose objects are algebras in $\mathcal{V}$ and the morphisms are the usual $\tau$ - homomorphisms (recall that $\tau$ is the signature of algebras in $\mathcal{V}$).

Let $u:A\rightarrow B$ be an arbitrary morphism in $\mathcal{V}$ and $u^*:Con(B)\rightarrow Con(A)$, $u^{\bullet}:Con(A)\rightarrow Con(B)$ are the maps defined by $u^*(\beta) = u^{-1}(\beta)$ and $u^{\bullet}(\alpha) = Cg_B(f(\alpha))$, for all $\alpha\in Con(A)$ and $\beta\in Con(B)$. Thus $u^{\bullet}$ is a the left adjoint of $u^*$: for all $\alpha\in Con(A)$, $\beta\in Con(B)$, we have $u^{\bullet}(\alpha)\subseteq \beta$ iff  $\alpha\subseteq u^*(\beta)$.

By \cite{GKM}, if $\alpha\in K(A)$ then $u^{\bullet}(\alpha)\in K(B)$, so we can consider the restriction $u^{\bullet}|_{K(A)}:K(A)\rightarrow K(B)$.

 According to \cite{GM1}, the morphism $u:A\rightarrow B$ of $\mathcal{V}$ is said to be admissible if $u^*(\psi)\in Spec(A)$, for any $\psi\in Spec(B)$. By Proposition 3.6 of \cite{GM1}, any surjective morphism of $\mathcal{V}$ is admissible.

The flat ring morphisms are usually defined in terms of the tensor product (see \cite{Atiyah}). According to Exercise 22 of \cite{Bourbaki}, p.65 or \cite{Picavet}, p.46, a morphism of commutative rings $u:R\rightarrow S$ is flat if and only if for any ideal $I$ of $R$ and for any finitely generated ideal $J$ of $R$, we have $(I:J)S = (IS:JS)$. This characterization of flat ring morphisms in terms of residuation allows us to define a notion of "flatness" for the morphisms of the given variety $\mathcal{V}$.

A morphism $u:A\rightarrow B$ in the variety $\mathcal{V}$ is said to be flat if for each $\theta\in Con(A)$ and for each $\alpha\in K(A)$, we have $u^{\bullet}(\alpha\rightarrow \theta)$ = $u^{\bullet}(\alpha)\rightarrow u^{\bullet}(\theta)$.

\begin{lema}
For any morphism $u:A\rightarrow B$ in the variety $\mathcal{V}$, the following are equivalent:
\usecounter{nr}
\begin{list}{(\arabic{nr})}{\usecounter{nr}}
\item $f$ is flat;
\item For each $\theta\in Con(A)$ and for each $\alpha\in K(A)$, we have $u^{\bullet}(\alpha)\rightarrow u^{\bullet}(\theta)\subseteq u^{\bullet}(\alpha\rightarrow \theta)$.
\end{list}
\end{lema}

For any $\theta \in Con(A)$ we denote $V_A(\theta) = V(\theta) = \{\phi \in Spec(A)|\theta\subseteq \phi\}$ and $D_A(\theta) = D(\theta) = Spec(A)- V(\theta)$. If $\alpha, \beta \in Con(A)$ then $D(\alpha)\bigcap D(\beta) = D([\alpha,\beta])$ and $V(\alpha)\bigcup V(\beta) = V([\alpha,\beta])$. For any family of congruences $(\theta_i)_{i\in I}$ we have $\bigcup_{i\in I}D(\theta_i) = D(\bigvee_{i\in I}\theta_i)$ and $\bigcap_{i\in I}V(\theta_i) = V(\bigvee_{i\in I}\theta_i)$. Thus $Spec(A)$ becomes a topological space whose open sets are $D(\theta),\theta\in Con(A)$. We remark that this topology extends the Zariski topology (defined on the prime spectra of commutative rings) and the Stone topology (defined on the prime spectra of bounded distributive lattices). Thus this topology on the prime spectrum $Spec(A)$ of the algebra $A$ will be named Zariski topology and the respective topological space will be denoted by $Spec_Z(A)$. We mention that the family $(D(\alpha))_{\alpha\in K(A)}$ is a basis of open sets for the Zariski topology. We shall denote by $Max_Z(A)$ the set $Max(A)$ considered as subspace of $Spec_Z(A)$.

\section{Reticulation of a universal algebra}

\hspace{0.5cm} Let us fix a semidegenerate congruence - modular variety $\mathcal{V}$ and $A$ an algebra of $\mathcal{V}$ such that the set $K(A)$ of finitely generated congruences of $A$ is closed under the commutator operation. Consider on $Con(A)$ the following equivalence relation: for all $\alpha, \beta\in Con(A)$, $\alpha\equiv \beta$ if and only if $\rho(\alpha) = \rho(\beta)$. Let $\hat \alpha$ be the equivalence class of $\alpha\in Con(A)$ and $0 = \hat{\Delta_A}, 1 = \hat{\nabla_A}$. Then $\equiv$ is a congruence of the lattice $Con(A)$ so the quotient set $L(A)$ = $K(A)/{\equiv}$ is a bounded distributive lattice, named the reticulation of the algebra $A$ (see \cite{GM2}). We shall denote by $\lambda_A:K(A)\rightarrow L(A)$ the function defined by $\lambda_A(\alpha) = \hat{\alpha}$, for all $\alpha\in K(A)$.

We remark that for all $\alpha,\beta\in K(A)$ we have $\lambda_A(\alpha) = \lambda_A(\beta)$ if and only if $\rho(\alpha) = \rho(\beta)$.

\begin{lema}\cite{GM2}
 For all congruences $\alpha, \beta \in K(A)$ the following hold:
\usecounter{nr}
\begin{list}{(\arabic{nr})}{\usecounter{nr}}
\item $\lambda_A(\alpha \lor \beta) = \lambda_A(\alpha)\lor \lambda_A(\beta)$;
\item $\lambda_A([\alpha,\beta])$ = $\lambda_A(\alpha)\land \lambda_A(\beta)$;
\item $\lambda_A(\alpha) = 1$ iff $\alpha = \nabla_A$;
\item $\lambda_A(\alpha) = 0$ iff $[\alpha,\alpha]^k = \Delta_A$, for some integer $k\geq 0$;
\item $\lambda_A([\alpha,\alpha]^k) = \lambda_A(\alpha)$, for all integer $k\geq 0$;
\item $\lambda_A(\alpha) = 0$ iff $\alpha\subseteq \rho(\Delta_A)$;
\item If $A$ is semiprime then $\lambda_A(\alpha) = 0$ iff $\alpha = \Delta_A$;
\item  $\lambda_A(\alpha)\leq \lambda_A(\beta)$ iff $\rho(\alpha)\subseteq\rho(\beta)$ iff $[\alpha,\alpha]^n \subseteq \beta$, for some integer $n\geq 0$.
\end{list}
\end{lema}

Let $L$ be a bounded distributive lattice and $Id(L)$ the set of its ideals. Then $Spec_{Id}(L)$ will denote the set of prime ideals in $L$ and $Max_{Id}(L)$ the set of maximal ideals in $L$. $Spec_{Id}(L)$ (resp. $Max_{Id}(L)$) endowed with Stone topology will be denoted by $Spec_{Id,Z}(L)$ (resp. $Max_{Id,Z}(L)$).

For any ideal $I$ of $L$ we denote $D_{Id}(I) = \{Q\in Spec_{Id}(L)|I\not\subseteq Q\}$ and $V_{Id}(I) = \{Q\in Spec_{Id}(L)|I\subseteq Q\}$. If $x\in L$ then we use the notation $D_{Id}(x) = D_{Id}((x])  = \{Q\in Spec_{Id}(L)|x\notin Q\}$ and $V_{Id}(x) = V_{Id}((x])  = \{Q\in Spec_{Id}(L)|x\in Q\}$, where $(x]$ is the principal ideal of $L$ generated by the set $\{x\}$. Recall from \cite{Johnstone} that the family $(D_{Id}(x))_{x\in L}$ is a basis of open sets for the Stone topology on $Spec_{Id}(L)$.

For all $\theta\in Con(A)$ and $I\in Id(L(A))$ we shall denote

$\theta^{\ast} = \{\lambda_A(\alpha)|\alpha\in K(A), \alpha\subseteq \theta \}$ and $I_{\ast} =\bigvee\{\alpha\in K(A)|\lambda_A(\alpha)\in I\}$.

Thus $\theta^{\ast}$ is an ideal of the lattice $L(A)$ and $I_{\ast}$ is a congruence of $A$. In this way one obtain two order - preserving functions $(\cdot)^{\ast}:Con(A)\rightarrow Id(L(A))$ and $(\cdot)_{\ast}:Id(L(A))\rightarrow Con(A)$.

The functions $(\cdot)^{\ast}$ and $(\cdot)_{\ast}$ will play an important role in proving the transfer properties of reticulation. The following four results constitute the first steps in obtaining transfer properties. They will be used many times in the proofs.

\begin{lema}
The following assertions hold
\usecounter{nr}
\begin{list}{(\arabic{nr})}{\usecounter{nr}}
\item If $\theta,\chi\in Con(A)$ then $[\theta,\chi]^{\ast}$ = $(\theta\land \chi)^{\ast}$ = $\theta^{\ast}\bigcap \chi^{\ast}$;
\item If $(\theta_i)_{i\in I}$ is a family of congruences of $A$ then $(\bigvee_{i\in I}\theta_i)^{\ast} = \bigvee_{i\in I}\theta_i^{\ast}$;
\item If $I$ and $J$ are ideals of the lattice $L(A)$ then $(I\lor J)_{\ast} = \rho(I_{\ast}\lor J_{\ast})$.
\end{list}
\end{lema}

\begin{proof}

(1) Since $(\cdot)^{\ast}:Con(A)\rightarrow Id(L(A))$ is an order - preserving function it is clear that $[\theta,\chi]^{\ast}\subseteq(\theta\land \chi)^{\ast}\subseteq \theta^{\ast}\bigcap \chi{\ast}$.

Let us consider an element $x\in \theta^{\ast}\bigcap \chi^{\ast}$, so there exist $\alpha,\beta\in K(A)$ such that $\alpha\subseteq\theta$, $\beta\subseteq\chi$ and $ x = \lambda_A(\alpha) = \lambda_A(\beta)$. Thus $x = x\land x = \lambda_A(\alpha)\land \lambda_A(\beta) = \lambda_A([\alpha,\beta])$ and $[\alpha,\beta]$ is a compact congruence such that $[\alpha,\beta]\subseteq [\theta,\chi]$, hence $x = \lambda_A([\alpha,\beta])\in [\theta,\chi]^{\ast}$.

(2) Let us consider an element $x\in (\bigvee_{i\in I}\theta_i)^{\ast}$ so there exists $\alpha\in K(A)$ such that $\alpha\subseteq \bigvee_{i\in I}\theta_i$ and $x = \lambda_A(\alpha)$. Since $\alpha$ is a compact congruence, there exists a finite subset $J$ of $I$ such that $\alpha\subseteq \bigvee_{i\in J}\theta_i$. Thus there exists a family $(\beta_i)_{i\in J}$ of compact congruences such that $\alpha\subseteq \bigvee_{i\in J}\beta_i$ and $\beta_i\subseteq\theta_i$, for all $i\in J$.

It follows that $\lambda_A(\beta_i)\in \theta_i^{\ast}$, for all $i\in J$, therefore $x = \lambda_A(\alpha)\leq \bigvee_{i\in J}\lambda_A(\beta_i)\in \bigvee_{i\in J}\theta_i^{\ast}\subseteq \bigvee_{i\in I}\theta_i^{\ast}$. Thus we have $(\bigvee_{i\in I}\alpha_i)^{\ast}\subseteq \bigvee_{i\in I}\theta_i^{\ast}$. The converse inclusion is obvious.

(3) By using Lemma 3.1 it is straightforward to prove the equality $(I\lor J)_{\ast} = \rho(I_{\ast}\lor J_{\ast})$.

\end{proof}

\begin{lema}
\cite{GM2} For all $\theta\in Con(A)$, $\alpha \in K(A)$ and $I\in Id(L(A))$ the following hold:
\usecounter{nr}
\begin{list}{(\arabic{nr})}{\usecounter{nr}}
\item $\alpha\subseteq I_{\ast}$ iff $\lambda_A(\alpha)\in I$;
\item $(\theta^{\ast})_{\ast} = \rho(\theta)$ and $(I_{\ast})^{\ast} = I$;
\item $\theta^{\ast} = (\rho(\theta))^{\ast}$ and $\rho(I_{\ast}) = I_{\ast}$;
\item If $\theta\in Spec(A)$ then $(\theta^{\ast})_{\ast} = \theta$ and $\theta^{\ast}\in Spec_{Id}(L(A))$;
\item If $I\in Spec_{Id}(L(A))$ then $I_{\ast}\in Spec(A)$;
\item If $\theta\in Spec(A)$ then $\alpha\subseteq \theta$ if and only if $\lambda_A(\alpha)\in \theta^{\ast}$;
\item If $\alpha\in K(A)$ then $\alpha^{\ast}$ is exactly the principal ideal $(\lambda_A(\alpha)]$ of the lattice $L(A)$.
\end{list}
\end{lema}

According to previous lemma one can consider the order - preserving functions $u:Spec(A)\rightarrow Spec_{Id}(L(A))$ and $v:Spec_{Id}L((A))\rightarrow Spec(A)$, defined by $u(\phi) = \phi^{\ast}$ and $v(P) = P_{\ast}$, for all $\phi\in Spec(A)$ and $P\in Spec_{Id}(L(A))$. For any $\theta\in Con(A)$ we have $u(V(\theta)) = V_{Id}(\theta^{\ast})$ (see the proof of Proposition 4.17 in \cite{GM2}). By Lemma 3.4(7), for each $\alpha\in K(A)$ we have $u(V(\alpha)) = V_{Id}(\alpha^{\ast}) = V_{Id}(\lambda_A(\alpha))$.

\begin{propozitie}\cite{GM2}
The two functions $u:Spec_Z(A)\rightarrow Spec_{Id,Z}(L(A))$ and $v:Spec_{Id,Z}L((A))\rightarrow Spec_Z(A)$ are homeomorphisms, inverse to one another.
\end{propozitie}

\begin{propozitie}\cite{GM2}
The two functions $(\cdot)^{\ast}|_{RCon(A)}:RCon(A)\rightarrow Id(L(A))$ and $(\cdot)_{\ast}:Id(L(A))\rightarrow RCon(A)$ are frame isomorphisms, inverse to one another.
\end{propozitie}

Recall from \cite{Birkhoff} that the Boolean center of a bounded distributive lattice $L$ is the Boolean algebra $B(L)$ of the complemented elements in $L$. It is easy to see that $\alpha\in B(Con(A))$ implies $\lambda_A(\alpha)\in B(L(A))$, hence one can consider the function $\lambda_A|_{B(Con(A))}:B(Con(A))\rightarrow B(L(A))$.

\begin{lema}
Assume that the algebra $A$ satisfies the condition $(\star)$. If $\alpha\in K(A)$ then $\lambda_A(\alpha)\in B(L(A))$ if and only of $[\alpha,\alpha]^n\in B(Con(A))$, for some integer $n\geq 0$.
\end{lema}

\begin{proof}
Assume  that $\lambda_A(\alpha)\in B(L(A))$ so there exists $\beta\in K(A)$ such that $\lambda_A([\alpha,\beta]) = \lambda_A(\alpha)\land \lambda_A(\beta) = \Delta_A$ and $\lambda_A(\alpha\lor \beta) = \lambda_A(\alpha)\lor \lambda_A(\beta) = \nabla_A$. Thus $\alpha\lor \beta = \nabla_A$ (cf. Lemma 3.1(3)) and $[\alpha,\beta]^n = \Delta_A$, for some integer $n\geq 0$ (cf. Lemma 3.1(4)). According to Lemma 2.10, there exists an integer $m\geq 0$ such that $[[\alpha,\alpha]^m,[\beta,\beta]^m] = \Delta_A$. If we apply Lemma 2.1(2) then we get $[\alpha,\alpha]^m\in B(Con(A))$ and $[\beta,\beta]^m\in B(Con(A))$. Conversely, from $[\alpha,\alpha]^n\in B(Con(A))$ we obtain $\lambda_A(\alpha) = \lambda_A([\alpha,\alpha]^n)\in B(Con(A))$.

\end{proof}

\begin{corolar}\cite{GM2}
Assume that the algebra $A$ satisfies the condition $(\star)$. Then the function $\lambda_A|_{B(Con(A))}:B(Con(A))\rightarrow B(L(A))$ is a Boolean isomorphism.
\end{corolar}

\begin{proof}
To prove that the function $\lambda_A|_{B(Con(A))}:B(Con(A))\rightarrow B(L(A))$ is an injective Boolean morphism is straightforward (see \cite {GM2}). The surjectivity of $\lambda_A|_{B(Con(A))}$ follows by applying the previous lemma: if $\lambda_A(\alpha)\in B(L(A))$ for some $\alpha\in K(A)$ then there exists an integer $n\geq 0$ such that $[\alpha,\alpha]^n\in B(Con(A))$ and $\lambda_A([\alpha,\alpha]^n) = \lambda_A(\alpha)$.
\end{proof}

The Boolean isomorphism between $B(Con(A))$ and $B(L(A))$ will be an important tool in proving characterization theorems for some types of algebras. We observe that Lemma 3.6 and  Corollary 3.7 are the first places where condition $(\star)$ appears.

If $L$ is a bounded distributive lattice and $I$ is an ideal of $L$ then the annihilator of $I$ is the ideal $Ann(I) = \{x\in L|x\land y = 0$ for all $y\in I\}$.

\begin{lema}\cite{GKM}
If $\alpha\in K(A)$ and $\phi\in Spec(A)$ then $Ann(\lambda_A(\alpha))\in \phi^{\ast}$ if and only if $\alpha\rightarrow \rho(\Delta_A)\subseteq \phi$.
\end{lema}

The following two results improve Propositions 6.25 and 6.26 of \cite{GM2}. They show the way in which the functions $(\cdot)^{\ast}$ and $(\cdot)_{\ast}$ preserve the annihilators.

\begin{propozitie}
If $\theta\in Con(A)$ then $Ann(\theta^{\ast}) = (\theta\rightarrow \rho(\Delta_A))^{\ast}$; if $A$ is semiprime then $Ann(\theta^{\ast}) = (\theta^{\perp})^{\ast}$.
\end{propozitie}

\begin{proof}
Assume that $x\in Ann(\theta^{\ast})$, so $x = \lambda_A(\alpha)$ for some $\alpha\in K(A)$ having the property that for all $\beta\in K(A)$ such that $\beta\subseteq \theta$ we have the equality $\lambda_A([\alpha,\beta]) = \lambda_A(\alpha)\land \lambda_A(\beta) = 0$. By Lemma 3.1(6) we obtain $[\alpha,\beta]\subseteq \rho(\Delta_A)$, hence $\alpha\subseteq \beta\rightarrow \rho(\Delta_A)$. Therefore the following hold:

$\alpha\subseteq \bigcap\{\beta\rightarrow \rho(\Delta_A)|\beta\in K(A), \beta\subseteq \theta\}$ = $\bigvee\{\beta\in K(A)|\beta\subseteq \theta\}\rightarrow \rho(\Delta_A)$ = $ \theta\rightarrow \rho(\Delta_A)$.

It follows that $x = \lambda_A(\alpha)\in (\theta\rightarrow \rho(\Delta_A))^{\ast}$, hence $Ann(\theta^{\ast})\subseteq (\theta\rightarrow \rho(\Delta_A))^{\ast}$.

In order to establish the converse inclusion, assume that $x\in (\theta\rightarrow \rho(\Delta_A))^{\ast}$, hence $x = \lambda_A(\alpha)$ for some $\alpha\in K(A)$ such that $\alpha\subseteq \theta\rightarrow \rho(\Delta_A)$. If $\beta\in K(A)$ and $\beta\subseteq \theta$ then $[\alpha,\beta]\subseteq \rho(\Delta_A)$. By applying Lemma 3.1(6) we get $\lambda_A(\alpha)\land \lambda_A(\beta) = \lambda_A([\alpha,\beta]) = 0$ for all $\beta\in K(A)$ such that $\beta\subseteq \theta$, therefore $x = \lambda_A(\alpha)\in Ann(\theta^{\ast})$. Thus the converse inclusion $(\theta\rightarrow \rho(\Delta_A))^{\ast}\subseteq Ann(\theta^{\ast})$ holds.

\end{proof}

\begin{propozitie}
Assume that the algebra $A$ fulfills the condition $(\star) $. If $I$ is an ideal of the lattice $L(A)$ then $(Ann(I))_{\ast} = I_{\ast}\rightarrow \rho({\Delta_A})$; if $A$ is semiprime then $(Ann(I))_{\ast} = (I_{\ast})^{\perp}$.
\end{propozitie}

\begin{proof}
Firstly we shall prove the inclusion $(Ann(I))_{\ast}\subseteq I_{\ast}\rightarrow \rho({\Delta_A})$. Let us consider a compact congruence $\alpha$ of $A$ such that $\alpha\subseteq (Ann(I))_{\ast}$. By Lemma 3.3(1) we have $\lambda_A(\alpha)\in Ann(I)$, therefore for each $\beta\in K(A)$ with $\lambda_A(\beta)\in I$ we have $\lambda_A([\alpha,\beta]) = \lambda_A(\alpha)\land \lambda_A(\beta) = 0$. Applying Lemma 3.1(6) we get $[\alpha,\beta]\subseteq \rho({\Delta_A})$. Thus $[\alpha,I_{\ast}]$ = $\bigvee\{[\alpha,\beta]|\beta\in K(A),\lambda_A(\beta)\in I\}\subseteq \rho({\Delta_A})$, hence $\alpha\subseteq \rho({\Delta_A})$. Thus the desired inclusion was proven.

Now we shall prove that for all $\alpha\in K(A)$, $\alpha\subseteq  I_{\ast}\rightarrow \rho({\Delta_A})$ implies $\alpha\subseteq (Ann(I))_{\ast}$. Assuming that $\alpha\in K(A)$ and $\alpha\subseteq I_{\ast}\rightarrow \rho({\Delta_A})$ we get

$\bigvee\{[\alpha,\beta]|\beta\in K(A),\lambda_A(\beta)\in I\} = [\alpha,\bigvee\{\beta\in K(A)|\lambda_A(\beta)\in I\}] = [\alpha,I_{\ast}]\subseteq \rho({\Delta_A})$.

Then for all $\beta \in K(A)$, $\lambda_A(\beta)\in I$ implies $[\alpha,\beta]\subseteq \rho({\Delta_A})$. Since the algebra $A$ fulfills condition $(\star)$, by applying Lemma 2.10 it follows that there exists an integer $m\geq 0$ such that $[[\alpha,\alpha]^m,[\beta,\beta]^m] = \Delta_A$, so $\lambda_A(\alpha)\land\lambda_A(\beta)$ = $\lambda_A([\alpha,\alpha]^m])\land\lambda_A([\beta,\beta]^m)$ = $\lambda_A([[\alpha,\alpha]^m,[\beta,\beta]^m]]) = 0$. Thus $\lambda_A(\alpha)\in Ann(I)$, hence $\alpha\subseteq  (Ann(I))_{\ast}$ (cf. Lemma 3.3(1)). We conclude that $I_{\ast}\rightarrow \rho({\Delta_A})\subseteq (Ann(I))_{\ast}$.

\end{proof}

It is well - known that the reticulation of commutative rings assures a covariant functor from the category of commutative rings to the category of bounded distributive lattices (see eg. \cite{Johnstone}). If we consider the variety $\mathcal{V}$ as a category (whose morphisms are $\tau$ - morphisms) then a similar result for reticulation construction does not hold (see \cite{GKM}).

The following proposition shows that if we restrict the class of morphisms to the admissible morphisms of $\mathcal{V}$ then the reticulation provides a functor from $\mathcal{V}$ to bounded distributive lattices.

\begin{propozitie}\cite{GKM}
Assume that $u:A\rightarrow B$ is a morphism in $\mathcal{V}$ and $K(A), K(B)$ are closed under commutator. If $u$ is admissible then there exists a morphism of bounded distributive lattices $L(u): L(A)\rightarrow L(B)$ such that the following diagram is commutative:

\begin{center}
\begin{picture}(150,70)
\put(0,50){$K(A)$}
\put(25,55){\vector(1,0){100}}
\put(70,60){$u^{\bullet}|K(A)$}
\put(130,50){$K(B)$}
\put(5,45){\vector(0,-1){30}}
\put(-10,30){$\lambda_{A}$}
\put(0,0){$L(A)$}
\put(25,5){\vector(1,0){100}}
\put(70,10){$L(u)$}
\put(130,0){$L(B)$}
\put(135,45){\vector(0,-1){30}}
\put(140,30){$\lambda_{B}$}
\end{picture}
\end{center}

\end{propozitie}

Let us denote by $\mathcal{V}^{\star}$ the following category: (1) the objects of $\mathcal{V}^{\star}$ are the algebras of $\mathcal{V}$; (2) the morphisms of $\mathcal{V}^{\star}$ are the admissible morphisms of $\mathcal{V}$.

According to Proposition 3.1, it is easy to see that the assignments $A\mapsto L(A)$ and $u\mapsto L(u)$ define a covariant functor $L$ defined from the category $\mathcal{V}^{\star}$ to the category $D_{0,1}$ of bounded distributive lattices. The functor $L:\mathcal{V}^{\star}\rightarrow D_{0,1}$ will be called the reticulation functor. The paper \cite{GKM} contains a lot of properties of reticulation functor and some examples of transfer properties.

\section{Pure congruences}

\hspace{0.5cm} An ideal $I$ of a commutative ring $R$ is said to be pure (or virginal, in the terminology of \cite{Borceux2}) if for all $x\in I$ we have $I + Ann(x) = R$. Similarly, according to \cite{Cornish1}, an ideal $I$ of a bounded distributive lattice $L$ is a $\sigma$ - ideal if for all $x\in L$ we have $I\lor Ann(x)= L$. By \cite{Borceux1}, \cite{Borceux2},  the set $VId(R)$ of pure ideals in a ring $R$ is a frame; similarly, by \cite{Cornish1}, the set $VId(L)$ of $\sigma$ - ideals of a bounded distributive lattice $L$ is a frame.

\begin{lema} \cite{Borceux1}
 For all ideal $I$ in a commutative ring $R$ the following are equivalent:
\usecounter{nr}
\begin{list}{(\arabic{nr})}{\usecounter{nr}}
\item $I$ is pure;
\item The canonical ring morphism $R\rightarrow R/I$ is flat;
\item $I = \{x\in R| I + Ann(x) = R\}$.
\end{list}
\end{lema}

Let us fix a semidegenerate congruence - modular variety $\mathcal{V}$ and $A$ an algebra in $\mathcal{V}$. A congruence $\theta$ of $A$ is said to be pure if for any compact congruence $\alpha$ of $A$ such that $\alpha\subseteq \theta$ we have $\theta\lor \alpha^{\perp} = \nabla_A$. Let us denote by $VCon(A)$  the set of pure congruences of $A$.

We remark that any complemented congruence $\theta$ of $A$ is pure: if $\beta\in K(A)$ and $\beta\subseteq\theta$ then $\theta^{\perp}\subseteq\beta^{\perp}$, therefore $\nabla_A = \theta\lor\theta^{\perp}\subseteq\theta\lor\beta^{\perp}$.

\begin{lema}
 For any congruence $\theta$ of $A$ the following are equivalent:
\usecounter{nr}
\begin{list}{(\arabic{nr})}{\usecounter{nr}}
\item $\theta$ is pure;
\item $\theta =\bigvee \{\alpha\in K(A)|\theta\lor \alpha^{\perp} = \nabla_A\}$;
\item For all $\alpha\in PCon(A)$ such that $\alpha\subseteq \theta$ we have $\theta\lor \alpha^{\perp} = \nabla_A$.
\end{list}
\end{lema}

\begin{proof}

$(1)\Rightarrow (2)$ Assume that $\alpha\in K(A)$ and $\theta\lor \alpha^{\perp} = \nabla_A$. Thus $\alpha = [\theta,\alpha]$, hence $\alpha\subseteq \theta$. Then we get the inclusion $\bigvee\{\alpha\in K(A)|\theta\lor \alpha^{\perp} = \nabla_A\}\subseteq \theta$.

In order to prove the converse inclusion $\theta\subseteq \bigvee \{\alpha\in K(A)|\theta\lor \alpha^{\perp} = \nabla_A\}$, let us consider a compact congruence $\alpha$ such that $\alpha\subseteq\theta$. Thus there exist the compact congruences $\alpha_1\ldots,\alpha_n$ such that $\theta\lor \alpha_i^{\perp} = \nabla_A$, for all $i = 1,\cdots,n$. Denoting $\beta = \bigvee_{i=1}^n\alpha_i$ we obtain $\theta\lor \beta^{\perp}$ = $\theta\lor (\bigvee_{i=1}^n\alpha_i)^{\perp}$ = $\theta\lor\bigcap _{i=1}^n\alpha_i^{\perp} = \nabla_A$ (cf. Lemma 2.4(2)). We found an $\beta\in K(A)$ such that $\alpha\subseteq\beta$ and $\theta\lor \beta^{\perp} = \nabla_A$, so the converse inclusion is established.

$(2)\Rightarrow (1)$ This implication follows similarly.

$(3)\Rightarrow (1)$ Let $\alpha$ be a compact congruence of $A$ such that  $\alpha\subseteq \theta$, hence there exists a finite family $(\alpha_i)_{i=1}^n$ of principal congruences such that $\alpha = \bigvee_{i=1}^n\alpha_i$. By hypothesis we have $\theta\lor \alpha_i^{\perp} = \nabla_A$, for any $i = 1,\cdots,n$. By using Lemma 2.4(2) we obtain $\theta\lor \alpha^{\perp} = \theta\lor (\bigvee_{i=1}^n\alpha_i)^{\perp} = \theta \lor \bigcap_{i=1}^n \alpha_i^{\perp} = \nabla_A$.

$(1)\Rightarrow (3)$ Obviously.

\end{proof}

The previous lemma extends the equivalence $(2)\Leftrightarrow (3)$ of Lemma 5.1. In what follows we shall obtain a generalization of the equivalence $(1)\Leftrightarrow (2)$ of the same lemma to the algebras $A$ with an associative commutator operation.

Let $\theta$ be a congruence of $A$ and $[\theta) = \{\chi\in Con(A)|\theta\subseteq \chi\}$. By \cite{Burris} we know that $[\theta)$ is a complete lattice endowed with a residuation operation $\rightarrow_{\theta}$ defined by $\alpha \rightarrow_{\theta} \beta = (\alpha\rightarrow \beta)\lor \theta$, for all $\alpha,\beta\in [\theta)$. Let us consider the lattice isomorphism $\Phi:[\theta)\rightarrow Con(A/\theta)$ defined by $\Phi(\alpha) = \alpha/\theta$, for any $\alpha\in [\theta)$ (cf. Theorem 6.20 of \cite{Burris}). Let $\Psi:Con(A)\rightarrow [\theta)$ be the function defined by $\Psi(\alpha) = \alpha\lor \theta$, for any $\alpha\in Con(A)$. Recall that the function $p_{\theta}^{\bullet}: Con(A)\rightarrow Con(A/{\theta})$ associated with the canonical admissible morphism $p_{\theta}:A\rightarrow A/{\theta}$ is given by $p_{\theta}^{\bullet}(\chi) = Cg_{A/{\theta}
}(p_{\theta}(\chi)) = (\chi\lor\theta)/{\theta}$, for any congruence $\chi$ of $A$.

Then the following diagram is commutative:

\begin{center}
\begin{picture}(150,70)
\put(-5,45){\scriptsize$Con(A)$}
\put(55,55){$p^{\bullet}_{\theta}$}
\put(20,50){ \vector(1,0){80}}
\put(105,45){\scriptsize $Con(A/\theta)$}
\put(10,40){\vector(2,-1){45}}
\put(20,20){\scriptsize$\Psi$}
\put(65,20){\vector(2,1){45}}
\put(85,20){\scriptsize$\Phi$}
\put(50,5){\scriptsize$[\theta)$}
\end{picture}
\end{center}

\begin{lema}
For all $\alpha,\beta\in [\theta)$ we have $\alpha/\theta\rightarrow \beta/\theta = ((\alpha\rightarrow\beta)\lor \theta)/\theta$ and $\Phi(\alpha\rightarrow_{\theta}\beta) = \Phi(\alpha)\rightarrow\Phi(\beta)$.

\end{lema}

\begin{proof}
The first equality follows by using (2.2) and the reticulation property (see \cite{GM3}). Thus the following equalities hold: $\Phi(\alpha\rightarrow_{\theta}\beta)$ = $\Phi((\alpha\rightarrow\beta)\lor \theta)/\theta$ = $((\alpha\rightarrow\beta)\lor \theta)/\theta$ = $\Phi(\alpha)\rightarrow\Phi(\beta)$.
\end{proof}

\begin{lema} For any congruence $\theta$ of $A$, the following properties are equivalent:
\usecounter{nr}
\begin{list}{(\arabic{nr})}{\usecounter{nr}}
\item $p_{\theta}$ is a flat morphism in $\mathcal{V}$;
\item For all $\chi\in Con(A)$ and $\alpha\in K(A)$ we have $\alpha\rightarrow (\chi\lor \theta)\subseteq (\alpha\rightarrow \chi)\lor \theta$.
\end{list}
\end{lema}

\begin{proof}
Firstly we observe that the following equalities hold: $(\alpha\lor \theta)\rightarrow_{\theta}(\chi\lor\theta))$ = $((\alpha\lor \theta)\rightarrow(\chi\lor\theta))\lor\theta$ = $(\alpha\rightarrow(\chi\lor\theta))\lor\theta$. Thus by using the previous commutative diagram and Lemmas 2.12 and 4.3, the following properties are equivalent:

$\bullet$ $p_{\theta}$ is a flat morphism in $\mathcal{V}$;

$\bullet$ For all $\chi\in Con(A)$ and $\alpha\in K(A)$, $p^{\bullet}_{\theta}(\alpha)\rightarrow p^{\bullet}_{\theta}(\chi)\subseteq p^{\bullet}_{\theta}(\alpha\rightarrow \chi)$;

$\bullet$ For all $\chi\in Con(A)$ and $\alpha\in K(A)$, $\Psi(\alpha)\rightarrow_{\theta}\Psi(\chi)\subseteq\Psi(\alpha\rightarrow \chi)$;

$\bullet$ For all $\chi\in Con(A)$ and $\alpha\in K(A)$, $(\alpha\lor \theta)\rightarrow_{\theta}( \chi\lor \theta)\subseteq (\alpha\rightarrow \chi)\lor \theta$;

$\bullet$ For all $\chi\in Con(A)$ and $\alpha\in K(A)$, $(\alpha\rightarrow (\chi\lor \theta))\lor \theta\subseteq (\alpha\rightarrow \chi)\lor \theta$;

$\bullet$ For all $\chi\in Con(A)$ and $\alpha\in K(A)$, $\alpha\rightarrow( \chi\lor \theta)\subseteq (\alpha\rightarrow \chi)\lor \theta$.

\end{proof}

\begin{propozitie} Let $A$ be an algebra of  $\mathcal{V}$ such that the commutator operation of $Con(A)$ is associative. Then for any $\theta\in Con(A)$, the following properties are equivalent:
\usecounter{nr}
\begin{list}{(\arabic{nr})}{\usecounter{nr}}
\item $p_{\theta}$ is a flat morphism in $\mathcal{V}$;
\item $\theta$ is a pure congruence.
\end{list}
\end{propozitie}

\begin{proof}

$(1)\Rightarrow (2)$ Assume that $p_{\theta}$ is a flat morphism. Let $\alpha$ be a compact congruence of $A$ such that $\alpha\subseteq\theta$. By applying Lemma 4.4 for $\chi = \Delta_A$ we obtain $\alpha\rightarrow \theta\subseteq\theta\lor \alpha^{\perp}$. Since  $\alpha\subseteq\theta$ implies $\alpha\rightarrow \theta = \nabla_A$, we obtain $\theta\lor\alpha^{\perp} = \nabla_A$. Then $\theta$ is a pure congruence.

$(1)\Rightarrow (2)$ Assume that $\theta$ is a pure congruence of $A$. By Lemma 4.4 it suffices to prove that or all $\chi\in Con(A)$ and $\alpha\in K(A)$ we have the inclusion $\alpha\rightarrow (\chi\lor \theta)\subseteq (\alpha\rightarrow \chi)\lor \theta$. Let $\beta$ be a compact congruence of $A$ such that $\beta\subseteq \alpha\rightarrow (\chi\lor \theta)$, so $[\alpha,\beta]\subseteq \gamma\lor \delta$ for some $\gamma,\delta\in K(A)$ such that $\gamma\subseteq\chi$ and $\delta\subseteq\theta$. But $\theta$ is a pure, so $\delta\subseteq\theta$ implies $\theta\lor\delta^{\perp} = \nabla_A$, hence there exist the compact congruences $\varepsilon$ and $\xi$ such that $\varepsilon\subseteq\theta$, $\xi\subseteq\delta^{\perp}$ and $\varepsilon\lor \xi = \nabla_A$.

From $[\alpha,\beta]\subseteq \chi\lor \delta$ and $[\delta,\xi] = \Delta_A$ we obtain $[[\alpha,\beta],\xi]\subseteq [\chi\lor\delta,\xi]$ = $[\chi,\xi]\lor[\delta,\varepsilon]$ = $[\chi,\xi]\subseteq\chi$. Since the commutator operation $[\cdot,\cdot]$ of $Con(A)$ is assumed to be associative we have $[\alpha,[\beta,\xi]] = [[\alpha,\beta],\xi]\subseteq\chi$, hence $[\beta,\xi]\subseteq\alpha\rightarrow\chi$. By taking into account that $\beta = [\beta,\nabla_A] = [\beta,\varepsilon\lor\xi] = [\beta,\varepsilon]\lor[\beta,\xi]$ and $[\beta,\varepsilon]\subseteq\varepsilon\subseteq\theta$ we get $\beta\subseteq\theta\lor(\alpha\rightarrow\chi)$. We have proven that $\alpha\rightarrow (\chi\lor \theta)\subseteq (\alpha\rightarrow \chi)\lor \theta$, therefore, by using Lemma 4.4, it follows that $p_{\theta}$ is a flat morphism.

\end{proof}

\begin{remarca}
In Remark 2.3 we observed that $(Con(A),\lor,\bigcap,[\cdot,\cdot]_A,\Delta_A,\nabla_A)$  is a $mi$ - structure in the sense of \cite{GeorgescuVoiculescu2}. An abstract notion of pure element was introduced in \cite{GeorgescuVoiculescu2} under the name of virginal element. Since the virginal elements of the $mi$ - structure $(Con(A),\lor,\bigcap,[\cdot,\cdot]_A,\Delta_A,\nabla_A)$ are exactly the pure congruences of $A$ all the results of \cite{GeorgescuVoiculescu2} remain true in the framework of this paper.

A notion of pure element in a quantale was introduced in \cite{PasekaRN} and papers \cite{GG}, \cite{GG1} contain several results on the pure elements in a coherent quantale. Since the commutator operation is rarely associative, we conclude that in general $(Con(A),\lor,\bigcap,[\cdot,\cdot]_A,\Delta_A,\nabla_A)$ is not a quantale, so we cannot apply the results of \cite{GG}, \cite{GG1} for studying our pure congruences.

\end{remarca}

\begin{lema}
\cite{GeorgescuVoiculescu2} The following properties hold in the algebra $A$ of $\mathcal{V}$:
\usecounter{nr}
\begin{list}{(\arabic{nr})}{\usecounter{nr}}
\item If $\theta,\chi$ are pure congruences then $[\theta,\chi] = \theta\bigcap\chi$;
\item If a congruence $\theta$ is pure if and only if for any congruence $\alpha$ of $A$ such that $\alpha\subseteq \theta$ we have $\theta\lor \alpha^{\perp} = \nabla_A$ ;
\item If $(\theta_i)_{i\in I}\subseteq VCon(A)$ then $\bigvee_{i\in I}\theta_i\in VCon(A)$;
\item The structure $(VCon(A),\lor,\bigcap,\Delta_A,\nabla_A)$ is a spatial frame.
\end{list}
\end{lema}

A pure congruence $\theta$ of $A$ is said to be purely - prime if for all $\chi,\varepsilon\in VCon(A)$, $[\chi,\varepsilon]\subseteq\theta$ implies $\chi\subseteq\theta$ or $\varepsilon\subseteq\theta$. The set of purely - prime congruences of $A$ is denoted by $Spec(VCon(A))$ ( = the pure spectrum of $A$). We remark that $Spec(VCon(A))$  is exactly the prime spectrum of the frame $VCon(A)$ (cf. \cite{Johnstone}).

By keeping the notations from \cite{GeorgescuVoiculescu2}, for any $\theta\in Con(A)$ we define the following congruences of $A$:

$O(\theta) = \bigvee\{\alpha\in Con(A)|[\alpha,\beta] = \Delta_A,$ for some congruence $\beta\not\subseteq\theta\}$

$Ker(\theta) = \bigvee\{\alpha\in K(A)|\alpha\subseteq\theta, \theta\lor \alpha^{\perp} = \nabla_A\}$

$Vir(\theta) = \bigvee\{\alpha\in VCon(A)|\alpha\subseteq\theta\}$.

In general we have $Vir(\theta)\subseteq Ker(\theta)\subseteq \theta$ and $\theta$ is pure if and only if $\theta = Vir(\theta)$. Assuming $\alpha\in K(A)$ and $\theta\lor \alpha^{\perp} = \nabla_A$ we have $\alpha = [\alpha,\nabla_A] = [\theta,\alpha]$, therefore $\alpha\subseteq\theta$. Then the following equality holds:

$Ker(\theta) = \bigvee\{\alpha\in K(A)| \theta\lor \alpha^{\perp} = \nabla_A\}$.

If $\phi\in Spec(A)$ then it is easy to see that $O(\phi) = \bigvee\{\alpha\in K(A)|\alpha^{\perp}\not\subseteq\phi\}$.

The three operators $O(\cdot)$, $Ker(\cdot)$ and $Vir(\cdot)$ will play an important role in transferring the properties of pure congruences from algebras to bounded distributive lattices by the reticulation functor and in proving characterization theorems for some important classes of algebras.

\begin{lema}
\cite{GeorgescuVoiculescu2}
\usecounter{nr}
\begin{list}{(\arabic{nr})}{\usecounter{nr}}
\item If $\phi\in Max(A)$ then $O(\phi) = Ker(\phi)$;
\item If the congruence $\theta$ is pure then $\theta = Vir(\bigcap\{\phi\in Max(A)|\theta\subseteq\phi\})$;
\item $Vir: Spec_Z(A)\rightarrow Spec_Z(VCon(A))$ is a continuous map.
\end{list}
\end{lema}

\begin{lema}\cite{GeorgescuVoiculescu2}
The function $\rho:VCon(A)\rightarrow RCon(A)$ is an injective frame morphism, left adjoint to $Vir: RCon(A)\rightarrow VCon(A)$.
\end{lema}

The following lemma improves the property (1) from Lemma 4.7.

\begin{lema}
If $\theta\in Con(A)$ and $\chi\in VCon(A)$ then $[\theta,\chi] = \theta\bigcap\chi$.
\end{lema}

\begin{proof}
Let $\alpha$ be a compact congruence such that $\alpha\subseteq \theta\bigcap\chi$. From $\alpha\subseteq \chi$ follows that $\alpha^{\perp}\lor \chi = \nabla_A$, hence $\alpha = [\alpha,\nabla_A] = [\alpha,\alpha^{\perp}\lor \chi] = [\alpha,\alpha^{\perp}]\lor [\alpha,\chi] = [\alpha,\chi]\subseteq [\theta,\chi]$. It follows that $\theta\bigcap\chi\subseteq [\theta,\chi]  $. The converse inclusion $[\theta,\chi]\subseteq \theta\bigcap\chi$ is always true hence $[\theta,\chi] = \theta\bigcap\chi$.
\end{proof}

\begin{propozitie}
Assume that the algebra A is semiprime. If $\theta\in VCon(A)$ then $\rho(\theta) = \theta$; in particular, if $\theta\in B(Con(A))$ then $\rho(\theta) = \theta$.
\end{propozitie}

\begin{proof}
Let $\theta$ be a pure congruence of $A$. We shall prove that $\rho(\theta)\subseteq \theta$. According cu Proposition 2.8, we must prove that for any $\alpha\in K(A)$ and any integer $n\geq 0$, $[\alpha,\alpha]^n\subseteq \theta$ implies $\alpha\subseteq\theta$. From $[\alpha,\alpha]^n\subseteq \theta$ it follows that $\theta\lor ([\alpha,\alpha]^n)^{\perp} = \nabla_A$, so there exists $\beta\in K(A)$ such that $\theta\lor\beta = \nabla_A$ and $\beta\subseteq ([\alpha,\alpha]^n)^{\perp}$. Therefore $[\beta,[\alpha,\alpha]^n] = \Delta_A$, hence $\lambda_A([\beta,\alpha])$ = $\lambda_A(\beta)\land \lambda_A(\alpha)$ = $\lambda_A([\beta,[\alpha,\alpha]^n]) = 0$. Since $A$ is semiprime, by Lemma 3.1(7) we get $[\beta,\alpha] = \Delta_A$, therefore $\alpha = [\nabla_A,\alpha] = [\theta\lor \beta,\alpha] = [\theta,\alpha]\lor [\beta,\alpha] = [\theta,\alpha]$. According to Lemma 4.10 we have $\alpha = \theta\bigcap\alpha$, so $\alpha\subseteq\theta$.

If $\theta\in B(Con(A))$ then $\theta\in VCon(A)$, hence $\rho(\theta) = \theta$.

\end{proof}

\begin{lema}
If $\theta\in Con(A)$ and $\alpha\in K(A)$ then $\alpha\subseteq Ker(\theta)$ if and only if $\theta\lor\alpha^{\perp} = \nabla_A$.
\end{lema}

\begin{proof}
Assume that $\alpha\in K(A)$ and $\alpha\subseteq Ker(\theta)$, hence there exist an integer $n\geq1$ and $\beta_1,\cdots,\beta_n\in K(A)$ such that $\alpha\subseteq\bigvee_{i=1}^n\beta_i$ and $\beta_i\subseteq\theta$, $\theta\lor \beta_i^{\perp} = \nabla_A$, for all $i=1,\cdots,n$. Denoting $\beta = \bigvee_{i=1}^n\beta_i$, we have $\beta\in K(A)$ and $\theta\lor\beta^{\perp} = \theta\lor\bigcap_{i=1}^n\beta_i^{\perp} = \nabla_A$ (cf. Lemma 2.4(2)). From $\alpha\subseteq\beta$ we get $\beta^
{\perp}\subseteq\alpha^{\perp}$, hence $\theta\lor\alpha^{\perp} = \nabla_A$. The converse implication is obvious.

\end{proof}

\begin{lema}
If $\phi$ is prime congruence of the algebra $A$ of the variety $\mathcal{V}$ then $Vir(\phi) = Vir(O(\phi))$.
\end{lema}

\begin{proof}
Assume that $\phi\in Spec(A)$. Then $O(\phi)\subseteq\phi$ implies $Vir(O(\phi))\subseteq Vir(\phi)$. Let $\alpha$ be a compact congruence of $A$ such that $\alpha\subseteq Vir(\phi)$. We observe that $\alpha\subseteq Vir(\phi)$ implies the equality $Vir(\phi)\lor \alpha^{\perp} = \nabla_A$, hence $\phi\lor \alpha^{\perp} = \nabla_A$. Thus there exists $\beta\in K(A)$ such that $\phi\lor\beta = \nabla_A$ and $\beta\subseteq \alpha^{\perp}$, hence $\beta\not\subseteq\phi$ and $[\alpha,\beta] = \Delta_A$. According to the definition of $O(\phi)$, it follows that $\alpha\subseteq O(\phi)$, therefore $Vir(\phi)\subseteq O(\phi)$. This last inclusion implies $Vir(\phi)\subseteq Vir(O(\phi))$.

\end{proof}

\begin{lema}
Assume that $\theta\in Con(A)$, $\alpha_1,\alpha_2\in K(A)$ and and $k$ is a positive integer. If $[\alpha_i,\alpha_i]^k \subseteq Ker(\theta)$, for $i =1,2$, then there exists an integer $m\geq 1$ such that $[\alpha_1\lor \alpha_2,\alpha_1\lor \alpha_2]^m\subseteq Ker(\theta)$.

\end{lema}

\begin{proof}
Since $K(A)$ is closed under the commutator operation, $[\alpha_i,\alpha_i]^k\in K(A)$, for $i = 1,2$. According to Lemma 4.12, $\theta\lor [\alpha_i,\alpha_i]^{\perp} = \nabla_A$, for $i = 1,2$, so there exist $\gamma_1,\cdots,\gamma_s\in K(A)$ and $\delta_1,\cdots,\delta_t\in K(A)$ such that $\theta\lor \bigvee_{i=1}^s\gamma_i = \theta\lor \bigvee_{j=1}^t\delta_j = \nabla_A$, $[\gamma_i,[\alpha_1,\alpha_1]^k] = \Delta_A$, for $i = 1,\cdots ,s$ and $[\delta_j,[\alpha_2,\alpha_2]^k] = \Delta_A$, for $j = 1,\cdots ,t$.

Denoting $\gamma = \bigvee_{i=1}^s\gamma_i$ and $\delta = \bigvee_{j=1}^t\delta_j$ we have $\gamma,\delta\in K(A)$ and $\theta\lor \gamma = \theta\lor \delta = \nabla_A$. We observe that $[\gamma,[\alpha_1,\alpha_1]^k] = \bigvee_{i=1}^s[\gamma_i,[\alpha_1,\alpha_1]^k]] = \Delta_A$ and  $[\delta,[\alpha_,\alpha_1]^k] = \bigvee_{i=1}^t[\delta_j,[\alpha_2,\alpha_2]^k]] = \Delta_A$. If we denote $\varepsilon = [\gamma,\delta]$ then $\varepsilon\in K(A)$ and $[\varepsilon,[\alpha_1,\alpha_1]^k]$ = $[\varepsilon,[\alpha_2,\alpha_2]^k] = \Delta_A$. According to Lemma 2.4(4), $[\alpha_1,\alpha_1]^{k^2}\subseteq [\alpha_1,\alpha_1]^k\lor[\alpha_2,\alpha_2]^k$, hence

$[\varepsilon,[\alpha_1\lor\alpha_2,\alpha_1\lor\alpha_2]^{k^2}]\subseteq [\varepsilon,[\alpha_1,\alpha_1]^k\lor[\alpha_2,\alpha_2]^k]$ = $[\varepsilon,[\alpha_1,\alpha_1]^k]\lor [\varepsilon,[\alpha_2,\alpha_2]^k] = \Delta_A$.

It follows that $\varepsilon\subseteq ([\alpha_1\lor\alpha_2,\alpha_1\lor\alpha_2]^{k^2})^{\perp}$. We remark that $\nabla_A = [\nabla_A,\nabla_A] = [\theta\lor\gamma,\theta\lor\delta]\subseteq \theta\lor[\gamma,\delta] = \theta\lor\varepsilon$, so $\theta\lor\varepsilon = \nabla_A$. Thus $\theta\lor ([\alpha_1\lor\alpha_2,\alpha_1\lor\alpha_2]^{k^2})^{\perp} = \nabla_A$, therefore, according to Lemma 4.12 we have $[\alpha_1\lor\alpha_2,\alpha_1\lor\alpha_2]^{k^2}\subseteq Ker(\theta)$.

\end{proof}

The following five results will establish a strong connections between the pure congruences of $A$ and the $\sigma$ - ideals of its reticulation $L(A)$.

\begin{propozitie}
If $\theta$ is a pure congruence of $A$ then $\theta^{\ast}$ is a $\sigma$ - ideal of the lattice $L(A)$.
\end{propozitie}

\begin{proof}
Recall that $\theta^{\ast} = \{\lambda_A(\alpha)|\alpha\in K(A), \alpha\subseteq\theta\}$. We shall prove that for each $\alpha\in K(A)$ such that $\lambda_A(\alpha)\in \theta^{\ast}$ we have $\theta^{\ast}\lor Ann(\lambda_A(\alpha)) = L(A)$. One can assume that $\alpha\subseteq\theta$, so $\theta\lor\alpha^{\perp} = \nabla_A$. Since any congruence can be written as a join of compact congruences we have $\theta = \bigvee_{j\in J}\gamma_j$, for some family $(\gamma_j)_{j\in J}\subseteq K(A)$, hence $\bigvee_{j\in J}\gamma_j\lor \alpha^{\perp} = \nabla_A$. Since $\nabla_A$ is a compact congruence there exists a finite subset $J_0$ of $J$ such that $\bigvee_{j\in J_0}\gamma_j\lor \alpha^{\perp} = \nabla_A$. We denote $\gamma = \bigvee_{j\in J_0}\gamma_j$, hence $\gamma\in K(A)$, $\gamma\lor \alpha^{\perp} = \nabla_A$ and $\gamma\subseteq\theta$. Thus there exists $\delta\in K(A)$ such that $\delta\subseteq \alpha^{\perp}$ and $\gamma\lor\delta = \nabla_A$. Then we get $\lambda_A(\gamma)\in \theta^{\ast}$ and $\lambda_A(\alpha)\land \lambda_A(\delta) = \lambda_A([\alpha,\delta]) = \lambda_A(\Delta_A) = 0$. It follows that $\lambda_A(\delta)\in Ann(\lambda_A(\alpha))$. On the other hand, $\gamma\lor\delta = \nabla_A$ implies that $\lambda_A(\gamma)\lor \lambda_A(\delta) = 1$, therefore $\theta^{\ast}\lor Ann(\lambda_A(\alpha)) = L(A)$.

\end{proof}

By Proposition 4.15, one can consider the function $w:VCon(A)\rightarrow VId(L(A))$, defined by $w(\theta) = \theta^{\ast}$, for any $\theta\in VCon(A)$.

\begin{propozitie}
The function $w:VCon(A)\rightarrow VId(L(A))$ is an injective frame morphism.
\end{propozitie}

\begin{proof}
In accordance with Lemma 3.2,(1) and (2) it follows that $w$ is a frame morphism, so we shall prove only the injectivity of $w$. Let us consider two congruences $\theta_1, \theta_2\in VCon(A)$ such that $\theta_1^{\ast} = \theta_2^{\ast}$. By Lemma 3.3(2) we have $\rho(\theta_1) = (\theta_1^{\ast})_{\ast} = (\theta_2^{\ast})_{\ast} = \rho(\theta_2)$. In accordance with the following inclusions $\theta_i\subseteq\rho(\theta_i)\subseteq\bigcap\{\phi\in Max(A)|\theta_i\subseteq \phi \}$, for $i = 1,2$ and taking into account Lemma 4.8(2) it follows that

$\theta_i\subseteq Vir(\rho(\theta_i))\subseteq Vir(\bigcap\{\phi\in Max(A)|\theta_i\subseteq \phi \}) = \theta_i$, for $i = 1,2$.

Then $\theta_1 = Vir(\rho(\theta_1)) = Vir(\rho(\theta_2)) = \theta_2$.  We conclude that $w$ is injective.

\end{proof}

The following result is a generalization of Proposition 4.1 of \cite{GeorgescuVoiculescu}.

\begin{teorema}
If the algebra $A$ satisfies the condition $(\star)$ then the function $w:VCon(A)\rightarrow VId(L(A))$ is a frame isomorphism.
\end{teorema}

\begin{proof}
We know from Proposition 4.16 that $w$ is an injective frame morphism, so it remains to show that $w$ is surjective. Let $J$ be a $\sigma$ - ideal of the lattice $L(A)$. We shall prove that there exists $\theta\in VCon(A)$ such that $w(\theta) = J$.

$Claim (1)$. $J_{\ast} = \rho(Ker(J_{\ast}))$.

(i) Firstly we shall prove the inclusion $J_{\ast}\subseteq \rho(Ker(J_{\ast}))$. Let $\alpha$ be a compact congruence of $A$ such that $\lambda_A(\alpha)\in J$. Since $J$ is a $\sigma$ - ideal of the lattice $L(A)$ we have $J\lor Ann(\lambda_A(\alpha)) = L(A)$. By using Lemma 3.2(3) we obtain  $J_{\ast}\lor (Ann(\lambda_A(\alpha)))_{\ast} = L(A)$, so we get the following equality:

$J_{\ast}\lor \bigvee\{\varepsilon\in K(A)|\lambda_A(\varepsilon)\in Ann(\lambda_A(\alpha))\} = \nabla_A$.

Since $\nabla_A\in K(A)$ there exist the compact congruences $\delta_1,\cdots,\delta_n$, $\varepsilon_1,\cdots,\varepsilon_m$ such that $\bigvee_{i=1}^n\delta_i\lor\bigvee_{j=1}^m\varepsilon_j = \nabla_A$, $\delta_i\subseteq J_{\ast}$, for $i = 1,\cdots,n$ and $\lambda_A(\varepsilon_j)\in Ann(\lambda_A(\alpha))$, for $j = 1,\cdots,m$. According to Lemma 3.1(6), from $\lambda_A([\varepsilon_j,\alpha]) = \lambda_A(\varepsilon_j)\land\lambda_A(\alpha) = 0$ we obtain $[\varepsilon_j,\alpha] \subseteq \rho(\Delta_A)$, for $j = 1,\cdots,m$. The algebra $A$ fulfills the condition $(\star)$, so by using Lemma 2.10 for all $j = 1,\cdots,m$ one can find an integer $p\geq 1$ such that $[[\varepsilon_j,\varepsilon_j]^p,[\alpha,\alpha]^p] = \Delta_A$, for $j = 1,\cdots,m$.

By using Lemma 2.4(5) we have

 $[\varepsilon_1\lor\cdots\lor \varepsilon_m]^{p^m}\subseteq [\varepsilon_1,\varepsilon_1]^p\lor\cdots\lor[\varepsilon_m,\varepsilon_m]^p$

therefore we get $[[\varepsilon_1\lor\cdots\lor \varepsilon_m]^{p^m},[\alpha,\alpha]^p]\subseteq [[\varepsilon_1,\varepsilon_1]^p\lor\cdots\lor[\varepsilon_m,\varepsilon_m]^p,[\alpha,\alpha]^p]$ = $[[\varepsilon_1,\varepsilon_1]^p,[\alpha,\alpha]^p]\lor\cdots\lor [[\varepsilon_m,\varepsilon_m]^p,[\alpha,\alpha]^p] = \Delta_A$.

Denoting $\delta = \delta_1\lor\cdots\lor\delta_n, \varepsilon = \varepsilon_1\lor\cdots\lor\varepsilon_m$ we have $\delta,\varepsilon\in K(A)$, $\delta\lor\varepsilon = \nabla_A$ and $[[\varepsilon,\varepsilon]^{p^m},[\alpha,\alpha]^p] = \Delta_A$. If we denote $s = p^m$ then  we get $[[\varepsilon,\varepsilon]^s,[\alpha,\alpha]^s] = \Delta_A$, hence $[\varepsilon,\varepsilon]^s\subseteq ([\alpha,\alpha]^s])^{\perp}$. According to Lemma 2.4(3), from $\delta\lor\varepsilon = \nabla_A$ we get $[\delta,\delta]^s\lor[\varepsilon,\varepsilon]^s = \nabla_A$, so $\delta\lor [\varepsilon,\varepsilon]^s = \nabla_A$. Thus one obtains the equality $\delta\lor ([\alpha,\alpha]^s])^{\perp} = \nabla_A$. Since $\delta_i\subseteq J_{\ast}$ for $i = 1,\cdots,n $ implies $\delta\subseteq J_{\ast}$, it follows that  $J_{\ast}\lor ([\alpha,\alpha]^s])^{\perp} = \nabla_A$. By Lemma 4.12 we obtain $[\alpha,\alpha]^s\subseteq Ker(J_{\ast})$, therefore, by using Proposition 2.8 it follows the inclusion $\alpha\subseteq\rho(Ker(J_{\ast}))$. We conclude that $J_{\ast} = \bigvee\{\alpha\in K(A)|\lambda_A(\alpha)\in J\}\subseteq \rho(Ker(J_{\ast}))$.

(ii) By applying Lemma 3.3(3), from $Ker(J_{\ast})\subseteq J_{\ast}$ we get $\rho(Ker(J_{\ast}))\subseteq \rho(J_{\ast}) = J_{\ast}$.

$Claim (2)$. $Ker(J_{\ast}) = Vir(J_{\ast})$.

Since $Vir(J_{\ast})\subseteq Ker(J_{\ast})\subseteq J_{\ast}$ it suffices to prove that $Ker(J_{\ast})$ is a pure congruence of $A$. Let us consider $\alpha\in K(A)$ such that $\alpha\subseteq Ker(J_{\ast})$. We shall prove that $Ker(J_{\ast})\lor \alpha^{\perp} = \nabla_A$. According to Lemma 4.12, from $\alpha\subseteq Ker(J_{\ast})$ it follows that $J_{\ast}\lor \alpha^{\perp} = \nabla_A$, hence, by using $Claim (1)$ we obtain $\rho(Ker(J_{\ast}))\lor \alpha^{\perp} = \nabla_A$.

Since $\nabla_A\in K(A)$ there exist an integer $k\geq 1$ and the compact congruences $\gamma_1,\cdots,\gamma_s$, $\delta_1,\cdots,\delta_t$ such that $\bigvee_{i=1}^s\gamma_i\lor \bigvee_{j=1}^t\delta_j = \nabla_A$, $[\gamma_i,\gamma_i]^k\subseteq Ker(J_{\ast})$, for $i = 1,\cdots ,s$ and $[\alpha,\delta_j] = \Delta_A$, for $j = 1,\cdots ,s$. According to Lemma 2.4(5), there exists an integer $p\geq 1$ such that $[\gamma_1\lor\cdots\lor\gamma_s,\gamma_1\lor\cdots\lor\gamma_s]^p\subseteq Ker(J_{\ast})$.

Denoting $\gamma = \gamma_1\lor\cdots\lor\gamma_s$, $\delta = \delta_1\lor\cdots\lor\delta_t$ we obtain $\gamma,\delta\in K(A)$, $\gamma\lor\delta = \nabla_A$ and $[\gamma,\gamma]^p\subseteq Ker(J_{\ast})$. We observe that $[\alpha,\delta] = \bigvee_{j =1}^t[\alpha,\delta_j] = \Delta_A$, therefore $\delta\subseteq\alpha^{\perp}$. Applying Lemma 2.4(3), from $\gamma\lor\delta = \nabla_A$ we obtain $[\gamma,\gamma]^p\lor[\delta,\delta]^p = \nabla_A$. Since $[\gamma,\gamma]^p\subseteq Ker(J_{\ast})$ and $[\delta,\delta]^p\subseteq\delta\subseteq\alpha^{\perp}$, we have $Ker(J_{\ast})\lor \alpha^{\perp} = \nabla_A$. We have proven that $Ker(J_{\ast})\in VCon(A)$, so the proof of $Claim(2)$ is finished.

Let us consider the pure congruence $Vir(J_{\ast})$. By applying $Claim(2)$ and Lemma 3.3(3) we obtain $w(Vir(J_{\ast})) = w(Ker(J_{\ast})) = (Ker(J_{\ast}))^{\ast} = (\rho(Ker(J_{\ast})))^{\ast}$. By taking into account $Claim (1)$ and Lemma 3.3(2) it follows that $w(Vir(J_{\ast})) = (J_{\ast})^{\ast} = J$, so $w$ is surjective.

\end{proof}

\begin{propozitie}
If the algebra $A$ is semiprime and $J$ is a $\sigma$ - ideal of $L(A)$ then $J_{\ast}$ is a pure congruence of $A$.

\end{propozitie}

\begin{proof}
Assume that $\alpha$ is a compact congruence of $A$ such that $\alpha\subseteq J_{\ast}$. We shall prove that $J_{\ast}\lor \alpha^{\perp} = \nabla_A$. By Lemma 3.3(1), from $\alpha\subseteq J_{\ast}$ we get $\lambda_A(\alpha)\in J$, hence $J\lor Ann(\lambda_A(\alpha)) = L(A)$, because $J$ is a $\sigma$ - ideal. Then there exist the congruences $\beta,\gamma\in K(A)$ such that $\lambda_A(\beta)\in J$, $\lambda_A(\gamma)\in Ann(\lambda_A(\alpha))$ and $\lambda_A(\gamma\lor\beta) = \lambda_A(\gamma)\lor\lambda_A(\beta) = 1$. From $\lambda_A(\gamma)\in Ann(\lambda_A(\alpha))$ it results that $\lambda_A([\gamma,\alpha]) = \lambda_A(\gamma)\land\lambda_A(\alpha) = 0$. Since A is semiprime we get $[\gamma,\alpha] = \Delta_A$, hence $\gamma\subseteq \alpha^{\perp}$ (by Lemma 3.1(7)). According to Lemma 3.1(3), $\lambda_A(\beta\lor\gamma) = 1$ implies $\beta\lor\gamma = \nabla_A$, hence $\beta\lor \alpha^{\perp} = \nabla_A$. Since $\lambda_A(\beta)\in J$ implies $\beta\subseteq J_{\ast}$, it follows that $J_{\ast}\lor \alpha^{\perp} = \nabla_A$. Then $J_{\ast}$ is a pure congruence of $A$.

\end{proof}

\begin{corolar}
If $A$ is a semiprime algebra then the map $w:VCon(A)\rightarrow VId(L(A))$ is a frame isomorphism.
\end{corolar}

\begin{proof}
We remind that $w(\theta) = \theta^{\ast}$, for all pure congruence $\theta$ of $A$. By Proposition 4.16, $w$ is an injective frame isomorphism.
Let $J$ be a $\sigma$ - ideal in $L(A)$. By the previous proposition, $J_{\ast}$ is a pure congruence of $A$. By applying Lemma 3.3(2) we get $w(J_{\ast}) = (J_{\ast})^{\ast} = J$, hence $w$ is surjective. Therefore $w$ is a frame isomorphism.

\end{proof}

\begin{lema}
Assume that $A$ is a semiprime algebra and $\theta\in Con(A)$. If $\rho(\theta)$ is a pure congruence then $\theta$ is also pure.
\end{lema}

\begin{proof}
Assume that $\rho(\theta)$ is a pure congruence and $\alpha\in K(A)$ such that $\alpha\subseteq\theta$. Then $\alpha\subseteq\rho(\theta)$, so $\rho(\theta)\lor \alpha^{\perp} = \nabla_A$, therefore there exist $\beta,\gamma\in K(A)$ such that $\beta\subseteq\rho(\theta)$, $\gamma\subseteq\alpha^{\perp}$ and $\beta\lor\gamma = \nabla_A$. By Proposition 2.8, there exists an integer $n\geq 1$ such that $[\beta,\beta]^n\subseteq\theta$. Thus $[\beta,\beta]^n\lor [\gamma,\gamma]^n = \nabla_A$ (by Lemma 2.4(3)) and $[\gamma,\gamma]^n\subseteq\alpha^{\perp}$, therefore $\theta\lor \alpha^{\perp} = \nabla_A$.

\end{proof}

Let $A$ and $B$ be two algebras in $\mathcal{V}$ such that $K(A),K(B)$ are closed under the commutator operation and $f:A\rightarrow B$ a morphism in $\mathcal{V}$. We know that $f$ induces a map $f^{\bullet}:Con(A)\rightarrow Con(B)$, defined by $f^{\bullet}(\alpha) = Cg_B(f(\alpha))$, for all $\alpha\in Con(A)$.

We shall say that the morphism $f:A\rightarrow B$ preserves the purity of congruences if $\theta\in VCon(A)$ implies $f^{\bullet}(\theta)\in VCon(B))$.

\begin{propozitie} Assume that the map $f^{\bullet}:Con(A)\rightarrow Con(B)$ preserves the commutator operation: for all $\alpha,\beta\in Con(A)$, we have $f^{\bullet}([\alpha,\beta]) = [f^{\bullet}(\alpha),f^{\bullet}(\beta)]$.

Then  the morphism $f:A\rightarrow B$ preserves the purity of congruences.

\end{propozitie}

\begin{proof} Let $\theta$ be a pure congruence of $A$. In order to show that $f^{\bullet}(\theta)$ is a pure congruence of $B$, let us assume that $\beta\in K(B)$ and $\beta\subseteq f^{\bullet}(\theta)$. We know from \cite{GKM} that $f^{\bullet}$ preserves arbitrary joins, hence

$\beta\subseteq f^{\bullet}(\bigvee\{\alpha\in K(A)|\alpha\subseteq\theta\}) = \bigvee\{f^{\bullet}(\alpha)|\alpha\in K(A),\alpha\subseteq\theta\}$.

Thus there exists $\alpha\in K(A)$ such that $\alpha\subseteq\theta$ and $\beta\subseteq f^{\bullet}(\alpha)$. From $[f^{\bullet}(\alpha^{\perp_A}),f^{\bullet}(\alpha)] = f^{\bullet}([\alpha^{\perp_A},\alpha]) = f^{\bullet}(\Delta_A) = \Delta_A$ we get $f^{\bullet}(\alpha^{\perp_A})\subseteq (f^{\bullet}(\alpha))^{\perp_B}$.

Since $\theta$ is a pure congruence of $A$, from $\alpha\subseteq\theta$ one infers that $\theta\lor \alpha^{\perp_A} = \nabla_A$, hence  $f^{\bullet}(\theta)\lor f^{\bullet}(\alpha^{\perp_A}) = \nabla_B$ (because $f^{\ast}$ preserves the joins). By taking into account that $f^{\bullet}(\alpha^{\perp_A})\subseteq (f^{\bullet}(\alpha))^{\perp_B}$ it follows that $f^{\bullet}(\theta)\lor (f^{\bullet}(\alpha))^{\perp_B} = \nabla_B$. But $\beta\subseteq f^{\bullet}(\alpha)$ implies that $(f^{\bullet}(\alpha))^{\perp_B}\subseteq \beta^{\perp_B}$, therefore $\beta^{\perp_B}\lor f^{\bullet}(\theta) = \nabla_B$. Thus $f^{\bullet}(\theta)$ is a pure congruence of $B$.

\end{proof}

Remind from Proposition 3.19 that any admissible morphism $f:A\rightarrow B$ of $\mathcal{V}$ induces a lattice morphism $L(f):L(A)\rightarrow L(B)$ such that $\lambda_B(f^{\bullet}(\alpha)) = L(f)(\lambda_A(\alpha))$, for all $\alpha\in K(A)$.

\begin{lema}
Let $f:A\rightarrow B$ be an admissible morphism of $\mathcal{V}$. Then for any $\theta\in Con(A)$ we have $(f^{\bullet}(\theta))^{\ast} = (L(f)(\theta^{\ast})]$.
\end{lema}

\begin{proof}
Recall that $(f^{\bullet}(\theta))^{\ast} = \{\lambda_B(\beta)|\beta\in K(B),\beta\subseteq f^{\bullet}(\theta)\}$ and $(L(f)(\theta^{\ast})]$ is the ideal of $L(B)$ generated by the subset $L(f)(\theta^{\ast})$ of $L(B)$. In order to prove the inclusion $(f^{\bullet}(\theta))^{\ast}\subseteq (L(f)(\theta^{\ast})]$, let us consider an arbitrary $y\in (f^{\bullet}(\theta))^{\ast}$, so there exists $\beta\in K(B)$ such that $y = \lambda_B(\beta)$ and $\beta\subseteq f^{\bullet}(\theta)$. Since $f^{\bullet}$ preserves joins, we have $f^{\bullet}(\theta) = \bigvee\{f^{\bullet}(\alpha)|\alpha\in K(A), \alpha\subseteq \theta\}$.  Thus there exists $\alpha\in K(A)$ such that $\alpha\subseteq\theta$ and $\beta\subseteq f^{\bullet}(\alpha)$, therefore $\lambda_A(\alpha)\in \theta^{\ast}$ and $y = \lambda_B(\beta)\leq \lambda_B(f^{\bullet}(\alpha)) = L(f)(\lambda_A(\alpha))$ (the last equality follows from the commutative diagram of Proposition 3.11). Conclude that $y\in (L(f)(\theta^{\ast})]$.

In order to prove the converse inclusion $(L(f)(\theta^{\ast})]\subseteq (f^{\bullet}(\theta))^{\ast}$, it suffices to check that $L(f)(\theta^{\ast})\subseteq (f^{\bullet}(\theta))^{\ast}$. Assume that $x\in \theta^{\ast}$, so there exists $\alpha\in K(A)$ such that $\alpha\subseteq\theta$ and $x = \lambda_A(\alpha)$. Thus $f^{\bullet}(\alpha)\in K(B)$ (because $f^{\bullet}$ preserves the compactness of congruences) and $f^{\bullet}(\alpha)\subseteq f^{\bullet}(\theta)$, hence by using again Proposition 3.11, it follows that $L(f)(x) = L(f)(\lambda_A(\alpha)) = \lambda_B(f^{\bullet}(\alpha))\in (f^{\bullet}(\theta))^{\ast}$.

\end{proof}

Let $f:L\rightarrow L'$ a morphism in the category of bounded distributive lattices and $f^{\bullet}:Id(L)\rightarrow Id(L')$ is the induced map: for all $I\in Id(L)$, $f^{\bullet}(I)$ is the ideal $(f(I)]$ of $L'$ generated by $f(I)$.

\begin{lema}
Let $f:L\rightarrow L'$ be a morphism in the category of bounded distributive lattices. If $I$ is a $\sigma$ - ideal of $L$ then $f^{\bullet}(I)$ is a $\sigma$ - ideal of $L'$.
\end{lema}

\begin{proof}
Let $J = f^{\bullet}(I) = (f(I)]$ and $y\in J$, hence there exist $a_1,\cdots,a_n\in I$ such that $y\leq \bigvee_{i=1}^nf(a_i)$. Since $I$ is a $\sigma$ - ideal of $L$, for each $i=1,\cdots,n$ we have $Ann(a_i)\lor I = L$, hence there exist $x_i\in Ann(I)$ and $y_i\in I$ such that $x_i\lor y_i = 1$. Thus for any $i =1,\cdots,n$, we have  $x_i\land a_i = 0$, so $f(x_i)\land f(a_i) = 0$, i.e. $f(x_i)\in Ann(f(a_i))$. We remark that $f(x_i)\lor f(y_i) = 1$ and $f(y_i)\in f(I)\subseteq J$, therefore $Ann(f(a_i))\lor J = L'$, for all $i=1,\cdots,n$. Thus $\bigcap_{i=1}^nAnn(f(a_i)) = Ann(\bigvee_{i=1}^nf(a_i))\subseteq Ann(y)$, so we get

$L' = \bigcap_{i=1}^n(Ann(f(a_i))\lor J) = (\bigcap_{i=1}^nAnn(f(a_i)))\lor J\subseteq Ann(y)\lor J$.

It follows that $Ann(y)\lor J = L'$, hence $J$ is a $\sigma$ - ideal of $L'$.

\end{proof}

\begin{propozitie}
If $A,B$ are two semiprime algebras in $\mathcal{V}$ then any admissible morphism $f:A\rightarrow B$ of $\mathcal{V}$ preserves the purity of congruences.
\end{propozitie}

\begin{proof}
Let $\theta$ be a pure congruence of $A$, so $\theta^{\ast}$ is a $\sigma$ ideal of $L(A)$ (cf. Proposition 4.15) and $L(f):L(A)\rightarrow L(B)$ is a morphism of bounded distributive lattices (cf. Proposition 3.11). According to Lemma 4.23, $(L(f)(\theta^{\ast})] = (L(f))^{\bullet}(\theta^{\ast})$ is a $\sigma$ - ideal of the lattice $L(B)$. By Lemma 4.22 we have  $(f^{\bullet}(\theta))^{\ast} = (L(f)(\theta^{\ast})]$, hence, by using Lemma 3.3(2) we get $\rho_B(f^{\bullet}(\theta)) = ((f^{\bullet}(\theta))^{\ast})_{\ast} =  ((L(f)(\theta^{\ast})]_{\ast}$. In accordance with Proposition 4.18 it follows that $\rho_B(f^{\bullet}(\theta))$ is a pure congruence of $B$, so $f^{\bullet}(\theta)$ is pure (cf. Lemma 4.20).

\end{proof}

\section{Transferring the operators $Ker(\cdot)$ and $O(\cdot)$ by reticulation}

\hspace{0.5cm} Let $A$ be an algebra in a semidegenerate congruence - modular variety $\mathcal{V}$. We will assume that $K(A)$ is closed under the commutator operation.

The study of pure congruences of $A$ and other issues is strongly related to the operators $Vir(\cdot)$, $Ker(\cdot)$ and $O(\cdot)$, introduced in Section 4. If we want to transfer the properties of $\sigma$ - ideals of the lattice $L(A)$ to the pure congruences of the algebra $A$, then we need to know how the reticulation preserves the mentioned operators. Theorem 4.15 established a strong connection between the pure congruences of $A$ and the $\sigma$ - ideals of the reticulation $L(A)$, so it remains to treat the mentioned problem for  $Ker(\cdot)$ and $O(\cdot)$. In order to obtain a good description of the behaviour of reticulation w.r.t. these operators it is useful to enlarge the class of pure congruences.

A congruence $\theta$ of $A$ is said to be weakly - pure (= w - pure) if for any $\alpha\in K(A)$, $\alpha\subseteq \theta$ implies $\theta\lor (\alpha\rightarrow \rho(\Delta_A)) = \nabla_A$.

\begin{lema}
Any pure congruence of $A$ is $w$ - pure.
\end{lema}

\begin{proof}
Assume that $\theta$ is a pure congruence of $A$. If $\alpha\in K(A)$ and $\alpha\subseteq \theta$ then $\alpha^{\perp} = \alpha\rightarrow \Delta_A\subseteq\alpha\rightarrow \rho(\Delta_A)$, hence $\nabla_A = \theta\lor \alpha^{\perp}\subseteq\theta\lor (\alpha\rightarrow \rho(\Delta_A))$, so $\theta$ is $w$ - pure.
\end{proof}

In a semiprime algebra $A$ the notions of pure and $w$ - pure congruences coincide. The set of $w$ - pure congruences of $A$ will be denoted by $V_wCon(A)$. By Lemma 5.1 we have $VCon(A)\subseteq V_wCon(A)$.

\begin{propozitie}

\usecounter{nr}
\begin{list}{(\arabic{nr})}{\usecounter{nr}}

\item $V_wCon(A)$ is closed under $[\cdot,\cdot]$ and $\bigcap$;
\item For any family $(\theta_i)_{i\in I}\subseteq V_wCon(A)$ we have $\bigvee_{i\in I}\theta_i\in V_wCon(A)$.

\end{list}
\end{propozitie}

The next theorem establishes the relationship between the $w$ - pure congruences of $A$ and the $\sigma$ - ideals of $L(A)$.

\begin{teorema}
Let $\theta$  a congruence of the algebra $A$ and $J$ an ideal of the reticulation $L(A)$.
\usecounter{nr}
\begin{list}{(\arabic{nr})}{\usecounter{nr}}
\item If $\theta$ is $w$ - pure in $A$ then $\theta^{\ast}$ is a $\sigma$ - ideal of $L(A)$;
\item If $J$ is a $\sigma$ - ideal of $L(A)$ then $J_{\ast}$ is a $w$ - pure congruence of $A$.
\end{list}
\end{teorema}

\begin{proof}

(1) Assume that $\theta$  a is $w$ - pure congruence of the algebra $A$ and $x\in \theta^{\ast}$, so $x = \lambda_A(\alpha)$ for some $\alpha\in K(A)$ such that $\alpha\subseteq\theta$. Since $\theta$ is $w$ - pure we have $\theta\lor(\alpha\rightarrow\rho(\Delta_A)) = \nabla_A$, so there exist $\beta,\gamma\in K(A)$ such that $\beta\subseteq\theta$, $\gamma\subseteq \alpha\rightarrow\rho(\Delta_A)$ and $\beta\lor\gamma = \nabla_A$. Thus $\lambda_A(\beta)\in \theta^{\ast}$, $[\alpha,\gamma]\subseteq \rho(\Delta_A)$ and $\lambda_A(\alpha)\lor\lambda_A(\gamma) = \lambda_A(\alpha\lor\gamma) = 1$. By Lemma 3.1(6), from $[\alpha,\gamma]\subseteq \rho(\Delta_A)$ we get  $\lambda_A(\alpha)\land\lambda_A(\gamma) = \lambda_A([\alpha,\gamma]) = 0$, hence $\lambda_A(\gamma)\in Ann(\lambda_A(\alpha))$. It follows that $\theta^{\ast}\lor Ann(x)= \theta^{\ast}\lor Ann(\lambda_A(\alpha)) = L(A)$, so $\theta^{\ast}$ is a $\sigma$ - ideal of $L(A)$.

(2) Assume that $J$ is a $\sigma$ - ideal of $L(A)$ and $\alpha\in K(A)$ such that $\alpha\subseteq J_{\ast}$. We shall prove that $J_{\ast}\lor (\alpha\rightarrow\rho(\Delta_A)) = \nabla_A$. According to Lemma 3.3(1) we have $\lambda_A(\alpha)\in J$, hence $J\lor Ann(\lambda_A(\alpha)) = L(A)$, so there exist $\beta,\gamma\in K(A)$ such that $\lambda_A(\beta)\in J$, $\lambda_A(\gamma)\in Ann(\lambda_A(\alpha))$ and $\lambda_A(\beta\lor\gamma) = \lambda_A(\beta)\lor\lambda_A(\gamma) = 1$. By Lemmas 3.3(1) and 3.1(3) we obtain $\beta\subseteq J_{\ast}$ and $\beta\lor\gamma = \nabla_A$. From $\lambda_A(\gamma)\in Ann(\lambda_A(\alpha))$ we obtain $\lambda_A([\alpha,\gamma]) = \lambda_A(\alpha)\land \lambda_A(\gamma) = 0$. By using Lemma 3.1(6) we get $[\alpha,\gamma]\subseteq\rho(\Delta_A)$, hence $\gamma\subseteq\alpha\rightarrow \rho(\Delta_A)$. Thus $\nabla_A = \beta\lor\gamma\subseteq J_{\ast}\lor (\alpha\rightarrow \rho(\Delta_A))$, so we get $J_{\ast}\lor (\alpha\rightarrow \rho(\Delta_A)) = \nabla_A$. We conclude that $J_{\ast}$ is a $w$ - pure congruence of $A$.

\end{proof}

We observe that the previous theorem generalises Propositions 4.15 and 4.18, so it is clear why we need to define the $w$ - pure congruences of $A$.

Now we will study the way in which the reticulation preserves the operators $Ker(\cdot)$ and $O(\cdot)$. Following \cite{Cornish1}, \cite{GeorgescuVoiculescu}, if $I$ is an ideal of a bounded distributive lattice $L$ then we denote $\sigma(I) = \{x\in L|I\lor Ann(x) = L\}$. We observe that in the frame $Id(L)$ of ideals in $L$ we have $Ker(I) = \sigma(I)$ (w.r.t. the notations of \cite{GeorgescuVoiculescu2}).

For an arbitrary congruence $\theta$ of an algebra $A$ we shall define the congruence $\tilde {Ker}(\theta) = \bigvee\{\alpha\in K(A)|\theta\lor (\alpha\rightarrow \rho(\Delta_A)) = \nabla_A\}$. For any $\alpha\in K(A)$ we have $\alpha^{\perp}\subseteq \alpha\rightarrow \rho(\Delta_A)$, so $Ker(\theta)\subseteq \tilde {Ker}(\theta)$.

\begin{lema}
If $\theta\in Con(A)$ and $\alpha\in K(A)$ then $\alpha\subseteq\tilde {Ker}(\theta)$ if and only if $\theta\lor (\alpha\rightarrow \rho(\Delta_A)) = \nabla_A$.
\end{lema}

\begin{proof}
Assume that $\alpha\in K(A)$ and $\alpha\subseteq\tilde {Ker}(\theta)$ so there exist an integer $n\geq 1$ and $\beta_1,\cdots,\beta_n\in K(A)$ such that $\alpha\subseteq\bigvee_{i=1}^n\beta_i$ and $\theta\lor (\beta_i\rightarrow \rho(\Delta_A)) = \nabla_A$, for all $i =1,\cdots, n$. Denoting $\beta = \bigvee_{i=1}^n\beta_i$, we have $\beta\in K(A)$ and $\alpha\subseteq\beta$. In accordance with Lemma 2.4(2) the following equalities hold:

$\nabla_A = \theta\lor\bigcap_{i=1}^n (\beta_i\rightarrow \rho(\Delta_A)) = \theta\lor ((\bigvee_{i=1}^n\beta_i)\rightarrow \rho(\Delta_A)) = \theta\lor (\beta\rightarrow \rho(\Delta_A))$.

We observe that $\alpha\subseteq\beta$ implies $\beta\rightarrow \rho(\Delta_A)\subseteq \alpha\rightarrow \rho(\Delta_A)$, therefore $\nabla_A =  \theta\lor (\beta\rightarrow \rho(\Delta_A))\subseteq\theta\lor (\alpha\rightarrow \rho(\Delta_A))$. It follows that $\theta\lor (\alpha\rightarrow \rho(\Delta_A)) = \nabla_A$. The converse implication is obvious.

\end{proof}

\begin{teorema}
Let $A$ an algebra of the variety $\mathcal{V}$. Then the following equalities hold:
\usecounter{nr}
\begin{list}{(\arabic{nr})}{\usecounter{nr}}
\item For any congruence $\theta$ of $A$ we have $(\tilde Ker(\theta))^{\ast} = \sigma(\theta^{\ast})$;
\item For any ideal $I$ of $L(A)$ we have $(\sigma(I))_{\ast} = \tilde Ker(I_{\ast})$.
\item If $A$ is semiprime then for all $\theta\in Con(A)$ and $I\in Id(A)$ we have $(Ker(\theta))^{\ast} = \sigma(\theta^{\ast})$ and $(\sigma(I))_{\ast} = Ker(I_{\ast})$.
\end{list}
\end{teorema}

\begin{proof}

(1) Firstly we shall prove that $\sigma(\theta^{\ast})\subseteq (\tilde Ker(\theta))^{\ast}$. Assume that $x\in \sigma(\theta^{\ast})$, so $\theta^{\ast}\lor Ann(x) = L(A)$, hence there exists $\alpha\in K(A)$ such that $x = \lambda_A(\alpha)$, therefore $\theta^{\ast}\lor Ann(\lambda_A(\alpha)) = L(A)$. Then there exist $\beta,\gamma\in K(A)$ such that $\lambda_A(\beta)\in \theta^{\ast}$, $\lambda_A(\gamma)\in Ann(\lambda_A(\alpha))$ and $\lambda_A(\beta\lor\gamma) = \lambda_A(\beta)\lor\lambda_A(\gamma) = 1$. From $\lambda_A(\beta)\in \theta^{\ast}$ we get $\lambda_A(\beta) = \lambda_A(\delta)$ for some $\delta\in K(A)$ such that $\delta\subseteq\theta$, therefore $\lambda_A(\delta\lor\gamma) = \lambda_A(\delta)\lor\delta_A(\gamma) = 1$. According to Lemma 3.1(3) we have $\delta\lor\gamma = \nabla_A$. On the other hand we have $\lambda_A([\gamma,\alpha]) = \lambda_A(\gamma)\land \lambda_A(\alpha) = 0$, so $[\gamma,\alpha]\subseteq\rho(\Delta_A)$ (cf. Lemma 3.1(6)). Therefore we get $\gamma\subseteq\alpha\rightarrow \rho(\Delta_A)$, hence $\nabla_A = \delta\lor\gamma\subseteq \theta\lor (\alpha\rightarrow \rho(\Delta_A)$, i.e. $\theta\lor (\alpha\rightarrow \rho(\Delta_A) = \nabla_A$. By Lemma 5.4 it follows that $\alpha\subseteq\tilde Ker(\theta)$, i.e. $x = \lambda_A(\alpha)\in (\tilde Ker(\theta))^{\ast}$. Then one obtains the inclusion $\sigma(\theta^{\ast})\subseteq (\tilde Ker(\theta))^{\ast}$.

In order to prove the converse inclusion $(\tilde Ker(\theta))^{\ast}\subseteq \sigma(\theta^{\ast})$, assume that $x\in (\tilde Ker(\theta))^{\ast}$, so $x = \lambda_A(\alpha)$, for some $\alpha\in K(A)$ such that $\alpha\subseteq \tilde Ker(\theta)$. By Lemma 5.4 we have $\theta\lor(\alpha\rightarrow \rho(\Delta_A)) = \nabla_A$, hence, by using Lemma 3.2(2) we obtain $\theta^{\ast}\lor (\alpha\rightarrow \rho(\Delta_A))^{\ast} = L(A)$. According to Proposition 3.9, $Ann(x) = Ann(\lambda_A(\alpha)) = (\alpha\rightarrow \rho(\Delta))^{\ast}$, hence  $\theta^{\ast}\lor Ann(x) = L(A)$, i.e. $x\in \sigma(\theta^{\ast})$.

(2) In order to prove the equality $(\sigma(I))_{\ast} = \tilde Ker(I_{\ast})$ we have to show that for any compact congruence $\alpha$ of $A$, the following equivalence holds: $\alpha\subseteq (\sigma(I))_{\ast}$ if and only if $\alpha\subseteq \tilde Ker(I_{\ast})$.

Let $\alpha$ be a compact congruence of $A$ such that $\alpha\subseteq (\sigma(I))_{\ast}$. By using Lemma 3.3(1) we have $\lambda_A(\alpha)\in \sigma(I)$, hence $I\lor Ann( \lambda_A(\alpha)) = L(A)$. Then there exist $\beta,\gamma\in K(A)$ such that $\lambda_A(\beta)\in I$, $\lambda_A(\gamma)\in Ann(\lambda_A(\alpha))$ and $\lambda_A(\beta\lor\gamma) = \lambda_A(\beta)\lor \lambda_A(\gamma) = 1$, hence $\beta\subseteq I_{\ast}$ (by Lemma 3.3(1)), $\lambda_A([\alpha,\gamma]) = \lambda_A(\alpha)\land \lambda_A(\gamma) = 0$ (by Lemma 3.1(2)) and $\beta\lor\gamma = \nabla_A$ (by Lemma 3.1(3)). According to Lemma 3.1(6), from $\lambda_A([\alpha,\gamma]) = 0$ we get $[\gamma,\alpha]\subseteq \rho(\Delta_A)$, hence $\gamma\subseteq\alpha\rightarrow \rho(\Delta)$. Then $\nabla_A = \beta\lor\gamma\subseteq I_{\ast}\lor(\alpha\rightarrow \rho(\Delta))$, hence $I_{\ast}\lor(\alpha\rightarrow \rho(\Delta_A)) = \nabla_A$. By Lemma 5.4 it follows that $\alpha\subseteq \tilde Ker(I_{\ast})$.

Conversely, suppose that $\alpha\subseteq \tilde Ker(I_{\ast})$, hence $I_{\ast}\lor(\alpha\rightarrow \rho(\Delta_A)) = \nabla_A$ (cf. Lemma 5.4). Then there exist $\beta,\gamma \in K(A)$ such that $\beta\subseteq I_{\ast}$, $\gamma\subseteq\alpha\rightarrow \rho(\Delta_A)$ and $\beta\lor\gamma = \nabla_A$, therefore $\lambda_A(\beta)\in I$ (by Lemma 3.1(1)), $[\gamma,\alpha]\subseteq \rho(\Delta_A)$ and $\lambda_A(\beta)\lor\lambda_A(\gamma) =1$. According to Lemma 3.1(6), $[\gamma,\alpha]\subseteq \rho(\Delta_A)$ implies $\lambda_A(\gamma)\land \lambda_A(\alpha)$ = $\lambda_A([\gamma,\alpha]) = 0$, hence $\lambda_A(\gamma)\in Ann(\lambda_A(\alpha))$. From $\lambda_A(\beta)\in I$, $\lambda_A(\gamma)\in Ann(\lambda_A(\alpha))$ and $\lambda_A(\beta)\lor\lambda_A(\gamma) =1$ it follows that $I\lor Ann(\lambda_A(\alpha)) = L(A)$, i.e. $\lambda_A(\alpha)\in \sigma(I)$. We conclude that $\alpha\subseteq (\sigma(I))_{\ast}$.

\end{proof}

Theorem 5.5 shows that the reticulation commutes with the operator $\tilde Ker(\cdot)$; if $A$ is semiprime then the reticulation commutes with $Ker(\cdot)$.

For any $\phi\in Spec(A)$ let us consider the following congruence of the algebra $A$: $\tilde O(\phi) = \bigvee\{\alpha\in K(A)|\alpha\subseteq \phi, \alpha\rightarrow \rho(\Delta_A)\not\subseteq \phi\}$. For all $\alpha\in K(A)$ we have $\alpha^{\perp}\subseteq \alpha\rightarrow\rho(\Delta_A)$, therefore $\tilde O(\phi)\subseteq O(\phi)$.

\begin{lema}
Let  If $\phi\in Spec(A)$ and $\alpha\in K(A)$. Then the following hold:
\usecounter{nr}
\begin{list}{(\arabic{nr})}{\usecounter{nr}}
\item $\alpha\subseteq  O(\phi)$ if and only if $\alpha^{\perp}\not\subseteq\phi$;
\item $\alpha\subseteq \tilde O(\phi)$ if and only if $\alpha\subseteq\phi$ and $\alpha\rightarrow \rho(\Delta_A)\not\subseteq \phi$.

\end{list}
\end{lema}

\begin{proof}
We shall prove only (2).

Assume that $\alpha\in K(A)$ and $\alpha\subseteq \tilde O(\phi)$, hence there exist $\beta_1,\cdots,\beta_n\in K(A)$ such that $\alpha\subseteq\bigvee_{i=1}^n\beta_i$, $\beta_i\subseteq\phi$ and $\beta_i\rightarrow\rho(\Delta_A)\not\subseteq\phi$, for $i =1,\cdots,n$. Then $\alpha\subseteq\phi$ and $(\bigvee_{i=1}^n\beta_i)\rightarrow\rho(\Delta_A)$ = $\bigcap_{i=1}^n(\beta_i\rightarrow \rho(\Delta_A))\not\subseteq\phi$ (because $\phi$ is a prime congruence). Since $(\bigvee_{i=1}^n\beta_i)\rightarrow\rho(\Delta_A)\subseteq \alpha\rightarrow \rho(\Delta_A)$, it results that $\alpha\rightarrow \rho(\Delta_A)\not\subseteq \phi$. The converse implication is obvious.

\end{proof}

Recall from \cite{Cornish1}, \cite{GeorgescuVoiculescu} that if $P$ is a prime ideal in a bounded distributive lattice $L$ then we denote $O(P) = \{x\in L|Ann(x)\not\subseteq P\}$. Thus $O(P)$ is an ideal of $L$ and $O(P)\subseteq P$.

\begin{teorema}
Let $A$ an algebra of the variety $\mathcal{V}$. Then the following  hold:
\usecounter{nr}
\begin{list}{(\arabic{nr})}{\usecounter{nr}}
\item For any congruence $\phi\in Spec(A)$ we have $O(\phi^{\ast}) = (\tilde O(\phi))^{\ast}$;
\item For any prime ideal $P$ of $L(A)$ we have $(O(P))_{\ast} = \tilde O(P_{\ast})$.
\end{list}
\end{teorema}

\begin{proof}

(1) By Lemma 3.3(4), $\phi^{\ast}$ is a prime ideal of $L(A)$. Firstly we shall prove that $O(\phi^{\ast})\subseteq(\tilde O(\phi))^{\ast}$. Assume that $x\in O(\phi^{\ast})$, so $x\in \phi^{\ast}$ and $Ann(x)\not\subseteq \phi^{\ast}$. Let us take a compact congruence $\alpha$ of $A$ such that $\alpha\subseteq\phi$ and $x = \lambda_A(\alpha)$, hence $Ann(\lambda_A(\alpha))\not\subseteq\phi^{\ast}$. By Lemma 3.8(2) we get $\alpha\rightarrow\rho(\Delta_A)\not\subseteq\phi$, hence by applying Lemma 5.6(2) we obtain $\alpha\subseteq\tilde O(\phi)$. It follows that $x = \lambda_A(\alpha)\in (\tilde O(\phi))^{\ast}$, hence the inclusion $O(\phi^{\ast})\subseteq(\tilde O(\phi))^{\ast}$ is proven.

Now we suppose that $x\in (\tilde O(\phi))^{\ast}$, so there exists $\alpha\in K(A)$ such that $\alpha\subseteq\tilde O(\phi)$ and $x = \lambda_A(\alpha)$. By Lemma 5.6 it follows that $\alpha\subseteq\phi$ and $\alpha\rightarrow\rho(\Delta_A)\not\subseteq\phi$. In accordance with Lemma 3.8 we have $Ann(\lambda_A(\alpha))\not\subseteq\phi^{\ast}$, hence $x = \lambda_A(\alpha)\in O(\phi^{\ast})$. Therefore we get the inclusion $(\tilde O(\phi))^{\ast}\subseteq O(\phi^{\ast})$.

(2) By Lemma 3.3,(2) and (5), $P = (P_{\ast})^{\ast}$ and $P_{\ast}$ is a prime congruence of $A$. Let $\alpha$ be a compact congruence of $A$ such that $\alpha\subseteq(O(P))_{\ast}$, hence, by applying Lemma 3.3(6) we get $\lambda_A(\alpha)\in O(P)$. It results that $\lambda_A(\alpha)\in P$ and $Ann(\lambda_A(\alpha))\not\subseteq P = (P_{\ast})^{\ast}$. From $\lambda_A(\alpha)\in P$ we obtain $\alpha\subseteq P_{\ast}$ and by Lemma 3.8 we have $\alpha\rightarrow \rho(\Delta_A)\not\subseteq P_{\ast}$, hence $\alpha\subseteq \tilde O(P_{\ast})$ (by Lemma 5.6). It follows that $(O(P))_{\ast}\subseteq \tilde O(P_{\ast})$.

In order to prove that $\tilde O(P_{\ast})\subseteq (O(P))_{\ast}$, assume that $\alpha$ is a compact congruence of $A$ such that $\alpha\subseteq \tilde O(P_{\ast})$. By applying Lemma 5.6 we have $\alpha\subseteq P_{\ast}$ and $\alpha\rightarrow \rho(\Delta_A)\not\subseteq P_{\ast}$. By taking into account Lemmas 3.8 and 3.3(2) we get $Ann(\lambda_A(\alpha))\not\subseteq (P_{\ast})^{\ast} = P$ and by Lemma 3.3(6), from $\alpha\subseteq P_{\ast}$ we obtain $\lambda_A(\alpha)\in P$. Observing that $\lambda_A(\alpha)\in P$ and $Ann(\lambda_A(\alpha))\not\subseteq (P_{\ast})^{\ast} = P$ implies $\lambda_A(\alpha)\in O(P)$, we conclude that $\alpha\subseteq (O(P))_{\ast}$.

\end{proof}

\begin{corolar}
Let $A$ a semiprime algebra of the variety $\mathcal{V}$. Then the following  hold:
\usecounter{nr}
\begin{list}{(\arabic{nr})}{\usecounter{nr}}
\item For any congruence $\phi\in Spec(A)$ we have $O(\phi^{\ast}) = ( O(\phi))^{\ast}$;
\item For any prime ideal $P$ of $L(A)$ we have $(O(P))_{\ast} = O(P_{\ast})$.
\end{list}
\end{corolar}

Theorem 5.7 shows that the reticulation commutes with the operator $\tilde O(\cdot)$; according to Corollary 5.8, if $A$ is semiprime then the reticulation commutes with $O(\cdot)$.

The following lemma is a well-known result in lattice theory (for a proof by using lattices of fractions, see Proposition 5.9 of \cite{GG1}).

\begin{lema}
If $P$ is a prime ideal in a bounded distributive lattice $L$ then we have $O(P) = \bigcap\{Q\in Spec_{Id}(L)|Q\subseteq P\}$.
\end{lema}

\begin{propozitie}
If $\phi$ is a prime congruence in a semiprime algebra $A$, then $O(\phi) = \bigcap\Lambda(\phi)$.

\end{propozitie}

\begin{proof}
Let $\alpha\in K(A)$ such that $\alpha\subseteq O(\phi)$, so $\alpha\subseteq\phi$ and $\alpha^{\perp}\not\subseteq\phi$ (cf. Lemma 5.6(1)). For all $\psi\in \Lambda(\phi)$ we have $\alpha^{\perp}\not\subseteq\psi$, hence $\alpha\subseteq\phi$ (because $\phi$ is prime). It follows that $\alpha\subseteq \bigcap\Lambda(\phi)$, hence $O(\phi)\subseteq \bigcap\Lambda(\phi)$.

In order to prove the converse inclusion $\bigcap\Lambda(\phi)\subseteq O(\phi)$, assume that $\alpha\in K(A)$ and $\alpha\subseteq \bigcap\Lambda(\phi)$. Then for all prime congruence $\psi$ such that $\psi\subseteq\phi$ we have $\alpha\subseteq\psi$. From Lemma 3.4(4) we know that $\phi^{\ast}$ is a prime ideal of $L(A)$.  According to Proposition 3.4, it follows that for all prime ideals $Q$ of the reticulation $L(A)$ such that $Q\subseteq\phi^{\ast}$ we have $\lambda_A(\alpha)\in Q$. By Lemma 5.9 we get $\lambda_A(\alpha)\in O(\phi^{\ast})$. The algebra $A$ is semiprime, therefore by using Corollary 5.8(1), it follows that $\lambda_A(\alpha)\in (O(\phi))^{\ast}$. In accordance with Lemma 3.3(4) and Corollary 5.8 the following hold: $\alpha\subseteq ((O(\phi))^{\ast})_{\ast}$ = $(O(\phi^{\ast}))_{\ast}$ = $O((\phi^{\ast})_{\ast}) = O(\phi)$. Therefore we get $\bigcap\Lambda(\phi)\subseteq O(\phi)$.

\end{proof}

\section{Flat and patch topologies on the spectra }

\hspace{0.5cm} A topological space $(X,\Omega)$ is said to be spectral \cite{Hochster} (or coherent in the terminology of \cite{Johnstone}) if it is sober and the family $K(\Omega)$ of compact open sets of $X$ is closed under the intersections, and forms a basis for the topology. Following \cite{Hochster}, \cite{Dickmann}, \cite{Johnstone} with any spectral space $(X,\Omega)$ one can assign the following topologies:

$\bullet$ the patch topology ( = the constructible topology), having as basis the family of sets $U\bigcap V$, where $U$ is a compact open set in $X$ and $V$ is the complement of a compact open set (this topological space is denoted by $X_P$)
;

$\bullet$ the flat topology ( = the inverse topology), having as basis the family of complements of compact open sets in $X$ (this topological space is denoted by $X_F$).

\begin{lema} \cite{Dickmann}, \cite{Johnstone}
$X_P$ is a Boolean space and $X_F$ is a spectral space.
\end{lema}

\begin{remarca}
If $L$ is a bounded distributive lattice and $X$ is the spectral space $Spec_{Id,Z}(L)$ then the family $(D_{Id}(a)\bigcap V_{Id}(b))_{a,b\in L}$ is a basis of open sets for $X_P$ and the family $(V_{Id}(b))_{b\in L}$ is a basis of open sets for $X_F$.
\end{remarca}

Let us fix a semidegenerate congruence - modular variety $\mathcal{V}$ and $A$ an algebra in $\mathcal{V}$.

According to Proposition 3.4, $Spec_Z(A)$ is homeomorphic to the spectral space $Spec_{Id,Z}(L(A))$, so it is a spectral space. The family of open sets of $Spec_Z(A)$ will be denoted by $\mathcal{Z} = \mathcal{Z}_A$. For any subset $S$ of $Spec(A)$, $cl_Z(S) = V(\bigcap S)$ is the closure of $S$ in $Spec_Z(A)$; for any $\phi\in Spec(A)$, we have $cl_Z(\{\phi\}) = V(\phi)$. Now one can consider the two topologies associated with the spectral space $X = Spec_Z(A)$: the patch topology and the flat topology. We shall denote $Spec_P(A) = X_P$ and $Spec_F(A) = X_F$;  $\mathcal{P} = \mathcal{P}_A$ will be the family of open sets in $Spec_P(A)$ and  $\mathcal{F} = \mathcal{F}_A$ the family of open sets in  $Spec_F(A)$.

\begin{remarca}

\usecounter{nr}
\begin{list}{(\arabic{nr})}{\usecounter{nr}}
\item The family $(D(\alpha)\bigcap V(\beta))_{\alpha,\beta\in K(A)}$ is a basis of open sets for $Spec_P(A)$ and the family $(V(\beta))_{\beta\in K(A)}$ is a basis of open sets for $Spec_F(A)$;
\item The patch topology on $Spec(A)$ is finer that the Zariski and the flat topology on $Spec(A)$ (i.e. $\mathcal{Z}\subseteq \mathcal{P}$ and $\mathcal{F}\subseteq \mathcal{P}$);
\item The inclusions  $\mathcal{Z}\subseteq \mathcal{P}$ and $\mathcal{F}\subseteq \mathcal{P}$ show that the identity functions $id: Spec_P(A)\rightarrow Spec_Z(A)$ and $id: Spec_P(A)\rightarrow Spec_F(A)$ are continuous.
\end{list}
\end{remarca}

\begin{lema}
The two inverse functions $u:Spec(A)\rightarrow Spec_{Id}(L(A))$ and $v:Spec_{Id}(L(A))\rightarrow Spec(A)$ are homeomorphisms w.r.t. the patch and the flat topologies.
\end{lema}

\begin{remarca}
Let $A,B$ two algebras in the variety $\mathcal{V}$. For any admissible morphism $f:A\rightarrow B$ of $\mathcal{V}$ the following properties hold:
\usecounter{nr}
\begin{list}{(\arabic{nr})}{\usecounter{nr}}
\item $(f^*)^{-1}(V_A(Cg_A(a,b))) = V_B(Cg_B(f(a),f(b)))$, for all $a,b\in A$;
\item $(f^*)^{-1}(V_A(\theta)) = V_B(f^{\bullet}(\theta))$, for all $\theta\in Con(A)$.
\end{list}
\end{remarca}

\begin{proof}

We shall prove the equality (2). We know from Section 1 that $f^{\bullet}$ is the left adjoint of $f^*$. Thus for any $\psi\in Spec(B)$ the following equivalences hold: $\psi\in (f^*)^{-1}(V_A(\theta))$ iff $\theta\subseteq f^*(\psi)$ iff $f^{\bullet}(\theta)\subseteq\psi$ iff $\psi\in V_B(f^{\bullet}(\theta))$.

\end{proof}

\begin{corolar}
If $f:A\rightarrow B$  is an admissible morphism of $\mathcal{V}$ then the map $f^*: Spec(B)\rightarrow Spec(A)$ is continuous w.r.t. Zariski topology, patch topology and flat topology.
\end{corolar}

For any $\phi\in Spec(A)$ we shall denote $\Lambda(\phi) = \{\psi\in Spec(A)|\psi\subseteq \phi\}$.

\begin{propozitie}
For any $\phi\in Spec(A)$, the flat closure of the set $\{\phi\}$ is the set $cl_F(\phi) = \Lambda(\phi)$.
\end{propozitie}

\begin{proof} According to the definition of the flat closure $cl_F(\phi) = cl_F(\{\phi\})$ we have

$cl_F(\phi)$ = $\{\psi\in Spec(A)|\forall \alpha\in K(A)[\psi\in V(\alpha)\Rightarrow V(\alpha)\bigcap \{\phi\}\neq\emptyset]\}$

$\hspace{1cm}$ = $\{\psi\in Spec(A)|\forall \alpha\in K(A)[\alpha\subseteq\psi \Rightarrow \alpha\subseteq\phi]\}$

Let us prove the inclusion $cl_F(\phi)\subseteq \Lambda(\phi)$. Let $\psi\in cl_F(\phi)$ and $\alpha\in K(A)$, hence $\alpha\subseteq\psi$ implies $\alpha\subseteq\psi$, therefore

$\psi = \bigvee\{\alpha\in K(A)|\alpha\subseteq\psi\}\subseteq \bigvee\{\alpha\in K(A)|\alpha\subseteq\phi\} = \phi$.

In order to prove the converse inclusion $\Lambda(\phi)\subseteq cl_F(\phi)$, we assume that $\psi\in \Lambda(\phi)$, hence $\psi\subseteq\phi$. Therefore for all $\alpha\in K(A)$, $\alpha\subseteq\psi$ implies $\alpha\subseteq\phi$, therefore $\psi\in cl_F(\phi)$.

\end{proof}

\begin{propozitie}
If $S\subseteq Spec(A)$ is a compact subset of $Spec_Z(A)$ then its flat closure is $cl_F(S) = \bigcup_{\phi\in S}\Lambda(\phi)$.
\end{propozitie}

\begin{proof}
For any $\phi\in S$ we have $\Lambda(\phi) = cl_F(\phi)\subseteq cl_F(S)$, therefore $\bigcup_{\phi\in S}\Lambda(\phi)\subseteq cl_F(S)$. In order to prove the converse inclusion $cl_F(S)\subseteq \bigcup_{\phi\in S}\Lambda(\phi)$, assume by absurdum that there exists $\psi\in cl_F(S) - \bigcup_{\phi\in S}\Lambda(\phi)$, therefore $\psi\not\subseteq \phi$ for any $\phi\in S$. Then for any $\phi\in S$ there exists $\alpha_{\phi}\in K(A)$ such that $\alpha_{\phi}\not\subseteq\phi$ and $\alpha_{\phi}\subseteq \psi$. This means that $S\subseteq \bigcup_{\phi\in S}D(\alpha_{\phi})$, so there exist $\phi_1,\cdots,\phi_n\in S$ such that $S\subseteq \bigcup_{i=1}^nD(\alpha_{\phi_i})$. Let us denote $\alpha = \bigvee_{i=1}^n\alpha_{\phi_i}$, hence $\alpha\in K(A)$ and $S\subseteq D(\alpha)$. We observe that $\alpha\subseteq\psi$, hence $\psi\in V(\alpha)$. Since $\psi\in cl_F(S)$ and $V(\alpha)$ is a open subset of $Spec_F(L)$, it follows that $S\bigcap V(\alpha)\neq\emptyset$, contradicting $S\subseteq D(\alpha)$. We conclude that $cl_F(S)\subseteq \bigcup_{\phi\in S}\Lambda(\phi)$.

\end{proof}

\begin{propozitie}
For any $\theta\in Con(A)$, $V(\theta)$ is a compact subset of $Spec_F(A)$.
\end{propozitie}

\begin{proof} Let us consider the function $\sigma:V(\theta)\rightarrow Spec(A/\theta)$, defined by $\sigma(\phi) = \phi/\theta$, for all $\phi\in V(\theta)$. It is known that $\sigma$ is an order - isomorphism. We shall prove that $\sigma$ is a homeomorphism w.r.t. flat topology. It suffices to show that $\sigma$ maps the closed subsets of $V(\theta)$ into the closed subsets of $Spec(A/\theta)$. A basic closed subset of $V(\theta)$ has the form $V(\theta)\bigcap D(\alpha)$ = $V(\theta)\bigcap\bigcup_{i=1}^nD(\epsilon_i)$, where $\alpha = \bigvee_{i=1}^n\epsilon_i$ with $\epsilon_i\in PCon(A)$, for $i=1,\cdots,n$.

For any $i=1,\cdots,n$, the following equalities hold: $\sigma(V(\theta)\bigcap D(\epsilon_i))$ = $\{\phi/\theta|\phi\in Spec(A),\theta\subseteq\phi,\epsilon_i\not\subseteq\phi\}$ = $\{\psi\in Spec(A/\theta)|(\epsilon_i\lor \theta)/\theta\not\subseteq\psi\}$ = $D_{A/\theta}((\epsilon_i\lor\theta)/\theta)$. Then we have $\sigma(V(\theta)\bigcap D(\alpha))$ = $\sigma(\bigcup_{i=1}^n((V(\theta)\bigcap D(\epsilon_i)))$ = $\bigcup_{i=1}^n\sigma(V(\theta)\bigcap D(\epsilon_i))$ = $\bigcup D_{A/\theta}((\epsilon_i\lor\theta/\theta))$. It follows that $\sigma(V(\theta)\bigcap D(\alpha))$ is a closed subset of $Spec_F(A/\theta)$. Since $Spec_F(A/\theta)$ is a compact space it follows that $V(\theta)$ is flat compact.

\end{proof}

\begin{corolar}
The compact open subsets of $Spec_F(A)$ have the form $\bigcup_{i=1}^n V(\alpha_i)$, where $\alpha_i\in K(A)$, for all $i = 1,\cdots,n$.
\end{corolar}

\begin{proof}
Recall that an open subset $U$ of $Spec_F(A)$ has the form $U = \bigcup_{i\in I}V(\alpha_i)$, with $\alpha_i\in K(A)$, for all $i\in I$. If the open set $U = \bigcup_{i\in I}V(\alpha_i)$ is compact in $Spec_F(A)$ then  $U = \bigcup_{i\in J}V(\alpha_i)$, for some finite subset $J$ of $I$. Conversely, by Proposition 6.9, every set $\bigcup_{i=1}^n V(\alpha_i)$ (where $\alpha_i\in K(A)$, for all $i = 1,\cdots,n$) is compact in $Spec_F(A)$.

\end{proof}

 Theorem 6.9 and Corollary 6.10 generalize some properties proven in \cite{Doobs1} for the flat topology on the prime spectrum of a commutative ring (see Lemma 2.1(b) of \cite{Doobs1}).

According to the Hochster theorem from \cite{Hochster}, for any bounded distributive lattice $L$ there exists a commutative ring $R$ such that the lattice $L$ and $L(R)$ are isomorphic (see also Proposition 3.13 of Chapter V in \cite{Johnstone}, p.202 and the discussion there). Then for any algebra $A$ of $\mathcal{V}$ such that $K(A)$ is closed under the commutator operation there exists a commutative ring $R$ such that the reticulations $L(A)$ and $L(R)$ of $A$ and $R$ are isomorphic lattices.

Assume that $A$ is an algebra of $\mathcal{V}$ such that $K(A)$ is closed under the commutator operation. By Proposition 7.7 of \cite{GM2}, the lattices $L(A)$ and $L(A/{\rho(\Delta_A)})$ are isomorphic. In particular, if $R$ is an arbitrary commutative ring and $n(R)$ is its nil - radical then the reticulations $L(R)$ and $L(R/{n(R)})$ of $R$ and $R/{n(R)}$ are isomorphic lattices. Then for any algebra $A$ of $\mathcal{V}$ such that $K(A)$ is closed under the commutator operation one can find a reduced (= semiprime) commutative ring $R$ such that $L(A)$ and $L(R)$ are isomorphic.

We shall fix the (unital) commutative ring $R$ associated with the algebra $A$ and we shall identify the isomorphic lattices $L(A)$ and $L(R)$.

According to Section 3, we can consider the following order isomorphisms:

$Spec(A) \xrightarrow[]{u_A} Spec_{Id}(L(A)) \xrightarrow[]{v_R} Spec(R)$ (i)

$Spec(R) \xrightarrow[]{u_R} Spec_{Id}(L(A)) \xrightarrow[]{v_A} Spec(A)$ (ii)

Let us denote $s = s_{AR} = v_R\circ u_A$ and $t = t_{AR} = v_A\circ u_R$. It is clear that $s$ and $t$ are preserving - order maps. From Section 3 we know that $v_A$ is the inverse of $u_A$ and $v_R$ is the inverse of $u_R$, therefore $t$ is the inverse of $s$. Therefore $s$ and $t$ establish an order - isomorphism between $(Spec(A),\subseteq)$ and $(Spec(R),\subseteq)$.

\begin{lema}
The functions $s:Spec(A)\rightarrow Spec(R)$ and $t:Spec(R)\rightarrow Spec(A)$ are homeomorphisms w.r.t. Zariski topology, patch topology and flat topology, inverse to one another.
\end{lema}

\begin{proof}
By Proposition 3.4 and Lemma 6.4.
\end{proof}

The previous lemma allows us to transfer some properties from the ring $R$ to the algebra $A$. In what follows we shall illustrate this thesis many  times. For example, by using this lemma one can translate Corollary 3.6 and Lemma 3.7 of \cite{Tar1} from rings to algebras and so we obtain other proofs of Propositions 6.7 and 6.8. In fact, by using Lemma 6.1 and some other similar transfer properties the most of results in \cite{Tar1} regarding the flat and patch topologies on  spectra of rings can be exported to the flat and patch topologies on spectra of algebras.

\begin{remarca}
Since $s$ and $t$ are order - isomorphisms one can consider the restrictions of these maps to maximal spectra, hence one obtains the functions $s:Max(A)\rightarrow Max(R)$ and $t:Max(R)\rightarrow Max(A)$ (we denote the restrictions with the same symbols). By applying Lemma 6.11, it follows that $s:Max(A)\rightarrow Max(R)$ and $t:Max(R)\rightarrow Max(A)$ are homeomorphisms w.r.t. Zariski topology, patch topology and flat topology, inverse to one another.
\end{remarca}

\begin{corolar}
The topological space $Max_F(A)$ is Hausdorff and zero - dimensional.
\end{corolar}

\begin{proof}
By \cite{Tar2}, $Max_F(R)$ is Hausdorff and zero - dimensional, then we apply Remark 6.12.
\end{proof}

Let $R$ be a commutative ring. Recall from \cite{Aghajani} that an ideal $I$ of $R$ is regular if it is generated by a set of idempotents of $R$. A maximal element of the set of proper regular ideals in $R$ is said to be a max - regular ideal of $R$.

If $L$ is a frame \cite{Johnstone} then $a\in L$ is regular if it is a join of complemented elements in $L$; a maximal proper regular element of $L$ is said to be a max - regular element. Denoting by $B(R)$ the Boolean algebra of idempotent elements in $R$, we know from Lemma 1 of \cite{Banaschewski} that the assignment $e\mapsto Re$ gives  a Boolean isomorphism $B(R)\rightarrow B(Id(R))$. Thus an ideal $I$  of a commutative ring $R$ is regular (resp. max - regular) if and only if $I$ is a regular (resp. max - regular) element of the frame $Id(R)$.

Let $A$ be an algebra in $\mathcal{V}$. A congruence $\theta$ of $A$ is said to be regular if it is the join of a family $(\alpha_i)_{i\in I}\subseteq B(Con(A))$: $\theta = \bigvee_{i\in I}\alpha_i$. Then the congruence $\theta$ of $A$ is regular if and only if $\theta = \bigvee\{\alpha\in B(Con(A))|\alpha\subseteq \theta\}$. We remark that the regular congruences generalize the notion of regular ideal in a commutative ring (see \cite{Aghajani}, p.125). A maximal element in the set of proper congruences of $A$ is called a max - regular congruence. The set $Sp(A)$ of max - regular congruences of $A$ is called the Pierce spectrum of $A$. For any proper regular congruence $\theta$ there exists a max - regular congruence $\phi$ such that $\theta\subseteq\phi$, so the Pierce spectrum of $A$ is non - empty. For any $\alpha\in B(Con(A))$ we denote $U(\alpha) = \{\chi\in Sp(A)|\alpha\not\subseteq\chi\}$. Thus it is easy to show that the family $(U(\alpha))_{\alpha\in B(Con(A))}$ is a basis of clopen sets for a topology on $Sp(A)$.

For any $\phi\in Spec(A)$ denote $s_A(\phi) = \bigvee\{\alpha\in B(Con(A))|\alpha\subseteq\phi\}$. Then $s_A(\phi)$ is a regular congruence of $A$ and $s_A(\phi)\subseteq \phi\subset \nabla_A$.

\begin{lema}
For any $\phi\in Spec(A)$, $s_A(\phi)$ is a max - regular congruence of $A$.
\end{lema}

\begin{proof}
In order to prove that $s_A(\phi)$ is a max - regular congruence of $A$, it suffices to establish the following implication: $\alpha\in B(Con(A))$ and $\alpha\not\subseteq\phi$ imply that  $s_A(\phi)\lor \alpha = \nabla_A$. Assume by absurdum that there exists $\alpha\in B(Con(A))$ such that $\alpha\not\subseteq\phi$ and $s_A(\phi)\lor \alpha \neq \nabla_A$, so $s_A(\phi)\lor \alpha\subseteq\psi$, for some max - regular congruence $\psi$. Since $\alpha\not\subseteq\phi$  and $\phi\in Spec(A)$ it follows that $\neg\alpha \subseteq\phi$, so $\neg\alpha \subseteq s_A(\phi)\subseteq\psi$. Thus we obtain $\nabla_A = \alpha\lor\neg\alpha \subseteq\psi$, contradicting that $\psi$ is max - regular.

\end{proof}

By Lemma 6.14 we can consider the function $s_A: Spec(A)\rightarrow Sp(A)$, defined by the assignment $\phi\mapsto s_A(\phi)$.

\begin{propozitie}
$Sp(A)$ is a Boolean space and  $s_A: Spec(A)\rightarrow Sp(A)$ is surjective and continuous w.r.t. both Zariski and flat topologies on $Spec(A)$.

\end{propozitie}

\begin{proof}
We know  that the family $(U(\alpha))_{\alpha\in B(Con(A))}$ is a basis of clopen sets of $Sp(A)$, so this space is zero - dimensional. It is easy to see that it is Hausdorff.
In order to show that $s_A$ is surjective, assume that $\psi\in Sp(A)$, then there exists $\phi\in Spec(A)$ such that $\psi\subseteq\phi$, hence $\psi\subseteq s_A(\phi)$. Since $\psi$ and $s_A(\phi)$ are max - regular we have $\psi = s_A(\phi)$, so $s_A: Spec(A)\rightarrow Sp(A)$ is surjective. For all $\alpha\in B(Con(A))$ we have $s_A^{-1}(U(\alpha)) = D(\alpha) = V(\neg\alpha)$, so the function $s_A$ is surjective and continuous w.r.t. both Zariski and flat topologies on $Spec(A)$. Since $Spec_Z(A)$ is compact and $Sp(A) = Im(s_A)$, it follows that the topological space $Sp(A)$ is compact. Thus $Sp(A)$ is a Boolean space.

\end{proof}

\begin{lema}
Any regular congruence of $A$ is pure.
\end{lema}

\begin{proof}
Firstly we shall prove that any complemented congruence $\chi$ of $A$ is pure. Let $\alpha\in K(A)$ such that $\alpha\subseteq\chi$, so $\chi^{\perp}\subseteq\alpha^{\perp}$. We know that in the Boolean algebra $B(Con(A))$ we have $\neg\chi = \chi{\perp}$. Thus $\nabla_A = \chi\lor\chi{\perp}\subseteq\chi\lor\alpha^{\perp}$, so $\chi$ is pure.

Now let us consider an arbitrary regular congruence $\chi$ of $A$ and $\alpha\in B(Con(A))$ such that $\alpha\subseteq\chi$. We know that $\alpha$ is pure, so $\alpha = Vir(\alpha)\subseteq Vir(\chi)$. It follows that $\chi = \bigvee\{\alpha\in B(Con(A))|\alpha\subseteq\chi\}\subseteq Vir(\chi)$, therefore $\chi = Vir(\chi)$. Conclude that $\chi$ is pure.

\end{proof}

For any $\phi\in Spec(VCon(A))$ let us consider the following congruence of $A$: $t_A(\phi) =\bigvee \{\alpha\in B(Con(A))|\alpha\subseteq\phi\}$. Then $t_A(\phi)$ is a regular congruence and $t_A(\phi)\subseteq\phi\subset \nabla_A$.

\begin{lema}
If $\phi\in Spec(VCon(A))$ then $t_A(\phi)$ is a max - regular congruence.
\end{lema}

\begin{proof}
Assume that $\phi\in Spec(VCon(A))$. In order to show that $t_A(\phi)$ is a max - regular congruence, it suffices to have the following implication: $\alpha\in B(Con(A))$ and $\alpha\not\subseteq\phi$ imply $t_A(\phi)\lor\alpha = \nabla_A$.

Suppose by absurdum that there exists $\alpha\in B(Con(A))$ such that  $\alpha\not\subseteq\phi$ and $t_A(\phi)\lor\alpha\neq \nabla_A$. Since $t_A(\phi)\lor\alpha$ is regular, there exists a max - regular element $\psi$ such that $t_A(\phi)\lor\alpha\subseteq\psi$. We know that $VCon(A)$ is a frame (cf. Lemma 4.7(4)). Since the complemented congruences $\alpha$ and $\neg\alpha$ are pure (cf. Lemma 6.16) and $\phi\in Spec(VCon(A))$, it follows that $\alpha\not\subseteq\phi$ and $\alpha\land\neg\alpha = \Delta_A$ imply that $\neg\alpha\subseteq\phi$, hence $\neg\alpha = t_A(\neg\alpha)\subseteq t_A(\phi)$. Thus $\nabla_A = \alpha\lor\neg\alpha\subseteq\alpha\lor t_A(\phi)\subseteq \psi$, contradicting that $\psi$ is a proper regular congruence. It follows that $t_A(\phi)\lor\alpha = \nabla_A$.
\end{proof}

According to the previous lemma, the assignment $\phi\mapsto t_A(\phi)$ defines a map $t_A: Spec(VCon(A))\rightarrow Sp(A)$.

\begin{propozitie}
The map $t_A: Spec(VCon(A))\rightarrow Sp(A)$ is continuous and surjective.
\end{propozitie}

\begin{proof}
Assume that $\psi\in Sp(A)$ then there exists $\varepsilon\in Max(A)$ such that $\psi\subseteq\varepsilon$. By Lemma 6.16, the max - regular congruence $\psi$ is pure, hence $\psi\subseteq Vir(\varepsilon)$. Since $\psi$ is a regular congruence, we have $\psi = t_A(\psi)\subseteq t_A(Vir(\varepsilon))$. Thus $\psi = t_A(Vir(\varepsilon))$, because both  $\psi$ and  $t_A(Vir(\varepsilon))$ are max - regular. It follows that $t_A$ is surjective. To prove that $t_A$ is a continuous map is straightforward.

\end{proof}

Let $L$ be a bounded distributive lattice and $Id(L)$ the frame of its ideals. An ideal $I$ of $L$ is regular if it is generated by a set of complemented elements of $L$. A max - regular ideal of $L$ is a maximal element of the set of proper regular ideals in $L$. The ideal - Pierce spectrum of the bounded distributive lattice $L$ is the set $Sp_{Id}(L)$ of  max - regular ideals of $L$. Of course the ideal - Pierce spectrum $Sp_{Id}(L)$ is a Boolean space w.r.t. a canonical topology. We observe that $Sp_{Id}(L)$  concides with the Pierce spectrum introduced by Cignoli in \cite{Cignoli}.

It is well - known that the assignment $e\mapsto (e]$ defines a Boolean isomorphism $B(L)\rightarrow B(Id(L))$. It follows that an ideal $I$ of $L$ is regular (resp. max - regular) if and only if $I$ is a regular (resp. max - regular) element of the frame $Id(L)$.

The following two propositions emphazise the way in which

$\bullet$ the map $(\cdot)^{\ast}$ transforms regular (resp. max - regular) congruences of $A$ into (resp. max - regular) ideals of $L(A)$;

$\bullet$ the map $(\cdot)_{\ast}$ transforms regular (resp. max - regular) ideals of $L(A)$ into (resp. max - regular) congruences of $A$.

\begin{propozitie}
Let $A$ an algebra of the variety $\mathcal{V}$. Then the following  hold:
\usecounter{nr}
\begin{list}{(\arabic{nr})}{\usecounter{nr}}
\item If $\theta$ is a regular congruence of $A$ then $\theta^{\ast}$ is a regular ideal of $L(A)$;
\item Assume that $A$ is semiprime. If $I$ is a regular ideal of $L(A)$ then $I_{\ast}$ is a regular congruence of $A$.
\end{list}
\end{propozitie}

\begin{proof}

(1) If $\theta$ is a regular congruence of $A$ then $\theta = \bigvee_{i\in J}\alpha_i$, for some family $(\alpha_i)_{i\in J}\subseteq B(Con(A))$, hence $\lambda_A(\alpha_i)\in B(L(A))$, for all $i\in J$. By Lemmas 3.2(1) and 3.3(7), we have $\theta^{\ast} = \bigvee_{i\in J}\alpha_i^{\ast}$ and $\alpha_i^{\ast} = (\lambda_A(\alpha_i)]\in B(Id(L(A)))$, so $\theta^{\ast}$ is a regular ideal of $L(A)$.

(2) Assume that $A$ is semiprime and $I$ is a regular ideal of $L(A)$, hence $I = \bigvee\{(e]|e\in B(L(A))\bigcap I\}$. We shall prove that $I_{\ast} = \bigvee\{\alpha\in B(Con(A))|\alpha\subseteq I_{\ast}\}$.

Let $\beta\in K(A)$ such that $\beta\subseteq I_{\ast}$, so $\lambda_A(\beta)\in I$ (by Lemma 3.3(1)), so there exists $e\in B(L(A))\bigcap I$ such that $\lambda_A(\beta)\leq e$ (because $I$ is a regular ideal of $L(A)$). Since $A$ is semiprime, the function $\lambda_A|_{B(Con(A))}:B(Con(A))\rightarrow B(L(A))$ is a Boolean isomorphism (see Proposition 6.18(iii) of \cite{GM2}). Thus there exists $\alpha\in B(Con(A))$ such that $e = \lambda_A(\alpha)$, so $\lambda_A(\beta)\leq \lambda_A(\alpha)$, i.e. $\rho(\beta)\subseteq\rho(\alpha)$. Since $A$ is semiprime, one can apply Proposition 4.11 to complemented congruence $\alpha$, therefore $\beta\subseteq\rho(\alpha) = \alpha$. We have proven the inclusion $I_{\ast}\subseteq \bigvee\{\alpha\in B(Con(A))|\alpha\subseteq I_{\ast}\}$. The inclusion $\bigvee\{\alpha\in B(Con(A))|\alpha\subseteq I_{\ast}\}\subseteq I_{\ast}$ is obvious, so  $I_{\ast} = \bigvee\{\alpha\in B(Con(A))|\alpha\subseteq I_{\ast}\}$. We conclude that $I_{\ast}$ is a regular congruence of $A$.

\end{proof}

\begin{propozitie}
Let $A$ a semiprime algebra of the variety $\mathcal{V}$. Then the following  hold:
\usecounter{nr}
\begin{list}{(\arabic{nr})}{\usecounter{nr}}
\item If $\theta$ is a max - regular congruence of $A$ then $\theta^{\ast}$ is a max - regular ideal of $L(A)$;
\item If $I$ is a max - regular ideal of $L(A)$ then $I_{\ast}$ is a max - regular congruence of $A$.
\end{list}
\end{propozitie}

\begin{proof}

(1) Assume that $\theta$ is a max - regular congruence of $A$. By the previous proposition, $\theta^{\ast}$ is a regular ideal of $L(A)$. In order to prove that $\theta^{\ast}$ is a max - regular ideal of $L(A)$, it suffices to show that for each $e\in B(L(A))$, $e\notin \theta^{\ast}$ implies $\theta^{\ast}\lor (e] = L(A)$. Since $A$ is semiprime, by using Proposition 6.18(iii) of \cite{GM2}, it follows that the function $\lambda_A|_{B(Con(A))}:B(Con(A))\rightarrow B(L(A))$ is a Boolean isomorphism, so there exists $\alpha\in B(Con(A))$ such that $e = \lambda_A(\alpha)$. From $e\notin \theta^{\ast}$ we get $\alpha\not\subseteq\theta$, hence $\theta\lor\alpha = \nabla_A$ (because $\theta$ is max - regular). By Lemma 3.3(7) we have $ \theta^{\ast}\lor (e] = \theta^{\ast}\lor\alpha^{\ast} = L(A)$.

(2) Assume that $I$ is a max - regular ideal of $L(A)$. By the previous proposition, $I_{\ast}$ is a regular congruence of $A$. In order to establish that $I_{\ast}$ is a max - regular congruence of $A$, it suffices to prove for each $\alpha\in B(Con(A))$, $\alpha\not\subseteq I_{\ast}$ implies $I_{\ast}\lor \alpha = \nabla_A$.
Let us assume that $\alpha\in B(Con(A))$ and $\alpha\not\subseteq I_{\ast}$, so $\lambda_A(\alpha)\in B(L(A))$ and $\lambda_A(\alpha)\notin I$ (cf. Lemma 3.3(1)). Since $I$ is a max - regular ideal it follows that $I\lor \alpha^{\ast} = I\lor(\lambda_A(\alpha)] = L(A)$. By applying Lemma 3.2(3), we get $\rho(I_{\ast}\lor \rho(\alpha)) = \rho(I_{\ast}\lor (\alpha^{\ast})_{\ast}) = (I\lor \alpha^{\ast})_{\ast} = (L(A))_{\ast} = \nabla_A$, hence, by using Lemma 2.7(6), it results that $I_{\ast}\lor \alpha = \nabla_A$.

\end{proof}

\begin{lema}
If $\theta$ is a regular congruence of $A$ then $\rho(\theta) = \theta$.
\end{lema}

\begin{proof}
Assume that $\theta$ is a regular congruence of $A$. Let $\alpha$ be a compact congruence of $A$ such that $\alpha\subseteq\rho(\theta)$, hence $[\alpha,\alpha]^n\subseteq\theta$, for some integer $n\geq 1$. Since $\theta = \bigvee\{\beta\in B(Con(A))|\beta\subseteq\theta\}$ and $[\alpha,\alpha]^n\in K(A)$, there exists $\beta\in B(Con(A))$ such that $\beta\subseteq\theta$ and $[\alpha,\alpha]^n\subseteq\beta$, so $\alpha\subseteq\rho(\beta)$. According to Proposition 4.11, $\beta\in B(Con(A))$ implies $\rho(\beta) = \beta$, hence $\alpha\subseteq\beta\subseteq\theta$. We have proven that $\rho(\theta)\subseteq\theta$, so $\rho(\theta) = \theta$.
\end{proof}

\begin{lema} Let $\theta$ be a congruence of $A$.
If $\rho(\theta)$ is a regular congruence of $A$ then $\rho(\theta) = \theta$.
\end{lema}

\begin{proof}
If $\rho(\theta)$ is a regular then $\rho(\theta) = \bigvee\{\alpha\in B(Con(A))|\alpha\subseteq\rho(\theta)\}$. Let $\alpha\in B(Con(A))$ such that $\alpha\subseteq\rho(\theta)$, hence there exists an integer $n\geq 0$ such that $[\alpha,\alpha]^n\subseteq\theta$. But $\alpha\in B(Con(A))$ implies $[\alpha,\alpha]^n = \alpha$, hence $\alpha\subseteq\theta$. We have proven that $\rho(\theta)\subseteq \theta$, so $\rho(\theta) = \theta$.
\end{proof}

\begin{lema}
Let $A$ a semiprime algebra of the variety $\mathcal{V}$ and $\theta\in Con(A)$. Then $\theta$ is a regular congruence of $A$ if and only if $\theta^{\ast}$ is a regular ideal of $L(A)$.
\end{lema}

\begin{proof}
If $\theta$ is a regular congruence of $A$ then $\theta^{\ast}$ is a regular ideal of $L(A)$ (by Proposition 6.19(1)). Conversely, assume that $\theta^{\ast}$ is a regular ideal of $L(A)$, hence, by applying Proposition 6.19(2), it follows that $\rho(\theta) = (\theta^{\ast})_{\ast}$ is a regular congruence of $A$. According to Lemma 6.22 we have $\rho(\theta) = \theta$, so $\theta$ is a regular congruence of $A$.
\end{proof}

Let $A$ be a semiprime algebra of $\mathcal{V}$. According to Proposition 6.20 we can consider the maps:

$(\cdot)^{\ast}|_{Sp(A)}:Sp(A)\rightarrow Sp_{Id}(L(A))$ and $(\cdot)_{\ast}|_{Sp_{Id}(L(A))}:Sp_{Id}(L(A))\rightarrow Sp(A)$.

\begin{propozitie}
If $A$ is a semiprime algebra of $\mathcal{V}$, then the following maps $(\cdot)^{\ast}|_{Sp(A)}:Sp(A)\rightarrow Sp_{Id}(L(A))$ and $(\cdot)_{\ast}|_{Sp_{Id}(A))}:Sp_{Id}(L(A)\rightarrow Sp(A)$ are homeomorphism, inverse to one another.

\end{propozitie}

\begin{proof}
By Lemmas 3.3(4) and 6.21, for any $\theta\in Sp(A)$ we have $(\theta^{\ast})_{\ast} = \rho(\theta) = \theta$. Applying Lemma 3.3(2), for any $I\in Sp_{Id}(L(A))$, we have $(I_{\ast})^{\ast} = I$. To prove that these two maps are continuous is straightforward.
\end{proof}

\section{Hyperarchimedean algebras}

\hspace{0.5cm} The hyperarchimedean algebras in a semidegenerate congruence - modular variety $\mathcal{V}$ were introduced in \cite{GM2} as generalizations of the zero - dimensional rings. In this section we obtain new characterizations of hyperarchimedean algebras, as well as other results that improve those of \cite{GM2}. Our results generalize some theorems of \cite{Aghajani}.

Let us fix a semidegenerate congruence - modular variety $\mathcal{V}$ and $A$ an algebra in $\mathcal{V}$. We know that the quotient algebra $A/\rho(\Delta_A)$ is semiprime and the ordered sets $(Spec(A),\subseteq)$ and $(Spec(A/\rho(\Delta_A),\subseteq)$ are order - isomorphic. If we restrict this order - isomorphism to maximal spectra then it results a bijection between $Max(A)$ and $Max(A/\rho(\Delta))$.

\begin{lema}
The topological spaces $Spec_Z(A)$ and $Spec_Z(A/\rho(\Delta_A))$ (resp. $Spec_P(A)$ and $Spec_P(A/\rho(\Delta_A))$, resp. $Spec_P(A)$ and $Spec_P(A/\rho(\Delta_A))$) are homeomorphic.
\end{lema}

A similar result for the maximal spectra of $A$ and $A/\rho(\Delta_A)$ also holds.

Recall from \cite{GM2} that the algebra $A$ is said to be hyperarchimedean if for all $\alpha\in PCon(A)$ there exists an integer $k\geq 1$ such that $[\alpha,\alpha]^k\in B(Con(A))$.

In what follows $A$ will be an algebra of $\mathcal{V}$ such that $K(A)$ is closed under the commutator operation, so one can work with the reticulation $L(A)$ of $A$.

Assume that $A$ satisfies the condition $(\star)$. By applying Lemma 3.6, it follows that $A$ is hyperarchimedean if and only if $\lambda_A(\alpha)\in B(Con(A))$, for all $\alpha\in PCon(A)$. If $A$ is hyperarchimedean and $\alpha\in K(A)$ then there exists an integer $n\geq 1$ and $\beta_1,\cdots,\beta_n\in PCon(A)$ such that $\alpha = \bigvee_{i=1} ^n\beta_i$, therefore $\lambda_A(\alpha) = \bigvee_{i=1}^n\lambda_A(\beta_i)\in B(Con(A))$.

It follows that if the algebra $A$ satisfies the condition $(\star)$, then $A$ is hyperarchimedean if and only if for all $\alpha\in K(A)$ there exists an integer $k\geq 1$ such that $[\alpha,\alpha]^k\in B(Con(A))$.

\begin{propozitie}
\cite{GM2} The following assertions are equivalent:
\usecounter{nr}
\begin{list}{(\arabic{nr})}{\usecounter{nr}}
\item $A$ is hyperarchimedean;
\item $A/\rho(\Delta_A)$ is hyperarchimedean;
\item $Spec(A) = Max(A)$;
\item $L(A)$ is a Boolean algebra.
\end{list}
\end{propozitie}

\begin{propozitie}
\cite{GG} If $L$ is a bounded distributive lattice then the following assertions are equivalent:
\usecounter{nr}
\begin{list}{(\arabic{nr})}{\usecounter{nr}}
\item $L$ is a Boolean algebra:
\item $Spec(L) = Max(L)$;
\item For any distinct prime ideals $M,N$ of $L$ there exist $x,y\in L$ such that $x\notin M$, $y\notin N$ and $x\land y = 0$.
\end{list}
\end{propozitie}

The equivalence of the properties (1) and (2) of Proposition 7.3 is the well - known Nachbin theorem.

\begin{propozitie}
Let us consider the following conditions:
\usecounter{nr}
\begin{list}{(\arabic{nr})}{\usecounter{nr}}
\item For all distinct prime congruences $\phi,\psi$ of $A$ there exist $\alpha,\beta\in PCon(A)$ such that $\alpha\not\subseteq\phi$, $\beta\not\subseteq\psi$ and $[\alpha,\beta] = \Delta_A$;
\item For all distinct prime ideals $I,J$ of $L(A)$ there exists $x,y\in L$ such that $x\notin I$, $y\notin J$ and $x\land y = 0$.

Then $(1)\Rightarrow (2)$; if $A$ is semiprime then $(b)\Rightarrow (a)$.
\end{list}
\end{propozitie}

\begin{proof}
Firstly we shall prove the implication $(1)\Rightarrow (2)$. Let $I, J$ be  two distinct prime ideals of the lattice $L(A)$. According to Proposition 3.4 there exist two distinct prime congruences $\phi,\psi$ of $A$ such that $I = \phi^{\ast}$ and $J = \psi^{\ast}$. By hypothesis (1), there exist $\alpha,\beta\in PCon(A)$ such that $\alpha\not\subseteq\phi$, $\beta\not\subseteq\psi$ and $[\alpha,\beta] = \Delta_A$. Thus $\lambda_A(\alpha)\notin \phi^{\ast}$ and $\lambda_A(\beta)\notin \psi^{\ast}$ (cf. Lemma 3.3(6)) and $\lambda_A(\alpha)\land\lambda_A(\beta) = \lambda_A([\alpha,\beta]) = 0$.

In order to prove the second part of proposition, let us assume that the algebra $A$ is semiprime and $\phi,\psi$ are two distinct prime congruences of $A$, so $I = \phi^{\ast}, J = \psi^{\ast}$ are distinct prime ideals in the lattice $L(A)$. According to the hypothesis (2), there exist $\alpha,\beta\in K(A)$ such that $\lambda_A(\alpha)\notin \phi^{\ast}$, $\lambda_A(\beta)\notin \psi^{\ast}$ and $\lambda_A([\alpha,\beta]) = \lambda_A(\alpha)\land\lambda_A(\beta) = 0$. By using again Lemma 3.3(6) we get $\alpha\not\subseteq\phi$ and $\beta\not\subseteq\psi$. But the algebra $A$ is semiprime, hence from $\lambda_A([\alpha,\beta]) = 0$ we obtain  $[\alpha,\beta] = 0$. Since the compact congruences are exactly the finitely generated congruences of $A$, we can find $\gamma,\delta\in PCon(A)$, such that $\gamma\subseteq\alpha$, $\delta\subseteq\beta$, $\gamma\not\subseteq\phi$ and $\delta\not\subseteq\psi$. It is clear that $[\gamma,\delta] = 0$, so the property (1) is fulfilled.

\end{proof}

\begin{propozitie}
 The following assertions are equivalent:
\usecounter{nr}
\begin{list}{(\arabic{nr})}{\usecounter{nr}}
\item $A$ is hyperarchimedean;
\item For all distinct prime congruences $\phi,\psi$ of $A$ there exist $\alpha,\beta\in PCon(A)$ such that $\alpha\not\subseteq\phi$, $\beta\not\subseteq\psi$ and $[\alpha,\beta]\subseteq \rho(\Delta_A)$;
\item For all distinct prime congruences $\chi,\nu$ of $A/\rho(\Delta_A)$ there exist $\alpha,\beta\in PCon(A/\rho(\Delta_A))$ such that $\alpha\not\subseteq\chi$, $\beta\not\subseteq\nu$ and $[\alpha,\beta] = \Delta_{A/\rho(\Delta_A)}$.
\end{list}
\end{propozitie}

\begin{proof}

$(1)\Rightarrow (2)$
Assume that $A$ is hyperarchimedean and $\phi,\psi$ are two distinct prime congruences of the algebra $A$, therefore the prime ideals $\phi^{\ast},\psi^{\ast}$ of the lattice $L(A)$ are distinct. According to Proposition 6.2, $L(A)$ is a Boolean algebra, hence we can find two elements $x,y$ of $L(A)$ such $x\notin \phi^{\ast},y\notin \psi^{\ast}$ and $x\land y = 0$ (cf. Proposition 7.3(3)). Then there exist $\alpha,\beta\in K(A)$ such that $x = \lambda_A(\alpha), y = \lambda_A(\beta)$, hence $\alpha\not\subseteq\phi$ and $\beta\not\subseteq\psi$ (cf. Lemma 3.3(6)). From $\lambda_A([\alpha,\beta]) = \lambda_A(\alpha)\land \lambda_A(\beta) = x\land y = 0$ we infer that $[\alpha,\beta]\subseteq \rho(\Delta_A)$ (cf. Lemma 3.1(6)). Since $\alpha$ and $\beta$ are finitely generated congruences of $A$, one can find $\gamma,\delta\in PCon(A))$ such that $\gamma\not\subseteq\phi$, $\delta\not\subseteq\psi$ and $[\beta,\gamma]\subseteq \rho(\Delta_A)$.

$(3)\Rightarrow (1)$
Recall that $A/\rho(\Delta_A)$ is semiprime, hence by Propositions 7.4 and 7.3 it follows that $L(A/\rho(\Delta_A))$ is a Boolean algebra, so $A/\rho(\Delta_A)$ is a hyperarchimedean algebra (by Proposition 7.2). By a new application of Proposition 7.2, it follows that $A$ is hyperarchimedean.

$(2)\Leftrightarrow (3)$ By using that $(Spec(A),\subseteq)$ and $(Spec(A/\rho(\Delta_A)),\subseteq)$ are order - isomorphic.

\end{proof}

The following lemma is an elementary result in topology.

\begin{lema}
\cite{Tar1}
Assume that $X$ and $Y$ are topological spaces, $X$ is compact and $Y$ is Hausdorff. Then any continuous function $f:X\rightarrow Y$ is a closed map. Moreover, if f is bijective, then it is a homeomorphism.
\end{lema}

\begin{teorema}
Let $A$ be an algebra of $\mathcal{V}$ such that $K(A)$ is closed under commutator operation. Then the following are equivalent:
\usecounter{nr}
\begin{list}{(\arabic{nr})}{\usecounter{nr}}
\item $A$ is hyperarchimedean;
\item For all distinct prime congruences $\phi,\psi$ of $A$ there exist $\alpha,\beta\in PCon(A)$ such that $\alpha\not\subseteq\phi$, $\beta\not\subseteq\psi$ and $[\alpha,\beta]\subseteq \rho(\Delta_A)$;
\item $Spec_Z(A)$ is a Hausdorff space;
\item $Spec_Z(A)$ is a Boolean space;
\item $\mathcal{Z}_A$ = $\mathcal{P}_A$;
\item $Spec_F(A)$ is a Hausdorff space;
\item $Spec_F(A)$ is a Boolean space;
\item $\mathcal{Z}_A$ = $\mathcal{F}_A$.
\end{list}
\end{teorema}

\begin{proof} We shall sketch a proof that illustrates how the reticulation can be used for transferring some results from rings to universal algebras. Let us fix a commutative ring $R$ such that the reticulations $L(A)$ and $L(R)$ of $A$ and $R$ are identical.

$(1)\Leftrightarrow (2)$ See Proposition 7.5.

$(1)\Leftrightarrow (3)$ By using Proposition 7.2, Remark 6.12, Theorem 3.3 of \cite{Aghajani} and Lemma 6.11, the following equivalences hold:

$A$ is hyperarchimedean iff $Spec(A) = Max(A)$

\hspace{3.5cm} iff $Spec(R) = Max(R)$

\hspace{3.5cm} iff  $Spec_Z(R)$ is Hausdorff

\hspace{3.5cm} iff $Spec_Z(A)$ is Hausdorff.

$(3)\Leftrightarrow (5)$ By using Lemma 6.11 and Theorem 3.3 of \cite{Aghajani}, the following equivalences hold: $Spec_Z(A)$ is a Hausforff space iff $Spec_Z(R)$ is a Hausforff space iff $\mathcal{Z_R}$ = $\mathcal{P_R}$ iff $\mathcal{Z_A}$ = $\mathcal{P_A}$.

The other equivalences of the theorem follow in a similar way, by using Theorem 3.3 of \cite{Aghajani}, Lemma 6.11 and Remark 6.12.

\end{proof}

\begin{remarca}
The previous result on the characterization of hyperarchimedean algebras extends Theorem 3.3 of \cite{Aghajani} that characterizes the zero - dimensional rings (= hyperarchimedean rings). Our proof of this result uses Theorem 3.3 of \cite{Aghajani} and some properties on the reticulations of algebras and rings (especially some consequences of the Hochster theorem \cite{Hochster}). Of course we can give a proof of Theorem 7.7 without using the reticulation.
\end{remarca}

Recall that the Jacobson radical of a commutative ring $R$ is the intersection of all maximal ideals of $R$. We shall generalize this notion to an algebra $A$ of $\mathcal{V}$. Then the Jacobson radical of $A$ is defined as $Rad(A) = \bigcap Max(A)$.

\begin{teorema}
Let $A$ be an algebra of $\mathcal{V}$ such that $K(A)$ is closed under commutator operation. Then $Max_F(A)$ is compact if and only if the quotient algebra $A/Rad(A)$ is hyperarchimedean.
\end{teorema}

\begin{proof}
Let us assume that  $A/Rad(A)$ is hyperarchimedean. We observe that the canonical morphism $\pi:A\rightarrow A/Rad(A)$ is admissible (cf. Lemma 3.6 of \cite{GM1}), therefore the map $\pi^*:Spec_F(A/Rad(A))\rightarrow Spec_F(A)$ is continuous (cf. Corollary 6.6). We remark that $Max_F(A)$ is the image of $Max_F(A/Rad(A))$ through the continuos map $\pi^*$ and $Max_F(A/Rad(A)) = Spec_F(A/Rad(A))$ (because $A/Rad(A)$ is hyperarchimedean). Since $Spec_F(A/Rad(A))$ is compact (by Lemma 6.1 it is a spectral space), it follows that $Max_F(A)$ is compact.

Now we assume that $Max_F(A)$ is compact. Firstly we shall prove that the quotient lattice $L(A)/(Rad(A))^{\ast}$ is a Boolean algebra. Let us consider $\alpha\in K(A)$ such that $\lambda_A(\alpha)/(Rad(A))^{\ast}\neq 0/(Rad(A))^{\ast}$, so $\lambda_A(\alpha)\notin (Rad(A))^{\ast}$. By Lemma 3.3(6) we get $\alpha\not\subseteq Rad(A)$, so there exists $\phi_{\alpha}\in Max(A)$ such that $\alpha\not\subseteq\phi_{\alpha}$. For any $\phi\in Max(A)$ we have $\alpha\subseteq\phi$ or $\alpha\not\subseteq\phi$. If $\alpha\not\subseteq\phi$ then there exists $\beta_{\phi}\in K(A)$ such that $\beta_{\phi}\subseteq\phi$ and $\alpha\lor \beta_{\phi} = \nabla_A$. But $\alpha\not\subseteq\phi_{\alpha}$ implies that the set $\{\phi\in Max(A)|\alpha\not\subseteq\phi\}$ is non - empty. We remark that $Max(A)$ =$\{\phi\in Max(A)|\alpha\subseteq\phi\}\bigcup\{\phi\in Max(A)|\alpha\not\subseteq\phi\}\subseteq V(\alpha)\bigcup (\bigcup_{\alpha\not\subseteq\phi}V(\beta_{\phi}))$.

Since $Max_F(A)$ is compact, there exist $\phi_1,\cdots,\phi_n\in Max(A)$ such that $\alpha\not\subseteq\phi_i$, for $i = 1,\cdots,n$ and $Max(A)\subseteq V(\alpha)\bigcup(\bigcup_{i=1} ^nV(\beta_{\phi_i}))$.

Let us denote $\beta_i = \beta_{\phi_i}$, for $i = 1,\cdots,n$ and $\gamma_1 = \beta_1$, $\gamma_2 = [\gamma_1,\beta_2],\cdots,\gamma_n = [\gamma_{n-1},\beta_n]$. It is easy to see that $\bigcup_{i=1} ^nV(\beta_{\phi_i}) = V(\gamma_n)$, therefore $Max(A)\subseteq V(\alpha)\bigcup V(\gamma_n) = V([\alpha,\gamma_n])$. It follows that $[\alpha,\gamma_n]\subseteq Rad(A)$, hence $\lambda_A([\alpha,\gamma_n])\in (Rad(A))^{\ast}$.

By using many times Lemma 2.4(2), from $\alpha\lor \beta_i = \nabla_A$, for $i = 1,\cdots,n$, we get by induction that $\alpha\lor \gamma_n = \nabla_A$. According to $\lambda_A(\alpha)\lor\lambda_A(\gamma_n) = \lambda_A(\alpha\lor\gamma_n) = 1$ and $\lambda_A(\alpha)\land\lambda_A(\gamma_n) = \lambda_A([\alpha,\gamma_n]) \subseteq \lambda_A(Rad(A))$ the following equalities hold in $L(A)/(Rad(A))^{\ast}$: $\lambda_A(\alpha)/((Rad(A))^{\ast}\lor\lambda_A(\gamma_n)/((Rad(A))^{\ast} = 1/((Rad(A))^{\ast}$ and $\lambda_A(\alpha)/((Rad(A))^{\ast}\land\lambda_A(\gamma_n)/((Rad(A))^{\ast} = 0/((Rad(A))^{\ast}$.

We have proven that $L(A)/(Rad(A))^{\ast}$ is a Boolean algebra. According to Proposition 7.6 of \cite{GM2}, the lattices $L(A/Rad(A))$ and $L(A)/(Rad(A))^{\ast}$  are isomorphic, therefore $L(A/Rad(A))$ is a Boolean algebra. By using Proposition 7.2 it follows that $A/Rad(A)$ is a hyperarchimedean algebra.

\end{proof}

\section{Congruence normal algebras}

\hspace{0.5cm} In this section we shall prove characterization theorems for two types of algebras, introduced in \cite{GKM}: the congruence normal algebras and the congruence $B$ - normal algebras. The first ones generalize the commutative Gelfand rings and the second one generalize the clean rings.

 The normal $mi$ - structures, introduced by Definition 3.1 of \cite{GeorgescuVoiculescu2}, are multiplicative lattices that generalize the following types of ordered algebras:

$\bullet$ the lattices of ideals in Gelfand rings (\cite{Johnstone}, p.199) and the lattices of ideals in normal lattices (\cite{Johnstone}, p.67);

$\bullet$ the lattices of filters in MV - algebras \cite{Cignoli1}, \cite{FG}, BL - algebras \cite{g}, Gelfand residuated lattices \cite{GCM},etc;

$\bullet$ the lattices of ideals in $F$ - rings (\cite{Johnstone}, p.217);

$\bullet$ the normal frames (\cite{Johnstone}, p.128);

$\bullet$ the normal quantales \cite{PasekaRN}.

Let us consider a semidegenerate congruence - modular variety $\mathcal{V}$ and $A$ an algebra in $\mathcal{V}$. We know that $Con(A)$ is a $mi$-structure in sense of \cite{GeorgescuVoiculescu2}. In general the commutator operation is not associative, so we cannot use the normal quantale theory \cite{PasekaRN} for studying the "normal objects" in $\mathcal{V}$. Applying Definition 3.1 of \cite{GeorgescuVoiculescu2} for the $mi$-structure $Con(A)$ we shall say the $A$ is a congruence normal algebra if for all $\theta,\nu\in Con(A)$ such that $\theta\lor\nu = \nabla_A$ there exist $\alpha,\beta\in Con(A)$ such that $\theta\lor\alpha = \nu\lor\beta = \nabla_A$ and $[\alpha,\beta] = \Delta_A$.

Recall from \cite{Johnstone}, \cite{Simmons} that a bounded distributive lattice $L$ is said to be normal if for all $a,b\in A$ such that $a\lor b = 1$ there exist $e,f\in L$ such that $a\lor e = b\lor f =  1$ and $e\land f = 0$. Remind that these normal lattices are the duals of "normal lattices" introduced by Cornish in \cite{Cornish}. We observe that the congruence normal lattices does not coincide with the normal lattices.

\begin{lema}
$A\in \mathcal{V}$ is a congruence normal algebra if and only if for all $\theta,\nu\in K(A)$ such that $\theta\lor\nu = \nabla_A$ there exist $\alpha,\beta\in K(A)$ such that $\theta\lor\alpha = \nu\lor\beta = \nabla_A$ and $[\alpha,\beta] = \Delta_A$.
\end{lema}

Now we fix an algebra $A\in \mathcal{V}$ such that $K(A)$ is closed under the commutator operations. The following two results show how the reticulation construction preserves the "normality" property.

\begin{lema}
If the algebra $A$ is congruence normal then the reticulation $L(A)$ is a normal lattice.
\end{lema}

\begin{proof}
Let $\alpha,\beta\in K(A)$ such that $\lambda_A(\alpha)\lor\lambda_A(\beta) = 1$, so $\lambda_A(\alpha\lor\beta) = 1$ (by Lemma 3.1(2)). According to Lemma 3.1(3) we have $\alpha\lor\beta = \nabla_A$. Since  the algebra $A$ is congruence normal there exist $\alpha_1,\beta_1\in K(A)$ such that $\alpha\lor\alpha_1 = \beta\lor\beta_1 = \nabla_A$ and $[\alpha_1,\beta_1] = \Delta_A$. According to Lemma 3.1,(1) and (2), it follows that $\lambda_A(\alpha)\lor\lambda_A(\alpha_1) = \lambda_A(\beta)\lor\lambda_A(\beta_1) = 1$ and $\lambda_A(\alpha_1)\land\lambda_A(\beta_1) = 0$. Thus $L(A)$ is a normal lattice.

\end{proof}

\begin{propozitie}
If the algebra $A$ satisfies the condition $(\star)$, then $A$ is congruence normal if and only if $L(A)$ is a normal lattice.
\end{propozitie}

\begin{proof}
Assume that $L(A)$ is a normal lattice. Let $\alpha,\beta\in K(A)$ such that $\alpha\lor\beta = \nabla_A$, hence $\lambda_A(\alpha)\lor\lambda_A(\beta)$ = $\lambda_A(\alpha\lor\beta) = 1$. Since $L(A)$ is normal there exist $\gamma,\delta\in K(A)$ such that $\lambda_A(\alpha)\lor\lambda_A(\gamma) = \lambda_A(\beta)\lor\lambda_A(\delta) = 1$ and $\lambda_A(\gamma)\land\lambda_A(\delta) = 0$, hence $\alpha\lor\gamma = \beta\lor\delta = \nabla_A$ (by Lemma 3.1,(3)) and $[[\gamma,\delta],[\gamma,\delta]]^n = \Delta_A$, for some integer $n\geq 1$ (by Lemma 3.1(4)). The algebra $A$ satisfies the condition $(\star)$, so by applying Lemma 2.10 one can find an integer $m\geq 1$ such that $[[\gamma,\gamma]^m,[\delta,\delta]^m] = \Delta_A$. By Lemma 2.4(3), from $\alpha\lor\gamma = \beta\lor\delta = \nabla_A$  we get $[\alpha,\alpha]^m\lor[\gamma,\gamma]^m = [\beta,\beta]^m\lor[\delta,\delta]^m = \nabla_A$, so $\alpha\lor [\gamma,\gamma]^m = \beta \lor[\delta,\delta]^m = \nabla_A$. Taking into account that $K(A)$ is closed under commutator operation, we have $[\gamma,\gamma]^m,[\delta,\delta]^m\in K(A)$, so we conclude that $A$ is congruence normal (by Lemma 8.1). The converse implication follows from the previous lemma.

\end{proof}

\begin{propozitie}
\cite{GeorgescuVoiculescu2} Suppose that $A$ is a congruence normal algebra, $\theta\in Con(A)$ and $\phi\in Max(A)$. The following hold:
\usecounter{nr}
\begin{list}{(\arabic{nr})}{\usecounter{nr}}
\item $\phi$ is the unique maximal congruence with $O(\phi)\subseteq \phi$;
\item $Ker(\phi)\subseteq \theta$ if and only if $\theta = \phi$ or $\theta = \nabla_A$;
\item $Ker(\theta)\subseteq\phi$ if and only if $\theta\subseteq\phi$;
\item $Vir(\theta) = Ker(\theta)$;
\item $Vir(\theta)\subseteq\phi$ if and only if $\theta\subseteq\phi$.
\end{list}
\end{propozitie}

\begin{propozitie}
\cite{GeorgescuVoiculescu2} The following are equivalent:
\usecounter{nr}
\begin{list}{(\arabic{nr})}{\usecounter{nr}}
\item $A$ is congruence normal;
\item $Vir:Con(A)\rightarrow VCon(A)$ preserves arbitrary joins;
\item For all $\theta,\nu\in Con(A)$, $\theta\lor\nu = \nabla_A$ implies $Vir(\theta)\lor Vir(\nu) = \nabla_A$.
\end{list}
\end{propozitie}

\begin{propozitie}
 Consider the following properties:
\usecounter{nr}
\begin{list}{(\arabic{nr})}{\usecounter{nr}}
\item $A$ is congruence normal;
\item If $\phi,\psi$ are two distinct maximal congruences of $A$ then there exist $\alpha,\beta\in K(A)$ such that $\alpha\not\subseteq\phi$, $\beta\not\subseteq\psi$ and $[\alpha,\beta] = \Delta_A$;
\item The inclusion $Max_Z(A)\subseteq Spec_Z(A)$ is a Hausdorff embedding (i.e.any pair of distinct points in $Max_Z(A)$ have disjoint neighbourhoods in $Spec_Z(A)$);
\item For any $\phi\in Spec(A)$, there exists a unique $\psi\in Max(A)$, such that $\phi\subseteq\psi$.

Then $(1)\Rightarrow (2)\Rightarrow (3)\Rightarrow (4)$. Moreover, if the algebra $A$ fulfills the condition $(\star)$, then $(4)\Rightarrow (1)$.
\end{list}
\end{propozitie}

\begin{proof}
By applying Proposition 3.2 of \cite{GeorgescuVoiculescu2} we obtain $(1)\Rightarrow (2)\Rightarrow (3)\Rightarrow (4)$.

Assume that $A$ satisfies the condition $(\star)$ and any prime congruence of $A$ is embedded in a unique maximal congruence of $A$. We know from Section 3 that $(Spec(A),\subseteq)$ and $Spec_{Id}(L(A)),\subseteq)$ are order - isomorphic, so any prime ideal of the lattice $L(A)$ can be embedded in a unique maximal ideal. According to Proposition 3.7  of \cite{Johnstone}, p.68, $L(A)$ is a normal lattice. Since the algebra $A$ fulfills the condition $(\star)$, one can apply Proposition 8.3, hence $A$ is congruence normal.

\end{proof}

Let $R$ be a fixed commutative ring such that $L(A) = L(R)$ (see the discussion from Section 6).

\begin{lema}
If the algebra $A$ verifies $\star$, then $A$ is congruence normal if and only if $R$ is a Gelfand ring.
\end{lema}

\begin{proof}
According to Proposition 3.7 of \cite{Johnstone}, p.198 and Proposition 8.3, the following are equivalent: $A$ is congruence normal iff $L(A) = L(R)$ is a normal lattice iff $R$ is a Gelfand ring.

\end{proof}

For any congruence $\theta$ in $A$ we denote $Rad(\theta) = \bigcap(Max(A)\bigcap V(\theta))$; in particular, $Rad(0) = Rad(A)$. If $R$ is a commutative ring and $P$ is a prime ideal of $R$, then we denote $\Lambda(P) = \{Q\in Spec(R)|Q\subseteq P\}$.

Theorem 4.3 of \cite{Aghajani} contains several old and new properties that characterize the Gelfand rings. The following result is a generalization of the main part of this theorem.

\begin{propozitie}
 If the algebra $A$ verifies the condition $(\star)$, then the following are equivalent:
\usecounter{nr}
\begin{list}{(\arabic{nr})}{\usecounter{nr}}
\item $A$ is a congruence normal algebra;
\item $Spec_Z(A)$ is a normal space;
\item The inclusion $Max_Z(A)\subseteq Spec_Z(A)$ has a continuous retraction $\gamma:Spec_Z(A)\rightarrow Max_Z(A)$;
\item If $\phi\in Max(A)$, then $\Lambda(\phi)$ is a closed subset of $Spec_Z(A)$;
\item If $\phi,\psi$ are two distinct maximal congruences of $A$ then $Vir(\phi)\lor Vir(\psi) = \nabla_A$;
\item For all $\theta\in Con(A)$ and $\phi\in Max(A)$, $Vir(\theta)\subseteq\phi$ implies $\theta\subseteq\phi$;
\item For all $\theta\in Con(A)$, $Max(A)\bigcap V(\theta) = Max(A)\bigcap V(Vir(\theta))$;
\item For all $\theta\in Con(A)$, $Rad(\theta) = Rad(Vir(\theta))$;
\item The function $\eta:Max_Z(A)\rightarrow Spec_Z(VCon(A))$, defined by the assignment $\phi\mapsto Vir(\phi)$ is a homeomorphism.
\end{list}
\end{propozitie}

\begin{proof}

 $(1)\Leftrightarrow (2)$ By \cite{Cornish} or \cite{Johnstone}, a bounded distributive lattice $L$ is normal if and only if $Spec_{Id,Z}(L)$ is a normal space. Therefore, by using Propositions 8.3 and 4.3, the following equivalences hold: $A$ is congruence normal iff $L(A)$ is a normal lattice iff $Spec_{Id,Z}(L(A))$ is a normal space iff $Spec_Z(A)$ is a normal space.

$(1)\Leftrightarrow (3)$  Propositions 8.3. and 3.4 together with Proposition 3.7  of \cite{Johnstone}, p.68 lead to the following equivalences: $A$ is congruence normal iff $L(A)$ is a normal lattice iff the inclusion $Max_Z(L(A))\subseteq Spec_Z(L(A))$ has a continuous retraction $Spec_{Id,Z}(L(A))\rightarrow Max_{Id,Z}(L(A))$ iff the inclusion $Max_Z(A)\subseteq Spec_Z(A)$ has a continuous retraction $\gamma:Spec_Z(A)\rightarrow Max_Z(A)$.

$(1)\Leftrightarrow (4)$ Recall from Section 6 that there exists an order - isomorphism $s:Spec(A)\rightarrow Spec(R)$, hence for any $\phi\in Spec(A)$ the restriction $s|_{\Lambda(\phi)}:\Lambda(\phi)\rightarrow \Lambda(s(\phi))$ is an order - isomorphism. By Lemma 6.11, the map $s$ is continuos w.r.t. the Zariski topology. Therefore, by using Theorem 4.3 of \cite{Aghajani}, Lemma 8.7 and Proposition 4.3, the following are equivalent: $A$ is congruence normal iff $R$ is a Gelfand ring iff for all $P\in Spec(R)$, $\Lambda(P)$ is a closed subset of $Spec_Z(R)$ iff for all $\phi\in Max(A)$, $\Lambda(\phi)$ is a closed subset of $Spec_Z(A)$.

$(1)\Rightarrow(5)$. By Proposition 8.5(3).

$(5)\Rightarrow(1)$ Assume that $\phi$ and $\psi$ are two distinct maximal congruences of $A$, hence $Vir(\phi)\lor Vir(\psi) = \nabla_A$. Therefore there exist two compact congruences $\alpha,\beta$ such that $\alpha\subseteq Vir(\phi)\subseteq\phi$, $\beta\subseteq Vir(\psi)\subseteq\psi$ and $\alpha\lor\beta = \nabla_A$. From $\alpha\subseteq Vir(\phi)$ we get $Vir(\phi)\lor \alpha^{\perp} = \nabla_A$, hence there exist $\gamma,\delta\in K(A)$ such that $\gamma\subseteq\phi$, $\delta\subseteq \alpha^{\perp}$ and $\gamma\lor\delta = \nabla_A$. It follows that $\alpha\not\subseteq \psi$, $\delta\not\subseteq\phi$ and $[\alpha,\delta] = \Delta_A$. In accordance with Proposition 8.6, we conclude that $A$ is congruence normal.

$(1)\Rightarrow(6)$ By Proposition 8.4,(5).

$(6)\Leftrightarrow(7)\Leftrightarrow(8)$. Straightforward.

$(1)\Rightarrow(9)$ By Theorem 3.5 of \cite{GeorgescuVoiculescu2}.

$(9)\Rightarrow(3)$ By Lemma 4.8(3), $Vir: Spec_Z(A)\rightarrow Spec_Z(VCon(A))$ is a continuous map and, by the hypothesis (9), the map $\eta:Max_Z(A)\rightarrow Spec_Z(VCon(A))$, defined by the assignment $\phi\mapsto Vir(\phi)$ is a homeomorphism. Then the map $\gamma = \eta^{-1}\circ Vir:Max_Z(A)\rightarrow Spec_Z(A)$ is a continuous retraction of the inclusion $Max_Z(A)\subseteq Spec_Z(A)$.

\end{proof}

The proof of the previous proposition illustrates how by means of the reticulation the properties of normal lattices (resp. of Gelfand rings) can be exported to congruence normal algebras.

Recall from \cite{Cignoli} that a bounded distributive lattice $L$ is said to be $B$ - normal if for all $a,b\in L$ such that $a\lor b = 1$ there exist $e,f\in B(L)$ such that $a\lor e = b\lor f = 1$ and $e\land f = 0$. Dually, a bounded distributive lattice $L$ is said to be a $B$ - conormal lattice if for all $a,b\in L$ such that $a\land b = 0$ there exist $e,f\in B(L)$ such that $a\land e = b\land f = 0$ and $e\lor f = 1$. Of course, any $B$ - normal lattice is normal.

Following \cite{GKM}, an algebra $A\in \mathcal{V}$ is said to be congruence $B$ - normal if for all $\phi,\psi\in Con(A)$, such that $\phi\lor\psi = \nabla_A$, there exist $\chi,\varepsilon\in B(Con(A))$ (hence $\chi,\varepsilon\in K(A)$, according to Lemma 2.6(3)) such that $\phi\lor\chi = \psi\lor\varepsilon = \nabla_A$ and $[\chi,\varepsilon] = \Delta_A$. Of course, any congruence $B$ - normal algebra is congruence normal. By Lemma 5.32 of \cite{GKM}, an algebra $A$ is congruence $B$ - normal if and only if for all $\alpha,\beta\in K(A)$, such that $\alpha\lor\psi = \nabla_A$, there exist $\gamma,\delta\in B(Con(A))$ such that $\alpha\lor\gamma = \beta\lor\delta = \nabla_A$ and $[\gamma,\delta]= \Delta_A$.

\begin{propozitie}

\usecounter{nr}
\begin{list}{(\arabic{nr})}{\usecounter{nr}}
\item If $A$ is a congruence $B$ - normal algebra, then $L(A)$ is a $B$ - normal lattice;
\item If $A$ verifies the condition $(\star)$, then $A$ is a congruence $B$ - normal algebra if and only if $L(A)$ is a $B$ - normal lattice.
\end{list}
\end{propozitie}

\begin{proof}
The property (1) follows from Proposition 5.33(i) of \cite{GKM}. According to Corollary 3.7, the condition $(\star)$ implies that the function $\lambda_A|_{B(Con(A))}:B(Con(A))\rightarrow B(L(A))$ is a Boolean isomorphism. Then the property (2) follows by applying Proposition 5.33(ii) of \cite{GKM}.
\end{proof}

Let an algebra $A\in \mathcal{V}$ and $\theta$ a congruence of $A$. Let us consider the function $p^{\bullet}_{\theta}:Con(A)\rightarrow Con(A/{\theta})$ defined by $p^{\bullet}_{\theta}(\alpha) = ((\alpha\lor\theta)/{\theta})$, for any $\alpha\in Con(A)$. According to \cite{GKM}, the restriction $p^{\bullet}_{\theta}|_{B(Con(A))}:B(Con(A))\rightarrow B(Con(A/{\theta}))$ is a Boolean morphism.

Following Definition 5.18 of \cite{GKM}, we say that the congruence $\theta\in Con(A)$ fulfills the Congruence Boolean Lifting Property (abbreviated $CBLP$) if the map $p^{\bullet}_{\theta}|_{B(Con(A))}:B(Con(A))\rightarrow B(Con(A/{\theta}))$ is surjective. The algebra $A$ fulfills $CBLP$ if all congruences of $A$ satisfy $CBLP$. For any commutative ring $R$, the following equivalences holds: $R$ fulfills $CBLP$ iff $R$ satisfies the lifting idempotent property (idempotents can be lifted modulo every ideal of $R$) iff $R$ is a clean ring iff $R$ is an exchange ring (cf. \cite{a}).

By Definition 5.23 of \cite{Cheptea}, a bounded distributive lattice $L$ satisfies $Id-CBLP$ if for any ideal $I$ of $L$, the canonical Boolean morphism $B(L)\rightarrow B(L/I)$ is surjective.

\begin{propozitie}
Assume that the algebra $A$ verifies the condition $(\star)$. Then $A$ fulfills $CBLP$ if and only if the lattice $L(A)$ fulfills $Id-CBLP$.
\end{propozitie}

\begin{proof}
Since the algebra $A$ verifies the condition $(\star)$, we know that the map $\lambda_A|_{B(Con(A))}: B(Con(A))\rightarrow B(L(A))$ is a Boolean isomorphism (see Corollary 3.7). Then the desired equivalence follows from Theorem 5.34 of \cite{GKM}.

\end{proof}

\begin{teorema}
Assume that the algebra $A$ verifies the condition $(\star)$. Then the following are equivalent:
\usecounter{nr}
\begin{list}{(\arabic{nr})}{\usecounter{nr}}
\item  $A$ fulfills $CBLP$:
\item  $A$ is congruence $B$ - normal.
\end{list}
\end{teorema}

\begin{proof}
Similarly, by using Theorem 5.34 of \cite{GKM}.

\end{proof}

\begin{lema}
Assume that the algebra $A$ verifies the condition $(\star)$. Then $U\subseteq Spec(A)$ is a clopen subset of $Spec_Z(A)$ if and only if $U = V(\gamma)$, for some $\gamma\in Con(B(A))$.
\end{lema}

\begin{proof}
Assume that $U\subseteq Spec(A)$ is a clopen subset of $Spec_Z(A)$. Thus there exist $\theta,\chi\in Con(A)$ such that $U = V(\theta)$ and $V(\theta\lor\chi) = V(\theta)\bigcap V(\chi) = \emptyset$ and $V([\theta,\chi]) = V(\theta)\bigcup V(\chi) = Spec(A)$, hence $\theta\lor\chi = \nabla_A$ and $[\theta,\chi]\subseteq\rho(\Delta_A)$. By the compacity of $\nabla_A$, there exist $\alpha,\beta\in K(A)$ such that $\alpha\subseteq\theta$, $\beta\subseteq\chi$ and $\alpha\lor\beta = \nabla_A$. By applying Lemma 2.10, there exists an integer $m\geq 0$ such that $[[\alpha,\alpha]^m,[\beta,\beta]^m] = \Delta_A$, therefore $V([\alpha,\alpha]^m)\bigcup V([\beta,\beta]^m)$ = $V([[\alpha,\alpha]^m,[\beta,\beta]^m]) = V(\Delta_A) = Spec(A)$. By Lemma 2.4(3), from $\alpha\lor\beta = \nabla_A$ we get $[\alpha,\alpha]^m\lor [\beta,\beta]^m = \nabla_A$, hence $[\alpha,\alpha]^m\in B(Con(A))$ (by Lemma 2.11(1)) and $V([\alpha,\alpha]^m)\bigcap V([\beta,\beta]^m)$ = $V([\alpha,\alpha]^m\lor [\beta,\beta]^m)$ = $V(\nabla_A) = \emptyset$. By taking into account that $U = V(\theta)\subseteq V(\alpha) = V([\alpha,\alpha]^m)$ and $V(\chi)\subseteq V(\beta) = V([\beta,\beta]^m])$, it follows that $U = V([\beta,\beta]^m)$. By denoting $\gamma = [\alpha,\alpha]^m$, we get $U = V(\gamma)$ and $\gamma\in B(Con(A))$. The converse implication is straightforward.
\end{proof}

\begin{lema}
Let $\alpha,\beta$ two compact congruences of $A$. Then $D(\alpha)\subseteq D(\beta)$ if and only if there exists an integer $n\geq 1$ such that $[\alpha,\alpha]^n\subseteq \beta$. If $\alpha\in B(Con(A))$ then $D(\alpha)\subseteq D(\beta)$ if and only if $\alpha\subseteq D(\beta)$.
\end{lema}

\begin{proof}
According to Lemma 4.3 of \cite{GM2}, $D(\alpha)\subseteq D(\beta)$ iff $V(\beta)\subseteq V(\alpha)$ if and only if $\alpha\subseteq \rho(\beta)$ iff there exists an integer $n\geq 0$ such that $[\alpha,\alpha]^n\subseteq\beta$. If $\alpha\in B(Con(A))$ then $[\alpha,\alpha]^n = \alpha$, therefore, by the first part of the lemma, it follows that $D(\alpha)\subseteq D(\beta)$ if and only if $\alpha\subseteq D(\beta)$.

\end{proof}

We recall that a topological space $X$ is said to be:

$\bullet$ zero - dimensional if X has a basis of open sets;

$\bullet$ strongly zero - dimensional if, given any closed set $T$ and any open set $V$ such that $T\subseteq V$, there exists a clopen subset $U$ of $X$ such that $T\subseteq U\subseteq V$ (according to \cite{lu}).

\begin{lema}\cite{lu}
Assume that $X$ is a topological space. Then the following are equivalent:
\usecounter{nr}
\begin{list}{(\arabic{nr})}{\usecounter{nr}}
\item  $X$ is strongly zero - dimensional;
\item  Any two disjoint closed subsets of $X$ can be separated through clopen subsets of $X$;
\item  If there exist two open subsets $U,V$ of $X$ such that $X = U\bigcup V$, then there exist two open clopen subsets $C,D$ of $X$ such $C\subseteq U$, $D\subseteq V$, $C\bigcup D = X$ and $C\bigcap D = \emptyset$.
\end{list}
\end{lema}

\begin{corolar}
The topological space $Max_Z(A)$ is zero - dimensional iff it is strongly zero - dimensional iff it is normal iff it is Boolean.
\end{corolar}

\begin{proof}
Similar to the proof of Corollary 6.6 from \cite{GCM}.
\end{proof}

\begin{teorema}
\cite{GKM} If $A$ is an algebra satisfying the condition $(\star)$, then the following are equivalent:
\usecounter{nr}
\begin{list}{(\arabic{nr})}{\usecounter{nr}}
\item $A$ fulfills $CBLP$;
\item $L(A)$ fulfills $Id-CBLP$;
\item $L(A)$ is a B-normal lattice;
\item $A$ is a congruence $B$ - normal algebra;
\item $Spec_Z(A)$ is a strongly zero - dimensional space.
\end{list}
\end{teorema}

\begin{proof}
Since the algebra $A$ verifies the condition $(\star)$, we know that the map $\lambda_A|_{B(Con(A))}: B(Con(A))\rightarrow B(L(A))$ is a Boolean isomorphism (see Corollary 3.7), so one can apply Theorem 5.34 of \cite{GKM}.

The equivalence of the properties (1) - (4) was established in \cite{GKM}, but not the equivalence of (5) with the first four conditions. Then we shall give here a proof for the equivalence of (4) and (5).

$(4)\Rightarrow(5)$ Assume that $A$ is a congruence $B$ - normal algebra. Let $\theta,\chi$ be two congruence of $A$ such that $D(\theta)\bigcup D(\chi) = Spec(A)$, therefore $\theta\lor\chi = \nabla_A$ (by Proposition 3.15 of \cite{GM2}). By hypothesis (4), there exist $\alpha,\beta\in B(Con(A))$ such that $\theta\lor\alpha = \chi\lor\beta = \nabla_A$ and $[\alpha,\beta] = \Delta_A$. Then $ [\theta,\beta] = [\theta,\beta]\lor [\alpha,\beta]  = [\theta\lor\alpha,\beta] = [\nabla_A,\beta] = \beta$, hence $\beta \subseteq \theta$; in a similar way, one can prove that $\neg\beta\subseteq \chi$. Thus $D(\beta)$ and $D(\neg\beta)$ are clopen subsets of $Spec_Z(A)$ (by Lemma 8.12) such that $D(\beta)\subseteq D(\theta)$, $D(\neg\beta)\subseteq D(\chi)$, $D(\beta)\bigcap D(\neg\beta) = \emptyset$ and $D(\beta)\bigcup D(\neg\beta) = Spec(A)$. According to Lemma 8.14(3), $Spec_Z(A)$ is a strongly zero - dimensional space.

$(5)\Rightarrow(4)$ Assume that $Spec_Z(A)$ is a strongly zero - dimensional space. Let $\theta,\chi\in Con(A)$ such that $\theta\lor\chi = \nabla_A$, therefore $D(\theta)\bigcup D(\chi) = Spec(A)$. By Lemmas 8.12 and 8.14(3), there exist $\alpha,\beta\in B(Con(A))$ such that $D(\alpha)\subseteq D(\theta)$, $D(\beta)\subseteq D(\chi)$, $D(\alpha)\bigcup D(\beta) = Spec(A)$ and $D([\alpha,\beta]) = D(\alpha)\bigcap D(\beta) = \emptyset$. By using Lemma 8.13 we get $\alpha\subseteq\theta$ and $\beta\subseteq \chi$. According to Propositions 3.14 and 3.15 of \cite{GM2}, $D(\alpha)\bigcup D(\chi) = Spec(A)$ implies $\alpha\lor\chi = \nabla_A$. Similarly, we have $\theta\lor\beta = \nabla_A$. On the other hand, $D([\alpha,\beta])= \emptyset$ implies $[\alpha,\beta] \subseteq\rho(\Delta_A)$, so there exists an integer $n\geq 1$ such that $[\alpha,\beta]^n = \Delta_A$. Since $\alpha,\beta\in B(Con(A))$, we get $ [\alpha,\beta] =  [\alpha,\beta]^n = \Delta_A$. It follows immediately that $\theta\lor\beta = \chi\lor\alpha = \nabla_A$, so $A$ is a congruence $B$ - normal algebra.

\end{proof}

\begin{lema}
 Consider the following conditions:
\usecounter{nr}
\begin{list}{(\arabic{nr})}{\usecounter{nr}}
\item If $\phi$ and $\psi$ are distinct maximal congruences in $A$, then there exists $\alpha\in B(Con(A))$ such that $\alpha\subseteq\phi$ and $\neg\alpha\subseteq\psi$;
\item If $P$ and $Q$ are distinct maximal ideals in $L(A)$, then there exists $e\in B(L(A))$ such that $e\in P$ and $\neg e \in Q$.

Then $(1)\Rightarrow (2)$; if $A$ is an algebra satisfying the condition $(\star)$, then $(2)\Rightarrow (1)$.
\end{list}
\end{lema}

\begin{proof}
Assume that the condition (1) holds. Let $P, Q$ be two distinct maximal ideals of $L(A)$, then $P = \phi^{\ast}$ and $Q = \psi^{\ast}$ for some distinct maximal congruences $\phi$ and $\psi$ of $A$. Then  there exists $\alpha\in B(Con(A))$ such that $\alpha\subseteq\phi$ and $\neg\alpha\subseteq\psi$, hence $\lambda_A(\alpha)\in B(L(A))$, $\lambda_A(\alpha)\in \phi^{\ast} = P$ and $\neg\lambda_A(\alpha) = \lambda_A(\neg\alpha)\in \psi^{\ast} = Q$, therefore the property $(2)$ holds.

Conversely, assume that $A$  satisfies the condition $(\star)$ and (2) fulfills. Let $\phi$ and $\psi$ be two distinct maximal congruences in $A$, hence $\phi^{\ast}$ and $\psi^{\ast}$ are distinct maximal ideals in $L(A)$, therefore there exists $e\in B(L(A))$ such that $e\in \phi^{\ast}$ and $\neg e \in \psi^{\ast}$. Thus there exist $\beta,\gamma\in K(A)$ such that $\lambda_A(\beta) = e$, $\lambda_A(\gamma) = \neg e$, $\beta\subseteq\phi$ and $\gamma\subseteq\psi$. Since $A$  satisfies the condition $(\star)$, there exists $\alpha\in B(Con(A))$ such that $e = \lambda_A(\alpha)$ (cf. Corollary 3.7). From $\lambda_A(\alpha) = e = \lambda_A(\beta)$ and $\alpha\in B(Con(A))$ we infer that there exists an integer $n\geq 0$ such that $\alpha = [\alpha,\alpha]^n\subseteq\beta\subseteq\phi$. Similarly, from $\lambda_A(\neg\alpha) = \neg e = \lambda_A(\gamma)$ we obtain $\neg \alpha\subseteq\psi$.

\end{proof}

\begin{corolar}
 Let $R$ be a commutative ring such that the reticulations $L(A)$ and $L(R)$ of $A$ and $R$ are identical. If the algebra $A$ satisfies the condition $(\star)$, then the following are equivalent:
\usecounter{nr}
\begin{list}{(\arabic{nr})}{\usecounter{nr}}
\item If $\phi$ and $\psi$ are distinct maximal congruences in $A$, then there exists $\alpha\in B(Con(A))$ such that $\alpha\subseteq\phi$ and $\neg\alpha\subseteq\psi$;
\item If $P$ and $Q$ are distinct maximal ideals in $R$, then there exists an idempotent $R$ such that $e\in P$ and $\neg e \in Q$.
\end{list}
\end{corolar}

\begin{proof}
By  applying the previous lemma twice.
\end{proof}

Theorem 4.3 of \cite{Aghajani} collects many old and new properties that characterize the clean rings. The following result extends to algebras the main part of this characterization theorem of clean rings.

\begin{teorema}
 If $A\in \mathcal{V}$ satisfies $(\star)$ then the following are equivalent:
\usecounter{nr}
\begin{list}{(\arabic{nr})}{\usecounter{nr}}
\item $A$ is a congruence $B$ - normal algebra;
\item If $\phi$ and $\psi$ are distinct maximal congruences in $A$ then there exists $\alpha\in B(Con(A))$ such that $\alpha\subseteq\phi$ and $\neg\alpha\subseteq\psi$;
\item $A$ is congruence normal and $Max_Z(A)$ is a zero - dimensional space;
\item $A$ is congruence normal and $Max_Z(A)$ is a Boolean space;
\item The family $(D(\alpha)\bigcap Max(A))_{\alpha\in B(Con(A))}$ is a basis of open sets for $Max_Z(A)$;
\item The map $s_A|_{Max(A)}:Max_Z(A)\rightarrow Sp(A)$ is a homeomorphism;
\item $A$ is congruence normal and $Max_Z(A)$ is a totally disconnected space;
\item If $\phi$ and $\psi$ are distinct maximal congruences in $A$ then there exists $\alpha\in B(Con(A))$ such that $\alpha\subseteq O(\phi)$ and $\neg\alpha\subseteq O(\psi)$;
\end{list}
\end{teorema}

\begin{proof}

$(1)\Leftrightarrow(2)$
By Hochster theorem \cite{Hochster}, there exists a commutative ring $R$ such that  the reticulations $L(A)$ and $L(R)$ of $A$ and $R$ are identical. By using Theorem 8.16 and Corollary 8.18 as well as Theorem 4.3 of \cite{Aghajani} the following properties are equivalent:

${\bullet}$ $A$ is a congruence $B$ - normal algebra;

${\bullet}$ $L(A)$ is a $B$ - normal lattice;

${\bullet}$ $R$ is a clean ring;

${\bullet}$ If $P$ and $Q$ are distinct maximal ideals in $R$, then there exists an idempotent element $e$ of $R$ such that $e\in P$ and $\neg e \in Q$;

${\bullet}$ If $\phi$ and $\psi$ are distinct maximal congruences in $A$, then there exists $\alpha\in B(Con(A))$ such that $\alpha\subseteq\phi$ and $\neg\alpha\subseteq\psi$.

$(1)\Leftrightarrow(3)$
Let $R$ be a commutative ring such that the reticulations $L(A)$ and $L(R)$ are identical (by Hochster theorem \cite{Hochster}) and $s:Spec(A)\rightarrow Spec(R)$ the order - isomorphism (w.r.t. inclusion) constructed in Section 6. We know that $s:Spec_Z(A)\rightarrow Spec_Z(R)$ and $s = s|_{Max(A)}:Max(A)\rightarrow Max(R)$ are homeomorphisms (by Lemma 6.11 and Remark 6.12). Thus $Max_Z(A)$ is zero - dimensional if and only if $Max_Z(R)$ is zero - dimensional. On the other hand, by applying Proposition 8.3 twice, it follows that $A$ is congruence normal if and only if $R$ is a Gelfand ring. Recall from Proposition 3.9 of \cite{Johnstone}, p. 201 or Theorem 4.3 of \cite{Aghajani} that the commutative ring $R$ is a clean ring if and only if $R$ is a Gelfand ring and $Max_Z(R)$ is zero - dimensional.

According to the previous remarks the following properties are equivalent:

${\bullet}$ $A$ is a congruence $B$ - normal algebra;

${\bullet}$ $R$ is a clean ring;

${\bullet}$ $R$ is a Gelfand ring and $Max_Z(R)$ is zero - dimensional;

${\bullet}$ $A$ is a congruence normal algebra and $Max_Z(A)$ is zero - dimensional.

$(3)\Leftrightarrow(4)$ By Corollary 8.15.

$(4)\Rightarrow(5)$
 By hypothesis, $Max_Z(A)$ is a Boolean space, so its family of clopen sets is a basis of open sets. According to Lemma 8.12, the family $(D(\alpha)\bigcap
Max(A))_{\alpha\in B(Con(A))}$ is exactly this basis of open sets for $Max_Z(A)$.

$(5)\Rightarrow(2)$
Let $\phi,\psi$ be two distinct maximal congruences of $A$. Thus $U = Spec(A) - \{\psi\} = Spec(A) - V(\psi) = D(\psi)$ is an open subset of $Spec_Z(A)$, hence $U\bigcap Max(A)$ is an open subset of $Max(A)$ that contains $\phi$. By the hypothesis (5), there exists $\alpha\in B(Con(A))$ such that $\phi\in D(\alpha)\bigcap Max(A)\subseteq U\bigcap Max(A)$. It is clear that $\alpha\not\subseteq\phi$, so $\neg\alpha\subseteq\phi$, because $\phi$ is a prime congruence. Since $\psi\notin U$ we get $\psi\notin D(\alpha)$, so $\alpha\subseteq\psi$. Then the property $(2)$ fulfills.

$(1)\Rightarrow(6)$
Assume that $A$ is a congruence $B$ - normal algebra. By Proposition 6.15, the map $s_A: Spec_Z(A)\rightarrow Sp(A)$ is surjective and continuous, so its restriction $s_A|_{Max(A)}: Max_Z(A)\rightarrow Sp(A)$ is also continuous.

Firstly, we shall prove that $s_A|_{Max(A)}: Max_Z(A)\rightarrow Sp(A)$ is injective. Let us consider two distinct points $\phi,\psi$ of $Max(A)$. We know that the hypothesis (1) implies the property (2), so there exists $\alpha\in B(Con(A))$ such that $\alpha\subseteq\phi$ and $\neg\alpha\subseteq\psi$. Thus $\alpha\subseteq s_A(\phi)$ and $\alpha\not\subseteq s_A(\psi)$, i.e. $s_A(\phi)\neq s_A(\psi)$.

Secondly, we shall prove that $s_A|_{Max(A)}: Max_Z(A)\rightarrow Sp(A)$ is surjective. Assume that $\psi$ is an arbitrary point of $Sp(A)$, so $\psi = s_A(\phi)$, for some $\phi\in Spec(A)$ (because $s_A: Spec_Z(A)\rightarrow Sp(A)$ is surjective). The algebra $A$ is congruence normal, so there exists a unique $\gamma(\phi)\in Max(A)$ such that $\phi\subseteq\gamma(\phi)$ (cf. Proposition 8.6(4)). Thus $\psi = s_A(\phi)\subseteq s_A(\gamma(\phi))$, so $\psi = s_A(\gamma(\phi))$, because $\psi$ and $s_A(\gamma(\phi))$ are max - regular congruences of $A$. It follows that $s_A|_{Max(A)}$ is surjective.

We know that $Max(A)$ and $Sp(A)$ are  Boolean spaces (by (4) and Proposition 6.15), therefore, by using Lemma 7.6, it follows that the map $s_A|_{Max(A)}: Max_Z(A)\rightarrow Sp(A)$ is a homeomorphism.

$(6)\Rightarrow(4)$
Assume that the map $s_A|_{Max(A)}:Max_Z(A)\rightarrow Sp(A)$ is a homeomorphism. We know from Proposition 6.5 that $s_A: Spec_Z(A)\rightarrow Sp(A)$ is a surjective continuous map. Thus $(s_A|_{Max(A)})^{-1}\circ s_A:Spec_Z(A)\rightarrow Max_Z(A)$ is a continuous retraction of the inclusion $Max_Z(A)\subseteq Spec_Z(A)$, so $A$ is congruence normal (by Proposition 8.8(3)). But $Sp(A)$ is a Boolean space (by Proposition 6.15), so the hypothesis (6) implies that $Max_Z(A)$ is also a Boolean space.

$(4)\Leftrightarrow(7)$
We apply Theorem 4.2 of \cite{Johnstone}, p.42 to the topological space $Max_Z(A)$.

$(2)\Rightarrow(8)$
Let $\phi\in Max(A)$ and $\alpha\in B(Con(A))$ such that $\alpha\subseteq \phi$. Thus $\neg \alpha\not\subseteq\phi$, so we get $\phi\lor\alpha^{\perp} = \phi\lor\neg \alpha = \nabla_A$. According to Lemma 4.12 we obtain $\alpha\subseteq Ker(\phi)$. Since $\phi\in Max(A)$, by Lemma 4.8 we have $\O(\phi) = Ker(\phi)$, hence $\alpha\subseteq O(\phi)$. We have proven that $\alpha\in B(Con(A))$ and $\alpha\subseteq \phi$ implies $\alpha\subseteq O(\phi)$.

Assume that $\phi$ and $\psi$ are distinct maximal congruences in $A$. According to $(2)$, there exists $\alpha\in B(Con(A))$ such that $\alpha\subseteq\phi$ and $\neg\alpha\subseteq\psi$. By using the previous observation, it follows that  $\alpha\subseteq O(\phi)$ and $\neg\alpha\subseteq O(\psi)$, so (2) implies (8).

$(8)\Rightarrow(2)$
We know that $O(\theta)\subseteq\theta$, for all congruence $\theta$ of $A$, so the desired implication is immediate.

\end{proof}

\section{Mp-algebras}

\hspace{0.5cm} Let us consider a semidegenerate congruence - modular variety $\mathcal{V}$ and $A$ an algebra in $\mathcal{V}$. The minimal prime spectrum of $A$ is the set $Min(A)$ of minimal prime congruences in $A$. For any $\phi\in Spec(A)$ there exists a minimal prime congruence $\psi$ such that $\psi\subseteq\phi$. If $R$ is a commutative ring that there exists a bijection between the minimal prime congruences and minimal prime ideals of $R$, so we shall use the notation $Min(R)$ for each of these two (identical) sets. If $L$ is a bounded distributive lattice then $Min_{Id}(L)$ will denote the set of minimal prime ideals of $L$.

In what follows we shall denote:

$\bullet$ $Min_Z(A)$ is $Min(A)$ endowed with the Zariski topology and $Min_F(A)$ is $Min(A)$ endowed with the flat topology;

$\bullet$ If $L$ is a bounded distributive lattice then, then $Min_{Id,Z}(L)$ is $Min_{Id}(L)$ endowed with the Stone topology and $Min_{Id,F}(L)$ is $Min_{Id}(L)$ endowed with the flat topology.

\begin{lema}
 Assume that $A$ is an algebra of $\mathcal{V}$ such that $K(A)$ is closed under the commutator operation. Then the following properties hold:
\usecounter{nr}
\begin{list}{(\arabic{nr})}{\usecounter{nr}}
\item $Min_Z(A)$ and $Min_{Id,Z}(L(A))$ are homeomorphic;
\item $Min_F(A)$ and $Min_{Id,F}(L(A))$ are homeomorphic.
\end{list}
\end{lema}

\begin{proof}
By Proposition 3.4, the bijective map $u|_{Min(A)}:Min(A)\rightarrow Min(L(A))$ is a Zariski homeomorphism. According to Lemma 6.4, this map is also a flat homeomorphism.
\end{proof}

\begin{corolar}
$Min_Z(A)$ is a zero - dimensional Hausdorff space and $Min_F(A)$ is a compact $T1$ space.
\end{corolar}

\begin{proof}
We know from \cite{Speed} that $Min_{Id,Z}(L(A))$ is a zero - dimensional Hausdorff space and $Min_{Id,F}(L(A))$ is a compact $T1$ space.

\end{proof}

\begin{propozitie}\cite{GM3}
 If $\phi\in Spec(A)$ then the following are equivalent:
\usecounter{nr}
\begin{list}{(\arabic{nr})}{\usecounter{nr}}
\item $\phi$ is a minimal prime congruence of $A$;
\item For all $\alpha\in K(A)$, $\alpha\subseteq \phi$ if and only if $\alpha\rightarrow \rho(\Delta_A)\not\subseteq\phi$.
\end{list}
\end{propozitie}

\begin{corolar}\cite{GM3}
 If $A$ is a semiprime algebra and $\phi\in Spec(A)$ then the following are equivalent:
\usecounter{nr}
\begin{list}{(\arabic{nr})}{\usecounter{nr}}
\item $\phi$ is a minimal prime congruence of $A$;
\item For all $\alpha\in K(A)$, $\alpha\subseteq \phi$ if and only if $\alpha^{\perp}\not\subseteq\phi$.
\end{list}
\end{corolar}

\begin{propozitie}
\cite{GM3} If $A$ is a semiprime algebra then the following are equivalent:
\usecounter{nr}
\begin{list}{(\arabic{nr})}{\usecounter{nr}}
\item $Min_Z(A) = Min_F(A)$;
\item $Min_Z(A)$ is a compact space;
\item $Min_Z(A)$ is a Boolean space;
\item For any $\alpha\in K(A)$ there exists $\beta\in K(A)$ such that $[\alpha,\beta] = \Delta_A$ and $(\alpha\lor\beta)^{\perp} = \Delta_A$.
\end{list}
\end{propozitie}

The conormal lattices were introduced in \cite{Cornish} under the name of normal lattices (for a discussion about this terminology, see \cite{Simmons} and \cite{Johnstone}, p. 78). Recall from \cite{Al-Ezeh1} that a bounded distributive lattice $L$ is said to be conormal if for all $a,b\in L$ such that $a\land b = 0$ there exist $x,y\in L$ such that $a\land x$ = $b\land y = 0$ and $x\lor y = 1$.

\begin{teorema}
\cite{Cornish} If $L$ is a conormal lattice then the following are equivalent:
\usecounter{nr}
\begin{list}{(\arabic{nr})}{\usecounter{nr}}
\item $L$ is a conormal lattice;
\item If $P$ and $Q$ are minimal prime ideals in $L$ then $P\lor Q = L$;
\item Any prime ideal of $L$ contains a unique minimal prime ideal;
\item Any minimal prime ideal of $L$ is a $\sigma$ - ideal;
\item If $x,y\in L$ then $x\land y = 0$ implies $Ann(x)\lor Ann(y) = L$;
\item For all $x,y\in L$ we have $Ann(x\land y) = Ann(x)\lor Ann(y)$;
\item For any $x\in L$, the annihilator $Ann(x)$ is a $\sigma$ - ideal.
\end{list}
\end{teorema}

On the other hand, a commutative ring $R$ is said to be an $mp$ - ring if each prime ideal of $R$ contains a unique minimal prime ideal. Following Theorem 6.2 of \cite{Aghajani}, we present several conditions that characterize the $mp$ - rings.

\begin{teorema}
\cite{Aghajani} For any commutative ring $R$ the following are equivalent
\usecounter{nr}
\begin{list}{(\arabic{nr})}{\usecounter{nr}}
\item $R$ is an $mp$ - ring;
\item If $P$ and $Q$ are minimal prime ideals in $R$ then $P + Q = R$;
\item $R/n(R)$ is an $mp$ - ring, where $n(R)$ is the nil - radical of $R$;
\item $Spec(R)$ is a normal space with respect to the flat topology;
\item The inclusion $Min(R)\subseteq Spec(R)$ has a continuous retraction with respect to the flat topology;
\item If $P$ is a minimal prime ideal of $R$ then $V(P)$ is a flat closed subset of $Spec(R)$;
\item For all $x,y\in R$, if $xy = 0$ implies $Ann(x^n) + Ann(y^n) = R$, for any integer $n\geq 1$;
\item Any minimal prime ideal of $R$ is the radical of a unique pure ideal of $R$.
\end{list}
\end{teorema}

If we compare the definitions and the previous characterizations of conormal lattices and $mp$ - rings then we observe that these structures have very similar descriptions. Then these particular situations led to the problem of finding a common generalization in the setting of universal algebra.

Let us fix an algebra $A\in \mathcal{V}$. Then $A$ is said to be an $mp$ - algebra if any prime congruence of $A$ contains a unique minimal prime congruence. We observe that $mp$ - rings are particular cases of the $mp$ - algebras, but the $mp$ - objects in the variety of bounded distributive lattices does not coincide with the conormal lattices.

In the rest of this section we will assume that $K(A)$ is closed under commutator operation. We fix a commutative ring $R$ such that the reticulations $L(A)$ and $L(R)$ of $A$ and $R$ are identical.

In what follows we shall show how the characterizations of conormal lattices (from Theorem 9.6) and characterizations of $mp$ -rings (from Theorem 9.7) can be extended to $mp$ - algebras. Firstly we shall prove some lemmas regarding the way in which the reticulation preserves some of these characterizations.

\begin{lema}
The following are equivalent:
\usecounter{nr}
\begin{list}{(\arabic{nr})}{\usecounter{nr}}
\item $A$ is an $mp$ - algebra;
\item The reticulation $L(A)$ is a conormal lattice.
\end{list}
\end{lema}

\begin{proof}

From Section 3 we know that the functions $u:Spec(A)\rightarrow Spec_{Id}(L(A))$ and $v:Spec_{Id}(L(A))\rightarrow Spec(A)$ are order - preserving isomorphisms (cf. Proposition 3.4). Then the equivalence of the properties (1) and (2) follows.
\end{proof}

\begin{corolar}
The following are equivalent:
\usecounter{nr}
\begin{list}{(\arabic{nr})}{\usecounter{nr}}
\item $A$ is an $mp$ - algebra;
\item $R$ is an $mp$ - ring.
\end{list}
\end{corolar}

\begin{proof}
According to the previous lemma, $A$ is an $mp$ - algebra iff the reticulation $L(A) = L(R)$ is a conormal lattice iff $R$ is an $mp$ - ring.
\end{proof}

The previous lemma and its corollary will be used for transferring some properties that characterize the conormal lattices (in Theorem 9.6) or the $mp$ - rings (in Theorem 9.7) in properties that characterize the $mp$ - algebras.

\begin{lema}
The following are equivalent:
\usecounter{nr}
\begin{list}{(\arabic{nr})}{\usecounter{nr}}
\item For all distinct $\phi,\psi\in Min(A)$ we have $\phi\lor\psi = \nabla_A$;
\item For all distinct $I,J\in Min_{Id}(L(A))$ we have $I\lor J = L(A)$.
\end{list}
\end{lema}

\begin{proof}
$(1)\Rightarrow(2)$ Let $I,J$ be two distinct minimal prime ideals of $L(A)$, so $I = \phi^{\ast}$ and $J = \psi^{\ast}$, for some distinct $\phi,\psi\in Min(A)$. Thus $\phi\lor\psi = \nabla_A$, hence, by using Lemma 3.2(2), we get $I\lor J = \phi^{\ast}\lor \psi^{\ast} = (\phi\lor\psi)^{\ast} = (\nabla_A)^{\ast} = L(A)$.

$(2)\Rightarrow(1)$ Assume that $\phi,\psi$ are two distinct minimal prime congruences of $A$, so $\phi^{\ast},\psi^{\ast}$ are distinct minimal prime ideals of $L(A)$. By Lemma 3.2(2) we have $(\phi\lor\psi)^{\ast} = \phi^{\ast}\lor \psi^{\ast} = L(A)$. Thus $\rho(\phi\lor\psi) = ((\phi\lor\psi)^{\ast})_{\ast} = \nabla_A$ (cf. Lemma 3.3(2)), so by using Lemma 2.7(3), we obtain $\phi\lor\psi = \nabla_A$.

\end{proof}

Recall that $R$ is a fixed commutative ring such that the reticulations $L(A)$ and $L(R)$ of $A$ and $R$ are identical.

\begin{corolar}
The following are equivalent:
\usecounter{nr}
\begin{list}{(\arabic{nr})}{\usecounter{nr}}
\item For all distinct $\phi,\psi\in Min(A)$ we have $\phi\lor\psi = \nabla_A$;
\item For all distinct $P,Q\in Min(R)$ we have $P + Q = R$.
\end{list}
\end{corolar}

\begin{proof}
By taking into account that $L(A) = L(R)$, the equivalence of the properties $(1)$ and $(2)$ follows by applying Lemma 9.10 twice.
\end{proof}

The following two lemmas can be obtained in a similar manner.

\begin{lema}
The following are equivalent:
\usecounter{nr}
\begin{list}{(\arabic{nr})}{\usecounter{nr}}
\item The inclusion $Min(A)\subseteq Spec(A)$ has a continuous retraction with respect to the flat topology;
\item The inclusion $Min(R)\subseteq Spec(R)$ has a continuous retraction with respect to the flat topology.
\end{list}
\end{lema}

\begin{lema}
The following are equivalent:
\usecounter{nr}
\begin{list}{(\arabic{nr})}{\usecounter{nr}}
\item If $\phi\in Min(A)$ then $V(\phi) = \{\psi\in Spec(A)|\phi\subseteq\psi\}$ is a flat closed subset of $Spec(A)$;
\item If $P\in Min(R)$ then $V(P) = \{Q\in Spec(R)|P\subseteq Q\}$ is a flat closed subset of $Spec(A)$.
\end{list}
\end{lema}

\begin{teorema}
Consider an algebra $A\in \mathcal{V}$ such that $K(A)$ is closed under commutator operation. Then the following equivalences hold:
\usecounter{nr}
\begin{list}{(\arabic{nr})}{\usecounter{nr}}
\item $A$ is an $mp$ - algebra;
\item For all distinct minimal prime congruences $\phi,\psi$ of $A$ we have the equality $\phi\lor \psi = \nabla_A$;
\item The quotient algebra $A/\rho(\Delta_A)$ is an $mp$ - algebra;
\item The inclusion $Min(A)\subseteq Spec(A)$ has a continuous retraction with respect to the flat topology;
\item $Spec(A)$ is a normal space with respect to the flat topology;
\item If $\phi\in Min(A)$ then $V(\phi)$ is a flat closed subset of $Spec(A)$.
\end{list}
\end{teorema}

\begin{proof}

$(1)\Leftrightarrow(2)$ By using Corollary 9.8, Theorem 6.2 of \cite{Aghajani} and Corollary 9.11, it follows that the following conditions are equivalent:

${\bullet}$ $A$ is an $mp$ - algebra;

${\bullet}$ $R$ is an $mp$ -ring;

${\bullet}$ For all distinct $P,Q\in Min(R)$ we have $P + Q = R$;

${\bullet}$ For all distinct $\phi,\psi\in Min(A)$ we have $\phi\lor\psi = \nabla_A$.

$(1)\Leftrightarrow(3)$ By Proposition 7.7 of \cite{GM2}, the lattices $L(A)$ and $L(A/\rho(\Delta_A))$ are isomorphic. By applying Lemma 9.8, the following equivalences hold: $A$ is an $mp$ - algebra iff $L(A)$ is a conormal lattice iff $L(A/\rho(\Delta_A))$ is a conormal lattice iff $A/\rho(\Delta_A)$ is an $mp$ - algebra.

$(1)\Leftrightarrow(4)$ By using Corollary 9.9, Theorem 6.2 of \cite{Aghajani} and Lemma 9.12, it follows that the following conditions are equivalent:

${\bullet}$ $A$ is an $mp$ - algebra;

${\bullet}$ $R$ is an $mp$ -ring;

${\bullet}$ The inclusion $Min(R)\subseteq Spec(R)$ has a continuous retraction with respect to the flat topology;

${\bullet}$ The inclusion $Min(A)\subseteq Spec(A)$ has a continuous retraction with respect to the flat topology.

$(1)\Leftrightarrow(5)$ We know that $Spec_F(A)$ and $Spec_F(R)$ are homeomorphic (by Lemma 6.11), then by using Corollary 9.8 and  Theorem 6.2 of \cite{Aghajani} we get the following equivalences: $A$ is an $mp$ - algebra iff $R$ is an $mp$ -ring iff $Spec_F(R)$ is normal iff $Spec_F(A)$ is normal.

$(1)\Leftrightarrow(6)$ This equivalence follows similarly, by using Corollary 9.8, Theorem 6.2 of \cite{Aghajani} and Lemma 9.13.

\end{proof}

\begin{propozitie} Let $A\in \mathcal{V}$ be an algebra such that $K(A)$ is closed under commutator operation. Then the following equivalences hold:
\usecounter{nr}
\begin{list}{(\arabic{nr})}{\usecounter{nr}}
\item $A$ is an $mp$ - algebra;
\item Any minimal prime congruence of $A$ is $w$ - pure.
\end{list}
\end{propozitie}

\begin{proof}
$(1)\Rightarrow (2)$
Let $\phi$ be a minimal prime congruence of $A$, hence $\phi^{\ast}$ is a minimal prime ideal of the lattice $L(A)$. According to Lemma 9.8, $L(A)$ is a conormal lattice, hence, by Theorem 9.6(4), it follows that  $\phi^{\ast}$ is a $\sigma$ - ideal of $L(A)$. By Lemmas 3.3 and Theorem 5.3(2), $\phi =  (\phi^{\ast})_{\ast}$ is a $w$ - pure congruence of $A$.

$(2)\Rightarrow (1)$
In order to prove that $L(A)$ is a conormal lattice, assume that $P$ is a minimal prime ideal of $L(A)$, hence $P = \phi^{\ast}$, for some minimal prime congruence $\phi$ of $A$. By the hypothesis (2), $\phi^{\ast}$ is a $w$ - pure congruence of $A$. In accordance with Theorem 5.3(1), $P = \phi^{\ast}$ is a $\sigma$ - ideal of $L(A)$. Applying Theorem 9.6, it follows that $L(A)$ is a conormal lattice, so $A$ is an $mp$ - algebra (cf. Lemma 9.8).

\end{proof}

\begin{teorema}
For any algebra $A\in\mathcal{V}$ such that $K(A)$ is closed under the commutator operation the following equivalences hold:

\usecounter{nr}
\begin{list}{(\arabic{nr})}{\usecounter{nr}}
\item $A$ is an $mp$ - algebra;
\item For all congruences $\alpha, \beta\in K(A)$, $[\alpha,\beta]\subseteq \rho(\Delta_A)$ implies the equality $(\alpha\rightarrow \rho(\Delta_A))\lor (\beta\rightarrow \rho(\Delta_A)) = \nabla_A$;
\item For all congruences $\alpha, \beta\in K(A)$, we have the following equality:

$(\alpha\rightarrow \rho(\Delta_A))\lor (\beta\rightarrow \rho(\Delta_A)) = [\alpha,\beta]\rightarrow \rho(\Delta_A)$ ;
\item For any $\alpha\in K(A)$, $\alpha\rightarrow \rho(\Delta_A)$ is a $w$ - pure congruence of $A$.
\end{list}
\end{teorema}

\begin{proof}

$(1)\Rightarrow (2)$
Assume by absurdum that there exist $\alpha, \beta\in K(A)$, such that $[\alpha,\beta]\subseteq \rho(\Delta_A)$ and $(\alpha\rightarrow \rho(\Delta_A))\lor (\beta\rightarrow \rho(\Delta_A))\neq \nabla_A$, so there exists a maximal congruence $\psi$ such that $(\alpha\rightarrow \rho(\Delta_A))\lor (\beta\rightarrow \rho(\Delta_A))\subseteq \psi$. Let us consider a minimal prime congruence $\phi$ such that $\phi\subseteq \psi$. By Proposition 9.15, the congruence $\phi$ is $w$ - pure. From $[\alpha,\beta]\subseteq \rho(\Delta_A)\subseteq \phi$ it follows that $\alpha\subseteq \phi$ or $\beta\subseteq \phi$ (because $\phi$ is a prime congruence). Assume that $\alpha\subseteq \phi$, so $\phi\lor (\alpha\rightarrow \rho(\Delta_A)) = \nabla_A$ (because $\phi$ is $w$ - pure), contradicting $\phi\subseteq \psi$ and $(\alpha\rightarrow \rho(\Delta_A))\subseteq \psi$. Similarly, the inclusion $\beta\subseteq \phi$ provides a contradiction. Thus the property (1) implies the property (2).

$(2)\Rightarrow (1)$
Let $\phi, \psi$ be two distinct minimal prime congruences of $A$ hence there exists $\beta\in K(A)$ such that $\beta\subseteq \phi$ and $\beta\not\subseteq \psi$. By Proposition 9.3, from $\beta\subseteq \phi$ we get $\beta\rightarrow \rho(\Delta_A)\not\subseteq \phi$, so there exists $\alpha\in K(A)$ such that $\alpha\subseteq \beta\rightarrow \rho(\Delta_A)$ and $\alpha\not\subseteq \phi$. Thus $[\alpha,\beta]\subseteq \rho(\Delta_A)$, hence, by the hypothesis (2) one gets $(\alpha\rightarrow \rho(\Delta_A))\lor (\beta\rightarrow \rho(\Delta_A)) = \nabla_A$. This last equality implies that there exist $\gamma, \delta\in K(A)$ such that $\gamma\subseteq \alpha\rightarrow \rho(\Delta_A)$, $\delta\subseteq \beta\rightarrow \rho(\Delta_A)$ and $\gamma\lor \delta = \nabla_A$. From $[\alpha,\gamma]\subseteq \rho(\Delta_A)\subseteq \phi$ and $\alpha\not\subseteq \phi$ one obtains $\gamma\subseteq \phi$. In a similar way one can show that $\delta\subseteq \psi$, so $\phi\lor \psi = \nabla_A$. By applying Theorem 9.14 it follows that $A$ is an $mp$ - algebra.

$(2)\Rightarrow (3)$
Firstly, we observe that to establish the inclusion $(\alpha\rightarrow \rho(\Delta_A))\lor (\beta\rightarrow \rho(\Delta_A))\subseteq [\alpha,\beta]\rightarrow \rho(\Delta_A)$ is straightforward: from $[\alpha,\beta]\subseteq \alpha$ and $[\alpha,\beta]\subseteq \beta$ we get $\alpha\rightarrow \rho(\Delta_A)\subseteq [\alpha,\beta]\rightarrow \rho(\Delta_A)$ and $\beta\rightarrow \rho(\Delta_A)\subseteq [\alpha,\beta]\rightarrow \rho(\Delta_A)$, hence $(\alpha\rightarrow \rho(\Delta_A))\lor (\beta\rightarrow \rho(\Delta_A))\subseteq [\alpha,\beta]\rightarrow \rho(\Delta_A)$.

In order to prove the converse inclusion $[\alpha,\beta]\rightarrow \rho(\Delta_A)\subseteq (\alpha\rightarrow \rho(\Delta_A))\lor (\beta\rightarrow \rho(\Delta_A))$ let us consider a congruence $\gamma\in K(A)$ such that $\gamma\subseteq [\alpha,\beta]\rightarrow \rho(\Delta_A)$. Thus $[\gamma,[\alpha,\beta]]\subseteq \rho(\Delta_A)$, hence, by applying Lemma 3.1(6) one obtains $\lambda_A([\gamma,[\alpha,\beta]]) = 0$. By observing that $\lambda_A([[\gamma,\alpha],\beta])$ = $\lambda_A(\alpha)\land \lambda_A(\beta)\land \lambda_A(\gamma)$ = $\lambda_A([\gamma,[\alpha,\beta]]) = 0$ (cf. Lemma 3.1(2)) and by applying again Lemma 3.1(6) one gets $[[\gamma,\alpha],\beta]\subseteq \rho(\Delta_A)$. In accordance with the hypothesis (2), it follows that $([\gamma,\alpha]\rightarrow \rho(\Delta_A))\lor (\beta\rightarrow \rho(\Delta_A)) = \nabla_A$, so there exist $\delta,\epsilon \in K(A)$ such that $\delta \subseteq [\gamma,\alpha]\rightarrow \rho(\Delta_A)$, $\epsilon \subseteq \beta\rightarrow \rho(\Delta_A)$ and $\delta \lor\epsilon = \nabla_A$. From  $\delta \subseteq [\gamma,\alpha]\rightarrow \rho(\Delta_A)$ we get $\lambda_A([\gamma,\delta],\alpha])$ = $\lambda_A([\delta,[\gamma,\alpha]]) = 0$. Applying again Lemma 3.1(6), we obtain $[[\gamma,\delta],\alpha]\subseteq \rho(\Delta_A)$, hence $[\gamma,\delta]\subseteq \alpha\rightarrow\rho(\Delta_A)$. Therefore we have $\gamma = [\gamma,\delta\lor \epsilon] = [\gamma,\delta]\lor [\gamma,\epsilon]\subseteq [\gamma,\delta]\lor \epsilon \subseteq (\alpha\rightarrow \rho(\Delta_A))\lor (\beta\rightarrow \rho(\Delta_A))$. Hence $[\alpha,\beta]\rightarrow \rho(\Delta_A)\subseteq (\alpha\rightarrow \rho(\Delta_A))\lor (\beta\rightarrow \rho(\Delta_A))$, so the desired equality follows.

$(3)\Rightarrow (2)$
Assume that $\alpha, \beta\in K(A)$ and $[\alpha,\beta]\subseteq \rho(\Delta_A)$ therefore $[\alpha,\beta]\rightarrow \rho(\Delta_A) = \nabla_A$, so we get the equality $(\alpha\rightarrow \rho(\Delta_A))\lor (\beta\rightarrow \rho(\Delta_A)) = \nabla_A$.

$(4)\Rightarrow (2)$
Assume that $\alpha,\beta$ are two compact congruences of $A$ such that $[\alpha,\beta]\subseteq \rho(\Delta_A)$, so $\beta\subseteq\alpha\rightarrow \rho(\Delta_A)$. By taking into account that $\alpha\rightarrow \rho(\Delta_A)$ is $w$ - pure we obtain the equality $(\alpha\rightarrow \rho(\Delta_A))\lor (\beta\rightarrow \rho(\Delta_A)) = \nabla_A$.

$(2)\Rightarrow (4)$
Assume that $\beta$ is a compact congruence of $A$ and $\beta\subseteq \alpha\rightarrow \rho(\Delta_A)$, hence $[\alpha,\beta] \subseteq \rho(\Delta_A)$. Applying hypothesis (2) it follows that  $(\alpha\rightarrow \rho(\Delta_A))\lor (\beta\rightarrow \rho(\Delta_A)) = \nabla_A$,so $\alpha\rightarrow \rho(\Delta_A)$ is a $w$ - pure congruence of $A$.

\end{proof}

\begin{lema}
 For any algebra $A$ the following hold:
\usecounter{nr}
\begin{list}{(\arabic{nr})}{\usecounter{nr}}
\item If $\phi\in Spec(A)$ then $\rho(O(\phi))\subseteq\phi$;
\item If $A$ satisfies $(\star)$ and $\phi\in Spec(A)$ then $\phi\in Min(A)$ if and only if $\rho(O(\phi)) = \phi$.
\end{list}
\end{lema}

\begin{proof} (1) If $\phi\in Spec(A)$ then $O(\phi)\subseteq \phi$ implies $\rho(O(\phi))\subseteq\phi$.

(2) Assume that $A$ satisfies $(\star)$ and $\phi\in Spec(A)$. Firstly we suppose that $\phi\in Min(A)$ and we have to prove that $\phi\in Spec(A)$ implies $\rho(O(\phi)) = \phi$. By (1) it suffices to show that $\phi\subseteq \rho(O(\phi))$. Let $\alpha$ be a compact congruence of $A$ such that $\alpha\subseteq \phi$. In accordance with Proposition 9.3 we have $\alpha\rightarrow \rho(\Delta_A)\not\subseteq \phi$, so there exists $\beta\in K(A)$ such that $\beta\subseteq \alpha\rightarrow \rho(\Delta_A)$ and $\beta\not\subseteq \phi$. Then $[\alpha,\beta]\subseteq \rho(\Delta_A)$, so by applying Lemma 2.10 we get $[[\alpha,\alpha]^m,[\beta,\beta]^m] = \Delta_A$, for some integer $m\geq 0$. Since $\phi$ is a prime congruence, $\beta\not\subseteq \phi$ implies $[\beta,\beta]^m\not\subseteq \phi$. From $[\beta,\beta]^m\subseteq ([\alpha,\alpha ]^m)^{\perp}$ and $[\beta,\beta]^m\not\subseteq \phi$ we get $([\alpha,\alpha ]^m)^{\perp}\not\subseteq \phi$. By applying Lemma 5.6(1) we obtain $[\alpha,\alpha]^m\subseteq O(\phi)$, therefore $\alpha\subseteq \rho(O(\phi))$ (cf. Proposition 2.8). We conclude that $\phi\subseteq \rho(O(\phi))$.

Conversely, let us assume that $\rho(O(\phi)) = \phi$ and consider a minimal prime congruence $\psi$ such that $\psi\subseteq\phi$. Let $\alpha$ be a compact congruence of $A$ such that $\alpha\subseteq O(\phi)$, hence $\alpha^{\perp}\not\subseteq \phi$ (by Lemma 5.6(1)). Thus $\alpha^{\perp}\not\subseteq \psi$, hence $\alpha\subseteq \psi$ (because $\psi$ is a prime congruence). It results that $O(\phi)\subseteq \psi$, hence $\phi\subseteq\rho(O(\phi))\subseteq O(\psi)\subseteq \psi$. Since $\psi\in Min(A)$ we get $\phi = \psi$, therefore $\phi$ is a minimal prime congruence.

\end{proof}

\begin{teorema}
Let $A$ be an algebra that fulfills $(\star)$. Then the following are equivalent:
\usecounter{nr}
\begin{list}{(\arabic{nr})}{\usecounter{nr}}
\item $A$ is an $mp$ - algebra;
\item For all distinct $\phi,\psi\in Min(A)$ we have $O(\phi)\lor O(\psi) = \nabla_A$.
\end{list}
\end{teorema}

\begin{proof}

$(1)\Rightarrow (2)$ Assume that $\phi$ and $\psi$ are two distinct minimal prime congruences of $A$, hence $\phi\lor\psi = \nabla_A$ (cf. Theorem 9.14). By applying Lemma 9.17(2) we have $\rho(O(\phi)) = \phi$ and  $\rho(O(\psi)) = \psi$, hence $\rho(O(\phi))\lor\rho(O(\psi)) = \nabla_A$. By Lemma 2.7(6) we obtain $O(\phi)\lor O(\psi) = \nabla_A$.

$(2)\Rightarrow (1)$ If $\phi$ and $\psi$ are two distinct minimal prime congruences of $A$ then $\nabla_A = O(\phi)\lor O(\psi)\subseteq \phi\lor \psi$. In accordance with Theorem 9.14 it follows that $A$ is an $mp$ - algebra.

\end{proof}

Recall from \cite{Al-Ezeh2} that a commutative ring $R$ is a $PF$ - ring if the annihilator of each element of $R$ is a pure ideal. We generalize this notion for algebras in the variety $\mathcal{V}$. An algebra $A\in\mathcal{V}$ is said to be a $PF$ - algebra if for all $\alpha\in PCon(A)$, $\alpha^{\perp}$ is a pure congruence.

\begin{lema}
For any algebra $A$ the following are equivalent:
\usecounter{nr}
\begin{list}{(\arabic{nr})}{\usecounter{nr}}
\item $A$ is a $PF$ - algebra;
\item For all $\alpha\in K(A)$, $\alpha^{\perp}$ is a pure congruence.
\end{list}
\end{lema}

\begin{proof}
$(1)\Rightarrow (2)$ Assume that $A$ is a $PF$ - algebra. Let $\alpha\in K(A)$, so $\alpha = \bigvee_{i=1}^n\alpha_i$ with $\alpha_i\in PCon(A)$, for $i = 1,\cdots,n$. Thus $\alpha^{\perp} = \bigcap_{i=1}^n\alpha_i^{\perp}$ is a pure congruence (cf. Lemma 4.7.(4)), any finite intersection of pure congruences is a pure congruence).

$(2)\Rightarrow (1)$ Obviously.
\end{proof}

\begin{lema}
Any $PF$ - algebra is semiprime.
\end{lema}

\begin{proof}
Assume that $A$ is a $PF$ - algebra. Let $\alpha\in K(A)$ and $n$ an integer such that $n \geq 0$ and $[\alpha,\alpha]^n = \Delta_A$. We shall prove that $\alpha = \Delta_A$. From $[[\alpha,\alpha]^{n-1},[\alpha,\alpha]^{n-1}] = [\alpha,\alpha]^n = \Delta_A$ it follows that $[\alpha,\alpha]^{n-1}\subseteq ([\alpha,\alpha]^{n-1})^{\perp}$. We remark that $[\alpha,\alpha]^{n-1}\in K(A)$, so by the previous lemma, $([\alpha,\alpha]^{n-1})^{\perp}$ is a pure congruence, hence $\nabla_A
 = ([\alpha,\alpha]^{n-1})^{\perp}\lor ([\alpha,\alpha]^{n-1})^{\perp} = ([\alpha,\alpha]^{n-1})^{\perp}$. Thus $[\alpha,\alpha]^{n-1}\subseteq ([\alpha,\alpha]^{n-1})^{{\perp}{\perp}} = \Delta_A$, hence $[\alpha,\alpha]^{n-1} = \Delta_A$. By applying many times this argument one obtains $\alpha = \Delta_A$.
\end{proof}

\begin{propozitie}
If $A$ is a $PF$ - algebra then the reticulation $L(A)$ is a conormal lattice.
\end{propozitie}

\begin{proof}
By Lemma 9.20, the $PF$ - algebra $A$ is semiprime. Let $x\in L(A)$, so there exists $\alpha\in K(A)$ such that $x = \lambda_A(\alpha)$. Taking into account Proposition 3.9 and Lemma 3.3(7) we have $Ann(x) = Ann(\lambda_A(\alpha)) = Ann((\lambda_A(\alpha)]) = Ann(\alpha^{\ast}) = (\alpha^{\perp})^{\ast}$. Since $A$ is a $PF$ - algebra, $\alpha^{\perp}$ is a pure congruence. By Proposition 4.15, $Ann(x) = (\alpha^{\perp})^{\ast}$ is a $\sigma$ - ideal of the lattice $L(A)$. According to Theorem 9.6, it follows that $L(A)$ is conormal.

\end{proof}

\begin{teorema}
 For any algebra $A$ the following are equivalent:
\usecounter{nr}
\begin{list}{(\arabic{nr})}{\usecounter{nr}}
\item $A$ is a $PF$ - algebra;
\item $A$ is a semiprime $mp$ - algebra.
\end{list}
\end{teorema}

\begin{proof}

$(1)\Rightarrow (2)$ By Lemmas 9.8, 9.20 and Proposition 9.21.

$(2)\Rightarrow (1)$ Assume that $A$ is a semiprime $mp$ - algebra and $\alpha\in K(A)$. We shall prove that $\alpha^{\perp}$ is a pure congruence of $A$. Let $\beta\in K(A)$ such that $\beta\subseteq\alpha^{\perp}$, so $\lambda_A(\alpha)\land\lambda_A(\beta) = \lambda_A([\alpha,\beta]) = \lambda_A(\Delta_A) = 0$, i.e. $\lambda_A(\beta)\in Ann(\alpha^{\ast}) = Ann(\lambda_A(\alpha))$. By Lemma 9.8, the reticulation $L(A)$ is a conormal lattice, hence $Ann(\lambda_A(\alpha))$ is a $\sigma$ - ideal of $L(A)$ (by Theorem 9.6). Consequently, the following equalities hold: $Ann(\alpha^{\ast})\lor Ann(\beta^{\ast}) = Ann(\lambda_A(\alpha))\lor Ann(\lambda_A(\beta)) = L(A)$.

In accordance with Lemma 3.3(4), Proposition 3.9 and Lemma 3.2(3) we have:

$\rho(\rho(\alpha^{\perp})\lor\rho(\beta^{\perp}))$ = $\rho(((\alpha^{\perp})^{\ast})_{\ast}\lor ((\beta^{\perp})^{\ast})_{\ast})$ = $\rho((Ann(\alpha^{\ast}))_{\ast}\lor(Ann(\beta^{\ast}))_{\ast})$ = $(Ann(\alpha^{\ast})\lor Ann(\beta^{\ast}))_{\ast}$ = $(L(A))_{\ast} = \nabla_A$.

By applying Lemma 2.7,(3) and (6) we get $\alpha^{\perp}\lor\beta^{\perp} = \nabla_A$, hence $\alpha^{\perp}$ is a pure congruence of $A$. Thus $A$ is a $PF$ - algebra (cf. Lemma 9.19).

\end{proof}

\begin{corolar}
If $A$ is a semiprime algebra then $A$ is a $PF$ - algebra if and only if any minimal prime congruence of $A$ is pure.
\end{corolar}

\begin{proof}
In a semiprime algebra, the pure and $w$ - pure congruences coincide, so by applying Proposition 9.15 and Theorem 9.22 we get the desired equivalence.
\end{proof}

\begin{corolar} For any algebra $A$ the following are equivalent:
\usecounter{nr}
\begin{list}{(\arabic{nr})}{\usecounter{nr}}
\item $A$ is a $PF$ - algebra;
\item For all congruences $\alpha, \beta\in K(A)$, $[\alpha,\beta] = \Delta_A$ implies the equality $\alpha^{\perp}\lor \beta^{\perp} = \nabla_A$;
\item For all congruences $\alpha, \beta\in K(A)$, the
 equality $\alpha^{\perp}\lor \beta^{\perp} = [\alpha,\beta]^{\perp}$ holds;
\item For any $\alpha\in K(A)$, $\alpha^{\perp}$ is a pure congruence of $A$.
\end{list}
\end{corolar}

\begin{proof}

$(1)\Rightarrow (3)$ By Theorem 9.16 and Lemma 9.20.

$(3)\Rightarrow (2)$ Obvious.

$(1)\Leftrightarrow (4)$ By Lemma 9.19.

$(2)\Leftrightarrow (4)$ Obvious.

\end{proof}

\begin{propozitie}
Let $A$ be a $PF$ - algebra. If $\phi\in Spec(A)$ then $O(\phi)$ is a minimal prime congruence of $A$.
\end{propozitie}

\begin{proof}
According to Proposition 5.10, it suffices to prove that $O(\phi)$ is a prime congruence of $A$. Consider $\alpha,\beta\in K(A)$ such that $[\alpha,\beta]\subseteq O(\phi)$, hence $[\alpha,\beta]^{\perp}\not\subseteq\phi$ (by Lemma 5.6(1)). By Corollary 9.24(3) we have $\alpha^{\perp}\lor \beta^{\perp} = [\alpha,\beta]^{\perp}$, so $\alpha^{\perp}\lor \beta^{\perp}\not\subseteq\phi$, hence $\alpha^{\perp}\not\subseteq\phi$ or $\beta^{\perp}\not\subseteq\phi$. By using again Lemma 5.6(1) we obtain $\alpha\subseteq O(\phi)$ or $\beta\subseteq O(\phi)$, hence $O(\phi)$ is a prime congruence.

\end{proof}

\begin{lema}
For any algebra $A$ we have $\bigcap\{O(\phi)|\phi\in Max(A)\} = \Delta_A$.

\end{lema}

\begin{proof} Let $\alpha\in K(A)$ be such that $\alpha\subseteq\bigcap\{O(\phi)|\phi\in Max(A)\}$. If $\alpha^{\perp}\neq\nabla_A$ then there exists $\phi\in Max(A)$ such that $\alpha^{\perp}\subseteq\phi$. By Lemma 5.6(1) we have $\alpha\not\subseteq O(\phi)$, contradicting $\alpha\subseteq\bigcap\{O(\phi)|\phi\in Max(A)\}$. It follows that $\alpha^{\perp} = \nabla_A$, hence $\alpha\subseteq\alpha^{{\perp}{\perp}} = (\nabla_A)^{\perp} = \Delta_A$. Thus $\alpha = \Delta_A$, so we conclude that $\bigcap\{O(\phi)|\phi\in Max(A)\} = \Delta_A$.

\end{proof}

\begin{propozitie} For any algebra $A$ the following are equivalent:
\usecounter{nr}
\begin{list}{(\arabic{nr})}{\usecounter{nr}}
\item $A$ is a $PF$ - algebra;
\item For any $\phi\in Spec(A)$, $O(\phi)$ is a prime congruence;
\item For any $\phi\in Max(A)$, $O(\phi)$ is a prime congruence;
\item For all $\phi,\psi\in Spec(A)$, $\phi\subseteq\psi$ implies $O(\phi) = O(\psi)$.
\end{list}
\end{propozitie}

\begin{proof}

$(1)\Rightarrow (2)$ By Proposition 9.25.

$(2)\Rightarrow (3)$ Obviously.

$(3)\Rightarrow (1)$ Assume that $\psi\in Spec(A)$ and take $\phi\in Min(A),\chi\in Max(A)$ such that $\phi\subseteq\psi\subseteq\chi$. Consider $\alpha\in K(A)$ such that $\alpha\subseteq O(\chi)$, hence $\alpha^{\perp}\not\subseteq\chi$ (by Lemma 5.6(1)). Thus there exists $\beta\in K(A)$ such that $[\alpha,\beta] = \Delta_A$ and $\beta\not\subseteq\chi$, hence $\alpha\subseteq\chi$. We have proven that $O(\chi)\subseteq\phi$, so $O(\chi) = \phi$ (because $O(\chi)\in Spec(A)$ and $\phi\in Min(A)$). Thus any prime congruence of $A$ includes a unique minimal prime congruence, so $A$ is an $mp$ - algebra.

Assume that $\alpha$ is a compact congruence of $A$ such that $\alpha\subseteq\rho(\Delta_A)$, hence there exists an integer $n\geq1$ such that $[\alpha,\alpha]^n = \Delta_A$. Let $\phi$ be an arbitrary maximal congruence of $A$, so $[\alpha,\alpha]^n = \Delta_A$ implies that $\alpha\subseteq O(\phi)$ (because $O(\phi)$ is a prime congruence). Thus $\alpha\subseteq \bigcap\{O(\phi)|\phi\in Max(A)\}$, so, by Lemma 9.26 we get $\alpha = \Delta_A$. We have proven that $A$ is semiprime, therefore it is a $PF$ - algebra (by Theorem 9.22).

$(1)\Rightarrow (4)$ Assume that $A$ is a $PF$ - algebra and $\phi,\psi\in Spec(A)$ such that $\phi\subseteq\psi$, therefore $O(\phi)\subseteq O(\psi)$. By Proposition 9.25,  $O(\phi)$ and $O(\psi)$ are minimal prime congruences, hence $O(\phi) = O(\psi)$.

$(4)\Rightarrow (2)$ Assume that $\phi\in Spec(A)$, so there exists $\psi\in Min(A)$ such that $\psi\subseteq\phi$. By applying Proposition 9.25 and hypothesis (4) we get $\psi = O(\psi) = O(\phi)$. Thus $\phi$ is a prime congruence.

\end{proof}

\section{Congruence purified algebras}

\hspace{0.5cm} The purified rings were introduced in Section 8 of \cite{Aghajani}: a commutative ring $R$ is said to be a purified ring if for all distinct minimal prime ideals $P,Q$ of $R$ there exists an idempotent element $e$ of $R$ such that $e\in P$ and $\neg e\in Q$. We shall define a notion of congruence purified algebra as an abstraction of the purified rings. The main objective of this section is to obtain a characterization theorem for congruence purified algebras (= Theorem 10.15). Of course this general result extends the characterization theorem of purified rings (= Theorem 8.5 of \cite{Aghajani}).

Let us fix an algebra $A$ in a semidegenerate congruence - modular variety $\mathcal{V}$. Then $A$ is said to be a congruence purified algebra if for all distinct minimal prime congruence $\phi,\psi$ of $A$ there exists $\alpha\in B(Con(A))$ such that $\alpha\subseteq\phi$ and $\neg\alpha\subseteq\psi$.

\begin{remarca}
It is well - known that the Boolean algebra of idempotents of a commutative ring $R$ is isomorphic to $B(Con(R))$ (see e.g. Lemma 1 of \cite{Banaschewski}). Then it follows that $R$ is a purified ring if and only if it is congruence purified.
\end{remarca}

\begin{lema}
Any congruence purified algebra is an $mp$ - algebra.
\end{lema}

\begin{proof}
Let $\phi,\psi$ two distinct minimal prime congruences of $A$, so there  $\alpha\in B(Con(A))$ such that $\alpha\subseteq\phi$ and $\neg\alpha\subseteq\psi$. In accordance with Theorem 9.14,(2) it follows that $\phi\lor \psi = \nabla_A$, hence $A$ is an $mp$ - algebra.
\end{proof}

In order to obtain a proof of the characterization theorem for congruence purified algebras we need some preliminary results.

\begin{lema}
If $A$ fulfills the condition $(\star)$, then the congruence $\rho(\Delta_A)$ has $CBLP$.
\end{lema}

\begin{proof}
Denote $\theta = \rho(\Delta_A)$ and consider the (admissible) canonical morphism $p_{\theta}:A\rightarrow A/\theta$. Recall from Section 8 that the map $p_{\theta}^{\bullet}:Con(A)\rightarrow Con(A/\theta)$ is defined by $p_{\theta}^{\bullet}(\nu) = (\nu\lor\theta)/\theta$, for all $\nu\in Con(A)$. We shall prove that the Boolean morphism $p_{\theta}^{\bullet}|_{B(Con(A))}:B(Con(A))\rightarrow B(Con(A/\theta))$ is surjective.

Let us consider a congruence $\chi$ of $A$ such that $\theta\subseteq\chi$. Assume that $\chi/\theta\in B(Con(A/\theta))$, so there exists $\nu\in Con(A)$ such that $\theta\subseteq\nu$, $\chi/\theta\lor\nu/\theta = \nabla_{A/\theta}$ and $\chi/\theta\bigcap\nu/\theta = \Delta_{A/\theta}$. Thus $(\chi\lor\nu)/\theta = \nabla_{A/\theta}$ and $(\chi\bigcap\nu)/\theta = \Delta_{A/\theta}$, hence we obtain $\chi\lor\nu = \nabla_A$ and $\chi\bigcap\nu\subseteq\rho( \Delta_A)$. Since $\Delta_A$ is compact, there exist two compact congruences $\alpha,\beta$ such that $\alpha\subseteq\chi$, $\beta\subseteq\nu$, $\alpha\lor\beta = \nabla_A$, therefore $[\alpha,\beta]\subseteq\alpha\bigcap\beta\subseteq\rho( \Delta_A)$.

By Lemma 2.10 there exists an integer $m\geq1$ such that $[[\alpha,\beta]^m,[\alpha,\beta]^m] = \Lambda_A$. According to Lemma 2.4(3), from  $\alpha\lor\beta = \nabla_A$ we get  $[\alpha,\alpha]^m\lor [\beta,\beta]^m] = \nabla_A$. An application of Lemma 2.11(1) gives $[\alpha,\alpha]^m,[\beta,\beta]^m\in B(Con(A))$ and $[\alpha,\alpha]^m = \neg [\beta,\beta]^m$.

We observe that $p_{\theta}^{\bullet}([\alpha,\alpha]^m)\subseteq \chi/\theta$ and $p_{\theta}^{\bullet}([\beta,\beta]^m)\subseteq \nu/\theta$, therefore

$\chi/\theta = \neg(\nu/\theta)\subseteq \neg p_{\theta}^{\bullet}([\beta,\beta]^m)$ = $p_{\theta}^{\bullet}(\neg [\beta,\beta]^m) = p_{\theta}^{\bullet}([\alpha,\alpha]^m)\subseteq \chi/\theta$.

It follows that $p_{\theta}^{\bullet}([\alpha,\alpha]^m) = \chi/\theta$, so $p_{\theta}^{\bullet}|_{B(Con(A))}:B(Con(A))\rightarrow B(Con(A/\theta))$ is surjective.

\end{proof}

\begin{teorema}
Let $A$ be an algebra that satisfies $(\star)$. Then $A$ is congruence purified if and only if $A/\rho(\Delta_A)$ is congruence purified.
\end{teorema}

\begin{proof}
Let us denote $\theta$ = $\rho(\Delta_A)$ and let us consider the admissible morphism $p_{\theta}:A\rightarrow A/\theta$. Recall that the map $p_{\theta}^{\bullet}:Con(A)\rightarrow Con(A/\theta)$ is defined by $p_{\theta}^{\bullet}(\nu) = (\nu\lor\theta)/\theta$, for all $\nu\in Con(A)$.

We observe that $p_{\theta}^{\bullet}|_{Spec(A)}:Spec(A)\rightarrow Spec(A/\theta)$ is an order - isomorphism and $p_{\theta}^{\bullet}|_{Min(A)}:Min(A)\rightarrow Min(A/\theta)$ is a bijection.

Assume that $A$ is a congruence purified algebra and consider two distinct minimal prime congruences $\chi$ and $\epsilon$ of $A/\theta$, so $\chi = p_{\theta}^{\bullet}(\phi)$ and $\epsilon = p_{\theta}^{\bullet}(\psi)$ for some distinct $\phi,\psi\in Min(A)$.

Then there exists $\alpha\in B(Con(A))$ such that $\alpha\subseteq\phi$ and $\neg \alpha\subseteq\psi$. Therefore $p_{\theta}^{\bullet}(\alpha)\in B(Con(A/\theta))$, $\neg p_{\theta}^{\bullet}(\alpha) = p_{\theta}^{\bullet}(\neg\alpha)\in B(Con(A/\theta))$ and $p_{\theta}^{\bullet}(\alpha)\subseteq p_{\theta}^{\bullet}(\phi) = \chi$, $\neg p_{\theta}^{\bullet}(\alpha)\subseteq p_{\theta}^{\bullet}(\psi) = \epsilon$. We have proven that $A/\theta$ is congruence purified.

Conversely, assume that $A/\theta$ is congruence purified and $\phi,\psi$ are two distinct minimal prime congruences of $A$, hence $\theta\subseteq\phi$, $\theta\subseteq\psi$ and $\phi/\theta,\psi/\theta$ are two distinct minimal prime congruences of $A/\theta$. Then there exists $\gamma\in B(Con(A/\theta))$ such that $\gamma\subseteq\phi/\theta$ and $\neg\gamma\subseteq\psi/\theta$. According to Lemma 10.3, $A/\theta$ has $CBLP$, therefore there exists $\alpha\in B(Con(A))$ such that $p_{\theta}^{\bullet}(\alpha) = \gamma$. Thus $(\alpha\lor\theta)/\theta = p_{\theta}^{\bullet}(\alpha) = \gamma\subseteq\phi/\theta$ and $(\neg\alpha\lor\theta)/\theta = p_{\theta}^{\bullet}(\neg\alpha) = \neg p_{\theta}^{\bullet}(\alpha) = \neg\gamma\subseteq\psi/\theta$, hence $\alpha\lor\theta\subseteq\phi$ and $\neg\alpha\lor\theta\subseteq\psi$. It follows that $\alpha\subseteq\phi$ and $\neg\alpha\subseteq\psi$, so $A$ is a congruence purified algebra.

\end{proof}

\begin{lema} For any semiprime algebra $A$ the following are equivalent:
\usecounter{nr}
\begin{list}{(\arabic{nr})}{\usecounter{nr}}
\item $A$ is a congruence purified  algebra;
\item For any distinct $P,Q\in Min(A)$ there exists $f\in B(L(A))$ such that $f\in P$ and $\neg f\in Q$.
\end{list}
\end{lema}

\begin{proof}

$(1)\Rightarrow (2)$ Assume that $P, Q$ are two distinct minimal prime ideals of $L(A)$, hence there exist the distinct minimal prime congruences $\phi,\psi$ of $A$ such that $P = \phi^{\ast}$ and $Q = \psi^{\ast}$. Then there exists $\alpha\in B(Con(A))$ such that $\alpha\subseteq\phi$ and $\neg \alpha\subseteq\psi$, therefore $\lambda_A(\alpha)\in B(L(A)))$ such that $\lambda_A(\alpha)\in \phi^{\ast} = P$ and $\neg\lambda_A(\alpha) = \lambda_A(\neg\alpha)\in\psi^{\ast} = Q$.

$(2)\Rightarrow (1)$ Let $\phi$ and $\psi$ be two distinct minimal prime congruences of $A$, so $\phi^{\ast}$ and $\psi^{\ast}$ are distinct minimal prime ideals of $L(A)$. By hypothesis (2), there exists $f\in B(L(A))$ such that $f\in \phi^{\ast}$ and $\neg f\in \psi^{\ast}$. Since $A$ is semiprime, $\lambda_A|_{B(Con(A))}:B(Con(A))\rightarrow B(L(A))$ is an isomorphism of Boolean algebras, hence there exists $\alpha\in B(Con(A))$ such that $f = \lambda_A(\alpha)$ and $\neg f = \lambda_A(\neg\alpha)$. Thus $\lambda_A(\alpha)\in \phi^{\ast}$ and $\lambda_A(\neg\alpha) \in \psi^{\ast}$, hence there exist $\beta,\gamma \in K(A)$ such that $\beta\subseteq\phi$, $\gamma\subseteq\psi$, $\lambda_A(\alpha) = \lambda_A(\beta)$ and $\lambda_A(\neg\alpha) = \lambda_A(\gamma)$.

It follows that $\rho(\alpha) = \rho(\beta)$ and $\rho(\neg\alpha) = \rho(\gamma)$, so there exists an integer $n\geq 1$ such that $\alpha = [\alpha,\alpha]^n\subseteq\beta\subseteq\phi$ and $\neg\alpha = [\neg\alpha,\neg\alpha]^n\subseteq\gamma\subseteq\psi$. Therefore $A$ is a congruence purified algebra.

\end{proof}

Recall that for the fixed algebra $A$ there exists a semiprime commutative ring such that their reticulations $L(A)$ and $L(R)$ are isomorphic.

\begin{propozitie} Let $A$ a be semiprime algebra and $R$ a semiprime commutative ring such that the reticulations $L(A)$ and $L(R)$ are isomorphic. Then the following are equivalent:
\usecounter{nr}
\begin{list}{(\arabic{nr})}{\usecounter{nr}}
\item $A$ is a congruence purified algebra;
\item $R$ is a purified ring.
\end{list}
\end{propozitie}

\begin{proof}
By  applying Lemma 10.5 twice.
\end{proof}

\begin{lema}
If $A$ is semiprime then the topological spaces $Sp(A)$ and $Sp(R)$ are homeomorphic.
\end{lema}

\begin{proof} We identify the isomorphhic lattices $L(A)$, $L(R)$ and their associated spaces $Sp_{Id}(L(A))$, $Sp_{Id}(L(R))$.
By Proposition 6.24, the both topological spaces $Sp(A)$ and $Sp(R)$ are homeomorphic to $Sp_{Id}(L(A))$.
\end{proof}

Let $L$ be a bounded distributive lattice. Recall that for each $I\in Id(L)$ we denote $V_{Id}(P) = \{Q\in Spec(A)|I\subseteq Q\}$ and for each $x\in L$, $V_{Id}(x)$ will denote the set $V_{Id}((x])$, where $(x]$ is the principal ideal of $L$ generated by $\{x\}$.

\begin{lema} For any semiprime algebra $A$ the following are equivalent:
\usecounter{nr}
\begin{list}{(\arabic{nr})}{\usecounter{nr}}
\item The family $(V_A(\alpha)\bigcap Min(A))_{\alpha\in B(Con(A))}$ is a basis of open sets for the space $Min_F(A)$;
\item The family $(V_{Id}(e)\bigcap Min(A))_{e\in B(L(A))}$ is a basis of open sets for the space $Min_F(L(A))$.
\end{list}
\end{lema}

\begin{proof} Since $A$ is semiprime then $\lambda_A|_{B(Con(A))}:B(Con(A))\rightarrow B(L(A))$ is a Boolean isomorphism (cf. Proposition 6.18,(iii) of \cite{GM2}). By Lemma 9.1, the map $u:Min_F(A)\rightarrow Min_{Id,F}(L(A))$, defined by $u(\phi) = \phi^{\ast}$, for all $\phi\in Min(A)$, is a homeomorphism. We know that for each $\alpha\in K(A)$ we have $u(V_A(\alpha)) = V_{Id}(\alpha^{\ast}) = V(\lambda_A(\alpha)))$, therefore $u(V_A(\alpha)\bigcap Min(A)) = V_{Id}(\alpha^{\ast})\bigcap Min(L(A)) = V(\lambda_A(\alpha))\bigcap Min(L(A))$. From these considerations it is easy to infer the equivalence of the properties (1) and (2).

\end{proof}

\begin{corolar} Let $A$ be a semiprime algebra and $R$ a semiprime commutative ring such that the reticulations $L(A)$ and $L(R)$ are isomorphic. The following are equivalent:
\usecounter{nr}
\begin{list}{(\arabic{nr})}{\usecounter{nr}}
\item The family $(V_A(\alpha)\bigcap Min(A))_{\alpha\in B(Con(A))}$ is a basis of open sets for $Min_F(A)$;
\item The family $(V_R(e)\bigcap Min(R))_{e\in B(R)}$ is a basis of open sets for $Min_F(R)$.
\end{list}
\end{corolar}

\begin{proof}
By applying Lemma 10.8 twice.
\end{proof}

\begin{lema} For any semiprime algebra $A$ the following are equivalent:
\usecounter{nr}
\begin{list}{(\arabic{nr})}{\usecounter{nr}}
\item Any $\phi\in Min(A)$ is a regular congruence of $A$;
\item Any $P\in Min_{Id}(L(A))$ is a regular ideal of $L(A)$.
\end{list}
\end{lema}

\begin{proof}

 $(1)\Rightarrow (2)$. If $P\in Min_{Id}(L(A))$ then $P = \phi^{\ast}$ for some $\phi\in Min(A)$, hence $\phi$ is a regular congruence. By Lemma 6.23, $P = \phi^{\ast}$ is a regular ideal of $L(A)$.

 $(2)\Rightarrow (3)$. This implication follows similarly, by using again Lemma 6.23.

\end{proof}

\begin{corolar} Let $A$ be a semiprime algebra and $R$ a semiprime commutative ring such that the reticulations $L(A)$ and $L(R)$ are isomorphic. The following are equivalent:
\usecounter{nr}
\begin{list}{(\arabic{nr})}{\usecounter{nr}}
\item Any $\phi\in Min(A)$ is a regular congruence of $A$;
\item Any $P\in Min(R)$ is a regular ideal of $R$.
\end{list}
\end{corolar}

\begin{lema} For any semiprime algebra $A$ the following are equivalent:
\usecounter{nr}
\begin{list}{(\arabic{nr})}{\usecounter{nr}}
\item $Min(A) = Sp(A)$ ;
\item $Min_{Id}(L(A)) = Sp_{Id}(L(A))$.
\end{list}
\end{lema}

\begin{proof} The equivalence of (1) and (2) follows from the following facts:

${\bullet}$  The map $(\cdot)^{\ast}: Spec(A)\rightarrow Spec_{Id}(L(A))$ is an order - isomorphism (cf. Proposition 3.4);

${\bullet}$ The map $ (\cdot)^{\ast}|_{Min(A)}: Min(A)\rightarrow Min_{Id}(L(A))$ is bijective (cf. Lemma 9.1);

${\bullet}$ The map $ (\cdot)^{\ast}|_{Sp(A)}: Sp(A)\rightarrow Sp_{Id}(L(A))$ is bijective (cf. Proposition 6.24).

\end{proof}

\begin{corolar} Let $A$ be a semiprime algebra and $R$ a semiprime commutative ring such that the reticulations $L(A)$ and $L(R)$ are isomorphic. Then $Min(A) = Sp(A)$ if and only if $Min(R) = Sp(R)$.

\end{corolar}

\begin{lema} For any semiprime algebra $A$ the following are equivalent:
\usecounter{nr}
\begin{list}{(\arabic{nr})}{\usecounter{nr}}
\item  Any pure congruence of $A$ is regular;
\item  Any pure ideal of $R$ is regular.
\end{list}
\end{lema}

\begin{proof}
By Corollary 4.19 and Lemma 6.23.
\end{proof}

Recall from \cite{Johnstone}, p. 69 that a topological space $X$ is said to be totally disconnected if the only connected subsets of $X$ are single points.

The following result generalizes a part of Theorem 8.5 of \cite{Aghajani}.

\begin{teorema}
Consider a semiprime algebra $A\in \mathcal{V}$ such that $K(A)$ is closed under commutator operation. Then the following equivalences hold:
\usecounter{nr}
\begin{list}{(\arabic{nr})}{\usecounter{nr}}
\item $A$ is a congruence purified algebra;
\item $A$ is an $mp$ - algebra and $Min_F(A)$ is totally disconnected;
\item The family $(V(\alpha)\bigcap Min(A))_{\alpha\in B(Con(A))}$ is a basis of open sets for $Min_F(A)$;
\item Any minimal prime congruence of $A$ is regular;
\item $Min(A) = Sp(A)$;
\item $A$ is an $mp$ - algebra and any pure congruence of $A$ is regular.
\end{list}
\end{teorema}

\begin{proof} Let us fix a semiprime commutative ring $R$ such that the reticulations $L(A)$ and $L(R)$ of $A$ and $R$ are isomorphic (see the Hochster theorem from \cite{Hochster}). We shall identify the isomorphic lattices $L(A)$ and $L(R)$. By using Lemma 9.1, the topological spaces $Min_F(A)$ and $Min_F(R)$ are homemorphic.

$(1)\Leftrightarrow (2)$ By using Proposition 10.6, Theorem 8.5 of \cite{Aghajani}, Corollary 9.9 and Lemma 9.1, the following properties are equivalent:

${\bullet}$ $A$ is a congruence purified algebra;

${\bullet}$ $R$ is a purified ring;

${\bullet}$ $R$ is an $mp$ - ring and $Min_F(R)$ is totally disconnected;

${\bullet}$ $A$ is an $mp$ - algebra and $Min_F(A)$ is totally disconnected.

$(1)\Leftrightarrow (3)$ This equivalence follows by applying Proposition 10.6, Theorem 8.5 of \cite{Aghajani} and Corollary 10.9 :$A$ is a congruence purified algebra iff $R$ is a purified ring iff the family $(V_R(e)\bigcap Min(R))_{e\in B(R)}$ is a basis of open sets for $Min_F(R)$ iff the family $(V_A(\alpha)\bigcap Min(A))_{\alpha\in B(Con(A))}$ is a basis of open sets for $Min_F(A)$.

$(1)\Leftrightarrow (4)$ By using Proposition 10.6, Theorem 8.5 of \cite{Aghajani} and Corollary 10.11 , the following properties are equivalent: $A$ is a congruence purified algebra iff $R$ is a purified ring iff any $P\in Min(R)$ is a regular ideal of $R$ iff any $\phi\in Min(A)$ is a regular congruence of $A$.

$(1)\Leftrightarrow (5)$ We apply Theorem 8.5 of \cite{Aghajani} and Corollary 10.13.

$(1)\Leftrightarrow (6)$ We apply Theorem 8.5 of \cite{Aghajani} and Lemma 10
.14.

\end{proof}

We observe that the proof of previous theorem is obtaining by transporting the properties that characterize the purified rings (in Theorem 8.5 of \cite{Aghajani}) from rings to algebras of $\mathcal{V}$.

\begin{corolar}
Any hyperarchimedean algebra $A$ is congruence purified.
\end{corolar}

\begin{proof}
According to Theorem 10.4, one can assume that $A$ is a semiprime hyperarchimedean algebra. By Theorem 7.7(4), it follows that $Max(A) = Min(A) = Spec(A)$ and $Spec_F(A)$ is a Boolean space. Thus $A$ is an $mp$ - algebra (cf. Theorem 9.14(5)) and $Min_F(A)$ is a Boolean space. By taking into account Theorem 10.15(2) it follows that $A$ is congruence purified.
\end{proof}

\begin{corolar}
Let $A$ be a $PF$ algebra. If $Min_Z(A)$ is a compact space then $A$ is congruence purified.
\end{corolar}

\begin{proof}
According to Theorem 9.22, $A$ is a semiprime $mp$ - algebra. Let us fix a semiprime commutative ring $R$ such that the reticulations $L(A)$ and $L(R)$ are isomorphic. Thus $R$ is a semiprime $mp$ - ring (by Corollary 9.9). If $Min_Z(A)$ is a compact space then $Min_Z(R)$ is also compact (by Lemma 9.1(1), $Min_Z(A)$ and $Min_Z(R)$ are homeomorphic). By Theorem 8.5 of \cite{Aghajani}, $R$ is a purified ring, hence $A$ is congruence purified (by Proposition 6.10).
\end{proof}

\begin{propozitie} Let us consider the semiprime algebras $A_1,\cdots,A_n$ in $\mathcal{V}$ such that $K(A_1),\cdots,K(A_n)$ are closed under the commutator operation and their product $A = \prod_{i=1}^nA_i$. Then $A$ is a congruence purified algebra if and only if  $A_i, i= 1\cdots n$ are congruence purified algebras.
\end{propozitie}

\begin{proof} For each $i=1,\cdots,n$ consider  a semiprime commutative ring $R_i$ such that the reticulations $L(A_i)$ and $L(R_i)$ are isomorphic (by the Hochster theorem). By Proposition 5.11 of \cite{GM2}, the reticulations $L(A)$ and $L(R)$ of $A$ and $R$ are isomorphic.

By using that the variety $\mathcal{V}$ has Horn - Fraser property (cf. Proposition 2.1) it is not difficult to prove that the product $A = \prod_{i=1}^nA_i$ is semprime.

According to Proposition 10.6 and  Corollary 8.7 of \cite{Aghajani}, the following assertions are equivalent:

${\bullet}$ $A$ is a congruence purified algebra;

${\bullet}$ $R$ is a purified ring;

${\bullet}$  $R_i$, for $i= 1,\cdots,n$ are congruence purified rings;

${\bullet}$  $A_i$, for $i= 1,\cdots,n$ are congruence purified algebras.

\end{proof}

\section{$PP$ - algebras}

\hspace{0.5cm} Recall that a commutative ring $R$ is said to be a $PP$ - ring (or a Baer ring) if the annihilator of any element of $R$ is generated by an idempotent element. This concept can be extended in a standard way to an algebra $A$ of a semidegenerate congruence - modular $\mathcal{V}$: $A$ is called $PP$ - algebra if for any $\alpha\in PCon(A)$ we have $\alpha^{\perp}\in B(Con(A))$.

\begin{lema} For any algebra $A$ the following are equivalent:
\usecounter{nr}
\begin{list}{(\arabic{nr})}{\usecounter{nr}}
\item  $A$ is a $PP$ - algebra;
\item  For any $\alpha\in K(A)$ we have $\alpha^{\perp}\in B(Con(A))$.
\end{list}
\end{lema}

\begin{proof} Assume that $A$ is a $PP$ - algebra and $\alpha\in K(A)$. Then there exist $\beta_1,\cdots,\beta_n\in PCon(A))$ such that $\alpha = \bigvee_{i=1}^n \beta_i$, hence
$\beta^{\perp}_i\in B(Con(A))$ for all $i=1,\cdots,n$. Therefore we get $\alpha^{\perp} = \bigcap_{i=1}^n \beta^{\perp}_i\in B(Con(A))$. The converse implication is obvious.
\end{proof}

\begin{lema}
Any $PP$ - algebra is semiprime.
\end{lema}

\begin{proof} Firstly we observe that for any $\theta\in Con(A)$ such that $\theta\subseteq\theta^{\perp}$ and $\theta^{\perp}\in B(Con(A))$ we have $\theta = \theta\bigcap \theta^{\perp} = [\theta,\theta^{\perp}] = \Delta_A$.

In order to prove that $A$ is semiprime let us assume that $\alpha$ is a compact congruence such that $[\alpha,\alpha]^n = \Delta_A$, for some integer $n\geq 1$. Thus $[\alpha,\alpha]^{n-1}\leq ([\alpha,\alpha]^{n-1})^{\perp}$ and $([\alpha,\alpha]^{n-1})^{\perp}\in B(Con(A))$ hence, by using the previous observation, one gets $[\alpha,\alpha]^{n-1} = \Delta_A$. By using many times this argument we obtain $\alpha = \Delta_A$. Thus $A$ is semiprime.
\end{proof}

Recall that a bounded distributive lattice $L$ is said to be a Stone lattice if for each $x\in L$ there exist $e\in B(L)$ such that $Ann(x) = (x]$. A Simmons' theorem from \cite{Simmons} says that the reticulation functor transforms the $PP$ - rings in Stone lattices. The following result extends the Simmons theorem from rings to algebras of $\mathcal{V}$.

\begin{teorema}  Let $A$ be an algebra in $\mathcal{V}$ such that $K(A)$ is closed under commutator operation. Let us consider the following properties:
\usecounter{nr}
\begin{list}{(\arabic{nr})}{\usecounter{nr}}
\item  $A$ is a $PP$ - algebra;
\item  The reticulation $L(A)$ is a Stone lattice.

Then the implication $(1)\Rightarrow (2)$ holds; if $A$ is semiprime then the converse implication $(2)\Rightarrow (1)$ is valid.
\end{list}
\end{teorema}

\begin{proof}

$(1)\Rightarrow (2)$ Assume that $A$ is a $PP$ - algebra. Thus $A$ is a semiprime algebra (cf. Lemma 11.2). hence, by using Proposition 6.18(iii) of \cite{GM2}, it follows that the map $\lambda_A|_{B(Con(A))}:B(Con(A))\rightarrow B(L(A))$ is a Boolean isomorphism. In order to prove that the reticulation $L(A)$ is a Stone lattice let us consider $x\in L(A)$, so there exists $\alpha\in K(A)$ such that $x = \lambda_A(\alpha)$. Since $A$ is a $PP$ - algebra we have $\alpha^{\perp}\in B(Con(A)))$ (by Lemma 11.1), hence $\lambda_A(\alpha^{\perp})\in B(L(A))$.

We shall prove that $Ann(x) = Ann(\lambda_A(\alpha)) = (\lambda_A(\alpha^{\perp})]$. By Lemma 3.3(7) and Proposition 3.9 we have $Ann(\lambda_A(\alpha)) = Ann(\alpha^{\ast}) = (\alpha^{\perp})^{\ast}$. Let us consider an element $y\in Ann(\lambda_A(\alpha)) = (\alpha^{\perp})^{\ast}$, so there exists a compact congruence $\beta$ such that $\beta\subseteq \alpha^{\perp}$ and $y = \lambda_A(\beta)$. Then $y = \lambda_A(\beta)\leq \lambda_A(\alpha^{\perp})$, so $y\in (\lambda_A(\alpha^{\perp})]$. We have established the inclusion $Ann(\lambda_A(\alpha))\subseteq (\lambda_A(\alpha^{\perp})]$.

By observing that $\lambda_A(\alpha)\land \lambda_A(\alpha^{\perp}) = \lambda_A([\alpha,\alpha^{\perp})]) = \lambda_A(\Delta_A) = 0$, hence $(\lambda_A(\alpha^{\perp})]\subseteq Ann(\lambda_A(\alpha))$.

$(2)\Rightarrow (1)$ (assuming that $A$ is semiprime). We suppose that $L(A)$ is a Stone lattice and will prove that $A$ is a $PP$ - algebra. Let $\alpha$ be a compact congruence of $A$, so there exists $f\in B(L(A))$ such that $Ann(\lambda_A(\alpha)) = (f]$. Since $\lambda_A|_{B(Con(A))}:B(Con(A))\rightarrow B(L(A))$ is a Boolean isomorphism it follows that $f = \lambda_A(\epsilon)$, for some $\epsilon\in B(Con(A))$. According to Proposition 3.9 we have $(\alpha^{\perp})^{\ast} = Ann(\lambda_A(\alpha)) = (\lambda_A(\epsilon)] = \epsilon^{\ast}$. By using Lemma 3.3(2) we get $\rho(\alpha^{\perp}) = ((\alpha^{\perp})^{\ast})_{\ast} = (\epsilon^{\ast})_{\ast} = \rho(\epsilon)$. In accordance with Lemma 6.21 we have $\rho(\epsilon) = \epsilon$, so $\rho(\alpha^{\perp}) = \epsilon\in B(Con(A))$. If we apply Lemma 6.22 then we obtain $\alpha^{\perp} = \epsilon\in B(Con(A))$, hence $\alpha^{\perp}\in B(Con(A))$. Therefore $A$ is a $PP$ -algebra.

\end{proof}

Let us fix $f:A\rightarrow B$ a morphism of $\mathcal{V}$. Then $f$ is said to be a $PP$ - morphism if for all $\alpha,\beta\in K(A)$, $\alpha^{\perp_A} = \beta^{\perp_A}$ implies $(f^{\bullet}(\alpha))^{\perp_B} = (f^{\bullet}(\beta))^{\perp_B}$. It is easy to see that any flat morphism of $\mathcal{V}$ is a $PP$ - morphism.

 Throughout this section we shall assume that $K(A)$ and $K(B)$ are closed to commutator operation.

\begin{lema}
If $f:A\rightarrow B$ and $g:B\rightarrow C$ are two $PP$ - morphisms in $\mathcal{V}$ then $g\circ f:A\rightarrow C$ is a $PP$ - morphism.
\end{lema}

\begin{lema}
Assume that the algebras $A$ and $B$ are semiprime and $f:A\rightarrow B$ is an admissible morphism. Then for any $\alpha\in K(A)$ we have $Ann(L(f)(\lambda_A(\alpha))) =  Ann(L(f)((\lambda_A(\alpha)]))$.

\end{lema}

\begin{proof} It suffices to show that for all $\beta\in K(A)$ the following equivalence holds:

${\bullet}$ $\lambda_A(\beta)\in Ann(L(f)(\lambda_A(\alpha)))$ if and only if $\lambda_A(\beta)\in Ann(L(f)(\lambda_A(\alpha)]))$.

Firstly assume that $\lambda_A(\beta)\in Ann(L(f)(\lambda_A(\alpha)))$. Let $y\in L(f)((\lambda_A(\alpha)])$, so there exists $\gamma\in K(A)$ such that $\lambda_A(\gamma)\leq \lambda_A(\alpha)$ and $y = L(f)(\lambda_A(\gamma))$. Thus $\lambda_A(\beta)\land y = \lambda_A(\beta)\land L(f)(\lambda_A(\gamma))\leq \lambda_A(\beta)\land L(f)(\lambda_A(\alpha)) = 0$, hence $\lambda_A(\beta)\in Ann(L(f)((\lambda_A(\alpha)]))$.

Conversely, let us assume that $\lambda_A(\beta)\in Ann(L(f)((\lambda_A(\alpha)]))$. We remark that $L(f)(\lambda_A(\alpha))\in L(f)((\lambda_A(\alpha)])$, therefore we get $\lambda_A(\beta)\land L(f)(\lambda_A(\alpha)) = 0$, hence $\lambda_A(\beta)\in Ann(L(f)(\lambda_A(\alpha)))$.

\end{proof}

\begin{lema}
 For any $\alpha\in K(A)$ we have $\alpha^{\perp} = (\rho(\alpha))^{\perp}$.

\end{lema}

\begin{proof}
The inclusion $\alpha\subseteq\rho(\alpha)$ implies that $(\rho(\alpha))^{\perp}\subseteq \alpha^{\perp}$. In order to prove that $\alpha^{\perp}\subseteq (\rho(\alpha))^{\perp}$ it suffices to verify that $\beta\in K(A)$ and $[\beta,\alpha] = \Delta_A$ implies $[\beta,\rho(\alpha)] = \Delta_A$. Assume that $\beta\in K(A)$ and $[\beta,\alpha] = \Delta_A$.

Since $\rho(\alpha) = \bigvee\{\alpha\in K(A)|[\gamma,\gamma]^n\subseteq \alpha$, for some  $n\geq 1\}$ and the commutator operation is distributive, the following equality holds:

$[\beta,\rho(\alpha)] = \bigvee\{[\beta,\gamma]|\gamma \in K(A)$ and $[\alpha,\gamma]^n\subseteq \alpha$, for some  $n\geq 1\}$.

Let $\gamma\in K(A)$ such that $[\gamma,\gamma]^n\subseteq \alpha$, for some  $n\geq 1$. From $[\beta,\alpha] = \Delta_A$ and $[\gamma,\gamma]^n\subseteq \alpha$ it follows that $[\beta,[\gamma,\gamma]^n] = \Delta_A$, so $\lambda_A([\beta,\gamma]) = \lambda_A([\beta,[\gamma,\gamma]^n]) = 0$. Since $A$ is semiprime we get $[\beta,\gamma] = \Delta_A$. It follows that $[\beta,\rho(\alpha)] = \Delta_A$.

\end{proof}

Let $g:L\rightarrow L'$ be a morphism in the category of bounded distributive lattices. Then $f$ is said to be a Stone morphism if for all $a,b\in L$, $Ann(a) = Ann(b)$ implies $Ann(f(a)) = Ann(f(b))$.

\begin{teorema}  Assume that $A$ and $B$ are two semiprime algebras of $\mathcal{V}$ and  $f:A\rightarrow B$ is an admissible morphism. The following properties are equivalent:
\usecounter{nr}
\begin{list}{(\arabic{nr})}{\usecounter{nr}}
\item  $f:A\rightarrow B$  is a $PP$ - morphism;
\item  $L(f):L(A)\rightarrow L(B)$ is a Stone morphism.
\end{list}
\end{teorema}

\begin{proof}

$(1)\Rightarrow (2)$ Assume that $Ann(\lambda_A(\alpha)) = Ann(\lambda_A(\beta))$, with $\alpha,\beta\in K(A)$, hence  $Ann((\lambda_A(\alpha)]) = Ann((\lambda_A(\beta)])$. According to  Proposition 3.10 it follows that $((\lambda_A(\alpha)]_{\ast})^{\perp_A} = (Ann((\lambda_A(\alpha)]))_{\ast} = (Ann((\lambda_A(\beta)]))_{\ast} = ((\lambda_A(\beta)]_{\ast})^{\perp_A}$. But $(\lambda_A(\alpha)]_{\ast} = (\alpha^{\ast})_{\ast} = \rho(\alpha)$ and $(\lambda_A(\beta)]_{\ast} = (\beta^{\ast})_{\ast} = \rho(\beta)$ (cf. Lemma 3.3,(3) and (7)), hence, by using Lemma 11.6, we get $\alpha^{\perp_A} = (\rho(\alpha))^{\perp_A} = (\rho(\beta))^{\perp_A} = \beta^{\perp_A} $.

Since $f:A\rightarrow B$  is a $PP$ - morphism we get $(f^{\bullet}(\alpha))^{\perp_B} = (f^{\bullet}(\beta))^{\perp_B}$, therefore by using Proposition 3.9 the following equalities hold:

$Ann((f^{\bullet}(\alpha))^{\ast}) = ((f^{\bullet}(\alpha))^{\perp_B})^{\ast} = ((f^{\bullet}(\beta))^{\perp_B})^{\ast} = Ann((f^{\bullet}(\beta))^{\ast})$.

It follows that $Ann((\lambda_B(f^{\bullet}(\alpha))]) = Ann((\lambda_B(f^{\bullet}(\beta))])$ (by Lemma 3.3(7)), so we obtain $Ann(\lambda_B(f^{\bullet}(\alpha))) = Ann(\lambda_B(f^{\bullet}(\beta)))$. By taking into account the commutative diagram of Proposition 3.11 it follows that $Ann(L(f)(\lambda_A(\alpha))) = Ann(L(f)(\lambda_A(\beta)))$. Thus $L(f)$ is a Stone morphism.

$(2)\Rightarrow (1)$ Assume that $\alpha,\beta\in K(A)$ such that $\alpha^{\perp_A} = \beta^{\perp_A}$. The following implications hold

$\alpha^{\perp_A} = \beta^{\perp_A}$ $\Rightarrow$ $Ann(\alpha^{\ast}) = (\alpha^{\perp_A})^{\ast} = (\beta^{\perp_A})^{\ast} = Ann(\beta^{\ast})$ (by Prop. 3.9)

\hspace{1.8cm} $\Rightarrow$ $Ann((\lambda_A(\alpha)]) = Ann((\lambda_A(\beta)])$ (by Lemma 3.3(7))

\hspace{1.8cm} $\Rightarrow$ $Ann(\lambda_A(\alpha)) = Ann(\lambda_A(\beta))$

\hspace{1.8cm} $\Rightarrow$ $Ann(L(f)(\lambda_A(\alpha))) = Ann(L(f)(\lambda_A(\beta)))$ ( by hypothesis)

\hspace{1.8cm} $\Rightarrow$ $Ann(\lambda_B(f^{\bullet}(\alpha))) = Ann(\lambda_B(f^{\bullet}(\beta)))$ (by Prop. 3.11)

\hspace{1.8cm} $\Rightarrow$ $Ann((\lambda_B(f^{\bullet}(\alpha))]) = Ann((\lambda_B(f^{\bullet}(\beta))])$

\hspace{1.8cm} $\Rightarrow$ $Ann((f^{\bullet}(\alpha))^{\ast}) = Ann((f^{\bullet}(\beta))^{\ast})$ (by Lemma 3.3(7))

\hspace{1.8cm} $\Rightarrow$ $(f^{\bullet}(\alpha))^{\perp_B})^{\ast} = (f^{\bullet}(\alpha))^{\perp_B})^{\ast}$ (cf. Prop. 3.9)

\hspace{1.8cm} $\Rightarrow$ $\rho((f^{\bullet}(\alpha))^{\perp_B})^{\ast}) = (f^{\bullet}(\alpha))^{\perp_B})^{\ast})_{\ast} = (f^{\bullet}(\beta))^{\perp_B})^{\ast})_{\ast} = \rho((f^{\bullet}(\beta))^{\perp_B})^{\ast})$

\hspace{1.8cm} $\Rightarrow$ $(f^{\bullet}(\alpha))^{\perp_B} = (f^{\bullet}(\beta))^{\perp_B}$ (by Lemma 11.6).

Therefore $f:A\rightarrow B$  is a $PP$ - morphism.
\end{proof}

\begin{lema} \cite{Cornish}, \cite{Cornish1} For any bounded distributive lattice $L$ the following are equivalent:
\usecounter{nr}
\begin{list}{(\arabic{nr})}{\usecounter{nr}}
\item  $L$ is a Stone lattice;
\item  $L$ is a conormal lattice and $Min_{Id,Z}(L)$ is compact;
\end{list}
\end{lema}

\begin{teorema}  For any semiprime algebra $A$ the following are equivalent:
\usecounter{nr}
\begin{list}{(\arabic{nr})}{\usecounter{nr}}
\item  $A$ is a $PP$ - algebra;
\item  The reticulation $L(A)$ is a Stone lattice;
\item  $L(A)$ is a conormal lattice and $Min_{Id,Z}(L)$ is compact;
\item  $A$ is a $mp$ - algebra and $Min_Z(A)$ is compact;
\item  $A$ is a $PF$ - algebra and $Min_Z(A)$ is compact.
\end{list}
\end{teorema}

\begin{proof}

$(1)\Leftrightarrow (2)$ By Theorem 11.3.

$(2)\Rightarrow (3)$ By Lemma 11.8.

$(3)\Rightarrow (4)$ We know that $A$ is a $mp$ - algebra iff $L(A)$ is a conormal lattice (cf. Lemma 9.8) and $Min_Z(A)$, $Min_{Id,Z}(L)$ are homeomorphic spaces (cf. Lemma 9.1(1).

$(4)\Rightarrow (5)$ By Theorem 9.22.
\end{proof}

\begin{lema} \cite{Cornish}, \cite{Cornish1} For any conormal lattice $L$ the following are equivalent:
\usecounter{nr}
\begin{list}{(\arabic{nr})}{\usecounter{nr}}
\item  $L$ is a Stone lattice;
\item  The inclusion $Min_{Id,Z}(L)\subseteq Spec_{Id,Z}(L)$ has a continuous retraction.
\end{list}
\end{lema}

\begin{teorema}  For any $PF$ - algebra $A$ the following are equivalent:
\usecounter{nr}
\begin{list}{(\arabic{nr})}{\usecounter{nr}}
\item  $A$ is a $PP$ - algebra;
\item  The reticulation $L(A)$ is a Stone lattice;
\item  The inclusion $Min_{Id,Z}(L(A))\subseteq Spec_{Id,Z}(L(A))$ has a continuous retraction;
\item  The inclusion $Min_Z(A)\subseteq Spec_Z(A)$ has a continuous retraction;
\item  $Min_A(Z)$ is a compact space.
\end{list}
\end{teorema}

\begin{proof}

$(1)\Leftrightarrow (2)$ By Theorem  11.3.

$(2)\Leftrightarrow (3)$ By Lemma 11.10.

$(3)\Leftrightarrow (4)$ We know that $Spec_Z(A)$ (resp. $Min_Z(A)$) is homeomorphic to $Spec_{Id,Z}(L(A))$ (resp. $Min_{Id,Z}(L(A))$).

$(4)\Leftrightarrow (5)$ By Theorem 11.9(5).

\end{proof}

\begin{propozitie}
Any $PP$ - algebra is a congruence purified algebra.
\end{propozitie}

\begin{proof}
Let $A$ be a $PP$ - algebra. By Theorem 11.9(3), $Min_Z(A)$ is a compact space, hence by using Proposition 9.5, it follows that $Min_Z(A) = Min_F(A)$ and $Min_Z(A)$ is a Boolean space. According to Theorem 10.15, $A$ is a congruence purified algebra.

\end{proof}

\end{document}